\documentclass[oneside,a4paper,reqno]{amsart}

\usepackage[utf8]{inputenc}

\usepackage{enumerate}  

\usepackage{mathrsfs}   

\usepackage{multirow}

\usepackage{amsmath, amsthm, amscd, amssymb, latexsym, eucal, mathtools}
\usepackage{tensor}

\usepackage[all]{xy}

\def\serieslogo@{} \def\@setcopyright{} \makeatother

\usepackage{boondox-calo}

\usepackage{upgreek}   

\usepackage{graphicx}

\usepackage{array}

\usepackage{multienum}

\usepackage[colorinlistoftodos]{todonotes}

\usepackage[bookmarksdepth=3]{hyperref}
\usepackage{color}
\usepackage{cite}
\usepackage[capitalise, noabbrev, nameinlink]{cleveref}
\crefname{subsection}{Subsection}{Subsections}

\usepackage{tikz-cd}
\usepackage{quiver}

\usepackage{slashbox}

\usepackage{makecell}

\makeatletter
\renewcommand*\env@matrix[1][c]{\hskip -\arraycolsep
	\let\@ifnextchar\new@ifnextchar
	\array{*\c@MaxMatrixCols #1}}
\makeatother

\usepackage{color}

\usepackage{adjustbox}

\usepackage{hhline}

\usepackage{float}

\usepackage{tikz}
\usetikzlibrary{calc}

\makeatletter
\def\@endtheorem{\endtrivlist}
\makeatother

\newcommand{\GG}{\mathcal{G}}

\newcommand{\lr}{\mathrm{l}}
\newcommand{\er}{\mathrm{e}}

\newcommand{\Ga}{\Gamma}
\newcommand{\La}{\Lambda}

\newcommand{\Ll}{\mathcal{l \! l}}
\newcommand{\lal}{\mathcal{l }}

\newcommand{\BQ}{ \mathcal{B}_{ Q} }
\newcommand{\BQA}{ \mathcal{B}_{\! Q}^{ A} }
\newcommand{\BQnotA}{ \mathcal{B}_Q^{\, \nott \!\! A} }
\newcommand{\CC}{ \mathcal{C} }
\newcommand{\NN}{ \mathcal{N} }

\newcommand{\BQa}{ \mathcal{B}_{\! Q}^{\, \alpha} }
\newcommand{\BQnota}{ \mathcal{B}_Q^{\, \nott \! \alpha} }

\newcommand{\BQB}{ \mathcal{B}_{\! Q}^{ B} }

\newcommand{\BQQnotB}{ \mathcal{B}_{Q'}^{\, \nott \!\! B} }

\newcommand{\epic}{\twoheadrightarrow}

\newcommand{\monicc}{\hookrightarrow}

\newcommand{\la}{\langle}
\newcommand{\ra}{\rangle}

\newcolumntype{C}[1]{>{\centering\arraybackslash}m{#1}}

\newcommand\TTT{\rule{0pt}{2.6ex}}       
\newcommand\BBB{\rule[-1.2ex]{0pt}{0pt}} 

\newcommand{\ol}[1]{\overline{#1}}

\newcommand\isomto{\stackrel{\sim}{\smash{\longrightarrow}\rule{0pt}{0.4ex}}}

\newcommand\wt[1]{\widetilde{#1}}

\DeclareMathOperator{\pd}{\mathsf{pd}}
\DeclareMathOperator{\fd}{\mathsf{fd}}
\DeclareMathOperator{\idim}{\mathsf{id}}

\DeclareMathOperator{\id}{\mathrm{id}}

\DeclareMathOperator{\Hom}{\mathrm{Hom}}

\DeclareMathOperator{\Ker}{\mathrm{Ker} \!}
\DeclareMathOperator{\Image}{\mathrm{Im} \!}

\DeclareMathOperator{\rmod}{\! - \mathrm{mod}}

\DeclareMathOperator{\rMod}{\! - \mathrm{Mod}}

\DeclareMathOperator{\tip}{\mathrm{Tip}}
\DeclareMathOperator{\ntip}{\mathrm{NonTip}}

\DeclareMathOperator{\nott}{\mathrm{not}}

\DeclareMathOperator{\fpd}{\mathsf{fin.dim}}
\DeclareMathOperator{\Fpd}{\mathsf{Fin.dim}}

\DeclareMathOperator{\rad}{\mathrm{rad}}
\DeclareMathOperator{\topp}{\mathsf{top}}

\DeclareMathOperator{\gd}{\mathsf{gl.dim}}

\DeclareMathOperator{\Tor}{\mathrm{Tor}}

\DeclareMathOperator{\supp}{\mathrm{supp}}

\DeclareMathOperator{\codomdim}{\mathrm{-}\,\mathrm{codom.}\mathrm{dim}}

\newtheorem{thm}{Theorem}[section]
\newtheorem{cor}[thm]{Corollary}
\newtheorem{lem}[thm]{Lemma}
\newtheorem{prop}[thm]{Proposition}

\newtheorem*{thmA}{Theorem A}
\newtheorem*{thmB}{Theorem B}
\newtheorem*{thmC}{Theorem C}
\newtheorem*{thmD}{Theorem D}

\newtheorem*{thmF}{Theorem F}

\theoremstyle{definition}

\newtheorem{defn}[thm]{Definition}

\newtheorem{exam}[thm]{Example}
\newtheorem{con}[thm]{Construction}

\newtheorem*{conE}{Construction E}

\theoremstyle{remark}

\newtheorem{rem}[thm]{Remark}

\crefname{exam}{Example}{Examples}

\crefname{lem}{Lemma}{Lemmas}

\crefname{cor}{Corollary}{Corollaries}

\def\a{\alpha}
\def\b{\beta}
\def\g{\gamma}
\def\d{\delta}
\def\e{\varepsilon}
\def\z{\zeta}
\def\i{\iota}
\def\k{\kappa}
\def\l{\lambda}
\def\m{\mu}

\hypersetup{hidelinks}

\begin{document}

	\title{Arrow reductions for the finitistic dimension conjecture}

	\author[Giatagantzidis]{Odysseas Giatagantzidis}
	
	\address{Department of Mathematics, Aristotle University of Thessaloniki, Thessaloniki 54124, Greece}
	
	\email{odysgiat@math.auth.gr}
	
	\date{\today}
	
	

	\begin{abstract}
		We present new techniques for removing arrows of bound quiver algebras, reducing thus the Finitistic Dimension Conjecture $\mathsf{(FDC)}$ for a given algebra to a smaller one. Unlike the classic arrow removal operation of Green-Psaroudakis-Solberg, our methods allow for removing arrows even when they occur in every generating set for the defining admissible ideal of the algebra.
		Our first main result establishes an equivalence for the finiteness of the fi\-ni\-ti\-stic (and global) dimensions of a ring $ \La $ and its quotient $ \La / K $, under specific homological and structural conditions on the ideal $ K $, in the broader context of left artinian rings. The application of this result to bound quiver algebras suggests the notion of removability for sets of arrows, and we prove that successive arrow removals of this sort lead to a uniquely defined arrow reduced version of the algebra.
		Towards the opposite direction, we characterize generalized arrow removal algebras via removable multiplicative bimodules, and introduce trivial one-arrow extensions as a novel combinatorial construction for adding arrows while preserving the finiteness of the finitistic (and global) dimensions.
		All these new techniques are illustrated through various concrete bound quiver algebras, for which confirming the $ \mathsf{(FDC)} $ had previously seemed intractable to the best of our knowledge.
	\end{abstract}

	\maketitle

	\setcounter{tocdepth}{1} \tableofcontents

	\section{Introduction and main results}

	\smallskip
	
	One of the longstanding homological conjectures in the representation theory of Artin algebras, namely the \emph{Finitistic Dimension Conjecture $\mathsf{(FDC)}$}, claims that the little finitistic dimension of every Artin algebra $ \La $, denoted by $\fpd \La $, is finite. Recall that the little finitistic dimension of any associative ring with unit is by definition the supremum of the projective dimensions of finitely generated modules with finite projective dimension. A stronger version of the conjecture asserts that the big finitistic dimension of every Artin algebra $ \La $, denoted by $\Fpd \La $, is finite, where the supremum is taken now over all modules (not necessarily finitely generated) with finite projective dimension.
	
	The importance of the finitistic dimensions lies in the fact that they provide a better measure for the homological complexity of the respective module categories of any ring $\La$ compared to the global dimension, denoted by $\gd \La$, when the latter is infinite. Moreover, most of the other homological conjectures in the representation theory of Artin algebras follow from the validity of the $\mathsf{(FDC)}$, see for instance the survey \cite{Happel}{}.

	Although there are many classes of Artin algebras that are known to satisfy the $\mathsf{(FDC)}$ -- see for instance the survey \cite{Huisgentale}{} -- there is a lack of techniques for proving that a given algebra outside of these classes satisfies the conjecture, even if we restrict our attention to finite dimensional algebras over a field.
	In an attempt to fill this gap, Green-Psaroudakis-Solberg introduced the \emph{arrow removal operation} \cite{arrowrem1}{} which made it possible to reduce the finiteness of the little finitistic dimension for a bound quiver algebra to the same problem for a smaller algebra by removing special sets of arrows. Specifically, they showed for a bound quiver algebra $\La = k Q / I $ that if (i) the admissible ideal $I$ can be generated by relations avoiding all arrows in some set $A$, and (ii) every path divided by an arrow in $ A $ (at least) twice is in $ I $, then the little finitistic dimension of $\La $ is finite if and only if the same holds for the quotient algebra of $\La $ over the ideal $K_A = \la A + I \ra $ generated by the arrows in $ A $.
	
	We recall that a bound quiver algebra $ \La = k Q / I $ consists of a finite quiver $ Q $, a field $ k $ and an ideal $ I $ of the path algebra $ k Q $. Furthermore, the ideal $ I $ is required to be admissible, that is every path occurring in it has length at least two, and for some $ N > 1 $ all paths of length $ N $ are contained in $ I $. It is well-known that bound quiver algebras are the basic finite dimensional algebras with one-dimensional simple modules. If the base field $ k $ is algebraically closed, then the latter condition is redundant, implying in particular that bound quiver algebras over $ k $ represent all finite dimensional algebras over $ k $ up to Morita equivalence.

	In this paper, we develop new tools that allow us to remove arrows even when they occur in every generating set for the defining admissible ideal. Our reduction technique is based on the following more general result.

	\begin{thmA}[\mbox{\cref{thm:main.III}{}}]		\hypertarget{thmA}{}
		Let $K $ be an ideal of a left artinian ring $\La $ contained in its Jacobson radical, and such that the natural epimorphism $\La \epic \La / K$ splits. Assume furthermore that $\pd K_\La < \infty $, and at least one of the conditions $\pd_\La K < \infty $ and $K ^2 = 0 $ is satisfied. Then
		\[
		\fpd \La < \infty \, \, \iff \, \, \fpd { \La / K } < \infty 
		\]
		and the equivalence remains valid if $\fpd $ is replaced by $\Fpd $ or $\gd$.
	\end{thmA}

	Returning to the context of a bound quiver algebra $ \La = k Q / I $, we define a set of arrows $ A $ to be \emph{removable} if (i) the ideal $K = K_A$ satisfies the assumptions of 
	\hyperlink{thmA}{Theorem~A}, and (ii) the natural epimorphism $\La \epic \La / K_A $ admits a canonical section monomorphism. See \cref{defn:ar.rem.6}{} and the preceding lemma for more details on the latter condition, which we call \emph{pre-removability}.
	
	
	A direct application of \hyperlink{thmA}{Theorem~A} yields the following result.
	
	\begin{thmB}[\mbox{\cref{thm:GAR}{}}]		\phantomsection		\hypertarget{thmB}{}
		Let $A $ be a removable set of arrows in $\La = k Q / I$. Then
		\[
		\fpd \La < \infty  \,\, \iff \,\,  \fpd { \La / \la A + I \ra } <\infty
		\]
		and the equivalence remains valid if $\fpd $ is replaced by $\Fpd $ or $\gd$.
	\end{thmB}

	Moreover, we show that if one keeps removing arrows from a bound quiver algebra through successive applications of \hyperlink{thmB}{Theorem~B} until no more arrows can be removed, then the resulting algebra is uniquely determined; see \cref{prop:ar.rem.2}{}. This is achieved when the quotient algebras produced at every step are considered as bound quiver algebras in the canonical way specified in \cref{defn:2}{}. See also \cref{exam:ar.rem.5}{} for a concrete algebra illustrating this dependence.
	
	We note that in \cite{EGPS}, a set of explicit combinatorial conditions has been established for an arrow $ \a $ occurring in a minimal monomial relation of $ \La = k Q / I $, ensuring that the finiteness of $ \fpd \La / K_{ \{ \a \} } $ implies the finiteness of $ \fpd \La $. Such an arrow is two-sided removable according to \cite[Lemma~4.2]{EGPS}{}, since it is already pre-removable by its definition. We note that the techniques used in \cite{EGPS}{} differ from those presented here, relying heavily on the use of non-commutative Gr\"{o}bner bases in the sense of Green \cite{Green}{} and arguments within the framework of abelian category cleft extensions.

	
	It is also worth noting that the arrow removal operation of Green-Psaroudakis-Solberg falls into the scope of our reduction technique, a non-trivial fact that requires the use of non-commutative Gr\"{o}bner bases in the sense of Green \cite{Green}{}. Specifically, we show in the \hyperlink{appendix}{Appendix} that a set of arrows $ A $ induces an arrow removal algebra of $ \La $ in the sense of \cite{arrowrem1}{} if and only if it is pre-removable and $ K_A $ is projective as a $\La$-bimodule; see \cref{prop:ar.rem.1}{}.

	Based on the results of \cite{arrowrem2}{} and the fact that condition (ii) of the arrow removal operation is redundant when $ A $ comprises a single arrow (see \cref{cor:ar.rem.1}{}), we show that arrows not occurring in some generating set for $ I $ do not affect the homological behavior of $\La = k Q / I $.
	Call a set of arrows $ A $ \emph{redundant} if for every $ \a \in A $ there is a generating set for $ I $ avoiding $ \a $. Then, we prove the following:

	\begin{thmC}[\mbox{\cref{cor:ar.rem.3}{}}]		\hypertarget{thmC}{}
		Let $\La = k Q / I $ be a bound quiver algebra. For any set of redundant arrows $ A $ and the quotient algebra $ \Ga = { \La / \la A + I \ra } $:
		\begin{enumerate}[\rm(i)]
			\item It holds that $\fpd \La < \infty $ if and only if $\fpd \Ga < \infty $, and the equivalence remains valid if $\fpd$ is replaced by $\Fpd$ or $\gd$.
			\item The algebra $\La $ is Iwanaga-Gorenstein if and only if the same holds for $  \Ga $.
			\item The algebra $\La $ satisfies the finite generation condition for the Hochschild cohomology if and only if the same holds for $  \Ga $.
			\item The singularity categories of $\La $ and $ \Ga $ are triangle equivalent.
		\end{enumerate}
	
	\end{thmC}

	Moreover, using techniques of cleft extensions developed in \cite{arrowrem1}{}, we show that the respective homological dimensions of the two algebras in \hyperlink{thmC}{Theorem~C} differ at most by one, and are almost always equal; see \cref{thm:1}{} for details.
	
	Turning to the inverse problem, we investigate ways of adding arrows to a bound quiver algebra so that the finiteness of the finitistic (or global) dimensions is preserved. To achieve this goal we utilize the notion of \emph{multiplicative bimodules}, see \cite{Pierce}{}. See also \cite{thetaexts}{}, where the same concept is introduced in the broader context of arbitrary rings under the term $\theta$-extensions. Specifically, a pair $ ( M, \theta ) $ is a multiplicative bimodule over a ring $ \Ga $ if $ M $ is a $ \Ga $-bimodule and $ \theta \colon M \otimes_\Ga M \to M $ is an associative $ \Ga $-bimodule homomorphism. Then $ E = \Ga \ltimes_\theta M $ denotes a new ring that is equal to $ \Ga \oplus M $ as an abelian group, where $ \wt{ M } = 0 \oplus M $ is an ideal.
	
	Our most general result in this direction characterizes \emph{generalized arrow removal algebras} of $ \La = k Q / I $ (i.e.\ algebras $ \Ga $ isomorphic to a quotient $ \La / K_A $ where $ A $ is a removable set of arrows) in terms of \emph{removable multiplicative bimodules}. This is essentially achieved through \cref{lem:ar.rem.8}{}; see also \hyperlink{thm:main.ar.rem.3}{\cref{thm:main.ar.rem.3}{}}{}. However, it is difficult to detect removable multiplicative bimodules for concrete examples in general, as it entails controlling the finiteness of the projective dimensions of $ \wt{ M } $ over $ E = \Ga \ltimes_\theta M $. Even the strongest homological condition on $ M $ over $ \Ga $, i.e.\ $ M $ being a projective $ \Ga $-bimodule, does not ensure removability; see \cref{exam:ar.rem.7}{}.

	We tackle this obstacle by introducing special kinds of removable multiplicative bimodules,
	the first type of which entails bimodules over an Artin algebra $ \Ga $ that satisfy some well-studied homological conditions; see \cref{defn:ar.rem.5}{}. In particular, the $ \Ga $-bimodules $ M $ considered here are tensor nilpotent, i.e.\ $ M^{ \otimes_\Ga i } = 0 $ for some $ i > 1 $, a necessary condition according to \cref{exam:ar.rem.8}{}.

	\begin{thmD}[\cref{cor:ar.rem.6}{}]		\hypertarget{thmD}{}
		Let $(M, \theta)$ be a multiplicative bimodule for an Artin algebra $\Ga$, where $M$ is a tensor nilpotent $ \Ga $-bimodule. If
		\begin{enumerate}[\rm(i)]
			\item the bimodule $M$ is perfect over $\Ga$, or
			\item the bimodule $M$ is right perfect over $\Ga$ and $\theta = 0$,
		\end{enumerate}
		then $ ( M , \theta ) $ is removable.
		In particular, it holds that
		\[
		\fpd { \Ga \ltimes_\theta M } < \infty  \, \, \iff  \, \,  \fpd \Ga < \infty
		\]
		and the equivalence remains valid if $\fpd $ is replaced by $\Fpd $ or $\gd $.
	\end{thmD}  
	
	We note that perfect bimodules have been studied recently in the context of injective generation for derived module categories \cite{PanosPsarou}{} and Iwanaga-Gorensteiness of noetherian rings \cite{Panos}{}. We also refer the reader to \cite{perfectbimods}{}, and to \cite{FGR,Beli2}{} for two classical treatments in the context of abelian categories.

	Our approach to the inverse of the arrow removal operation terminates with the following purely combinatorial construction attaching a new arrow to a bound quiver algebra, and possibly new relations ensuring in particular that every path passing through the new arrow at least twice is zero.

	\begin{conE}[\cref{exam:ar.rem.11}{}]		\hypertarget{conE}{}
		Let $\Ga = k Q_\Ga / I_\Ga $ be a bound quiver algebra, and let $i , j $ be a pair of distinct vertices. Let $ V $ be a $ \Ga $-submodule of $\Ga e_i$ such that $ e_j \Ga e_i \subseteq V \subseteq \rad _\Ga \Ga e_i $. We denote by $ \Ga_{ i \to j }^V $ the algebra $ k Q / I $, where:
		\begin{enumerate}[\rm(i)]
			\item $Q = Q_\Ga \dot \cup \{ \a \colon i \to j \}$ for a new arrow $\a \colon i \to j$;
			\item $I $ is the ideal of $k Q $ generated by $ I_\Ga \cup \{ z \a \, | \, z \in k Q_\Ga e_i \textrm{ and } z + I_\Ga \in V \} $.
		\end{enumerate}
		Furthermore, we call the algebra $ \Ga_{ i \to j }^V $ a \emph{trivial one-arrow extension} of $ \Ga $. 
	\end{conE}

	Our interest in the above construction is due to the following result, which rests partly on \hyperlink{thmD}{Theorem~D}.
	
	\begin{thmF}		\hypertarget{thmF}{}
		Let $ \Ga_{ i \to j }^V $ be a trivial one-arrow extension of $ \Ga $. Then
			\[
				\fpd { \Ga_{ i \to j }^V } < \infty  \,\, \iff \,\, 	\fpd \Ga < \infty 
			\]
		and the equivalence remains valid if $ \fpd $ is replaced by $ \Fpd $ or $ \gd $.
	\end{thmF}

	We close this introduction by outlining the contents of the paper. 
	
	We start \cref{subsec:rad.pres.clefts.1}{} by recalling that a ring cleft extension $(\La, \Ga, \pi , \i)$ consists of ring homomorphisms $ \i \colon \Ga \monicc \La $ and $ \pi \colon \La \epic \Ga $ such that $ \pi \i = \id_\Ga $. Then we introduce a special class of such extensions, i.e.\ \emph{radical-preserving ring cleft extensions with superfluous kernel}, and establish their basic ring-theoretic properties. In \cref{subsec:rad.pres.clefts.2}{}, we introduce a class of $\La$-modules associated with an arbitrary ring cleft extension $(\La, \Ga, \pi , \i )$, i.e.\ $ \La $-modules with a $ \Ga $-submodule admitting an \emph{absorbing} $ \La $-complement, and study their crucial homological properties; see \cref{defn:cleft.1}{} and \cref{lem:cleft.1}{}. The close homological connection between these $\La$-modules and their $\Ga$-submodules with an absorbing $\La$-complement leads in particular to the proof of \hyperlink{thmA}{Theorem~A} in \cref{subsec:rad.pres.clefts.3}{}.
	
	In \cref{subsec:3.1}{}, we prove \hyperlink{thmB}{Theorem~B}, and express finite global dimension as the removability of the set of all arrows in a bound quiver algebras; see \cref{exam:ar.rem.0}{}. In \cref{subsec:3.2}{}, we introduce some non-triviality conditions for a bound quiver algebra ensuring that the algebra does not belong to a class known to satisfy the $ \mathsf{(FDC)} $, or fulfill the requirements of a well-known reduction technique to the best of our knowledge. In \cref{subsec:3.3}{}, we apply \hyperlink{thmB}{Theorem~B} to concrete bound quiver algebras that are non-trivial in the sense of \cref{defn:irredu}{}.

	In \cref{subsec:4.1}{}, we develop the technical machinery leading to the definition of the \emph{arrow reduced version} of a bound quiver algebra. In particular, we use this machinery to show that if we keep removing removable sets of arrows from a bound quiver algebra while considering the produced quotients as bound quiver algebras in a canonical way, then the resulting bound quiver algebra when no more arrows can be removed is uniquely determined, even though the intermediate algebras can differ. We call the arrows of the initial algebra that are not part of its arrow reduced version \emph{eventually removable}.
	
	In \cref{subsec:4.2}{}, we develop two combinatorial criteria, one for eliminating arrows from being eventually removable and one for showing that a bound quiver algebras satisfying certain conditions is Iwanaga-Gorenstein. In \cref{subsec:4.3}{}, we compute the arrow reduced version of the example algebras in \cref{subsec:3.3}{}, proving thus that they all have finite finitistic dimensions.
	
	In \cref{subsec:adding.ar.1}{}, we lay the foundation for the rest of the section. In particular, we show that \emph{nilpotent multiplicative bimodules} over a bound quiver algebra are exactly the ones inducing a bigger bound quiver algebra by adding new arrows, and possibly new relations as well. As a consequence of this result we deduce that the natural algebra epimorphism induced by an ideal contained in the square of the Jacobson radical of a bound quiver algebra never splits (\cref{cor:ar.rem.4}{}).
	
	In \cref{subsec:adding.ar.2}{}, we characterize generalized arrow removal algebras as described above and prove \hyperlink{thmD}{Theorem~D}. Finally, in \cref{subsec:adding.ar.3}{}, we introduce the notion of \emph{strongly-finite} right projective dimension of a bimodule over a perfect ring, taking into account the interaction between the homological aspect of the bimodule from the right and the support of the bimodule from the left; see \cref{defn:strongly-finite}{} for details. In particular, we show that such a bimodule over an Artin algebra is tensor nilpotent and right perfect (\cref{lem:ar.rem.5}{}), and thus the induced trivial multiplicative bimodule is removable. Moreover, we introduce \hyperlink{conE}{Construction~E} and prove \hyperlink{thmF}{Theorem~F} by showing that a trivial one-arrow extension of a bound quiver algebra $ \Ga $ is isomorphic to the trivial extension of $ \Ga $ induced by a bimodule of strongly-finite right projective dimension.
	This novel construction is then illustrated in \cref{exam:3}{} with two concrete bound quiver algebras, which are highly non-trivial in the sense of \cref{defn:irredu}{} and are shown to have finite finitistic dimensions through \hyperlink{thmF}{Theorem~F}.

	\subsection*{Notation}

We denote by $J( R)$ the Jacobson radical of an associative ring $R$ with unit. A module over a ring will be a \emph{left} module unless stated otherwise.
For a module ${}_R M$, we denote by $\rad _R M$ its radical and by $\topp_R M$ the induced quotient $M / \rad _R M $. By $\pd {}_R M$, $\fd {}_R M$ and $\idim {}_R M$ we denote the projective, flat and injective dimension of $ M $, respectively. Similarly, we write $N_R$ to denote a right $R$-module, its radical is denoted by $\rad N_R $ and so on.
We denote by $R \rMod$ (resp.\ $R \rmod$) the category of left (finitely generated) $R$-modules. The respective right $R$-module categories are denoted by $ R^{\mathrm{op}} \rMod $ and $ R^{\mathrm{op}} \rmod $. The little finitistic, big finitistic and global dimension of $R$ are denoted by $\fpd R$, $\Fpd R$ and $\gd R$, respectively.
If $R$ is a semiprimary ring, then $\Ll (R )$ denotes its \emph{Loewy length}.


\subsection*{Acknowledgements}

I would like to thank my Ph.D.\ supervisor, Chrysostomos Psa\-rou\-da\-kis, for several useful discussions regarding this paper. I would also like to express my gratitude to Steffen Koenig for his warm hospitality at the Institute of Algebra and Number Theory of the University of Stuttgart, in the spring semester of 2024. A big part of this project was carried out during this research stay, when I also had the opportunity to present parts of the project at the seminar of the group. Finally, I would also like to thank Sibylle Schroll for her kind hospitality at the Mathematical Institute of the University of Cologne in March 2025, where I also had the opportunity to present parts of the project at the local seminar.

The present research project was supported by the Hellenic Foundation for
Research and Innovation (3rd Call for HFRI Ph.D.\
Fellowships, FN: 47510/03.04.2022).

	\medskip
	
	\section{Submodules with absorbing complements}	\label{sec:rad.pres.clefts}
	
	\smallskip
	
	This section is divided into three parts. In the first one, we introduce radical-preserving ring cleft extensions with superfluous kernel and establish some of their basic ring-theoretic properties. In the second part, we introduce a special class of $\La $-modules for a ring cleft extension $( \La , \Ga , \pi , \i )$, consisting of the modules that admit a $\Ga$-submodule with an absorbing $\La $-complement, see \cref{defn:cleft.1}{}. Furthermore, we showcase the strong homological relationship between such $\La$-modules and the associated $\Ga $-submodules, by showing that the minimal projective resolution of such a $\Ga$-submodule is a direct summand of the minimal projective resolution of the initial $\La$-module, see \cref{lem:cleft.1}{}. Finally, we exploit this relationship in order to establish a general reduction result for artinian rings (\cref{thm:main.III}{}), which will serve as the basis for the main theorem of \cref{sec:gen.ar.rem}{}.

	\subsection{Radical-preserving cleft extensions with superfluous kernel}		\label{subsec:rad.pres.clefts.1}
	
	In this subsection, we introduce the notion of radical-preserving ring cleft extensions with superfluous kernel and establish preliminary ring-theoretic facts about them. The usefulness of this new notion lies on the fact that it provides a context for the interpretation of the nice behavior of the arrow removal operation \cite{arrowrem1}{} with respect to the invariance of finiteness of the little finitistic dimension, while allowing for a significant generalization.
	
	A \emph{ring cleft extension} is a quadruple $(\La, \Ga, \i, \pi )$ where $ \i \colon \Ga \monicc \La$ and $ \pi \colon \La \epic \Ga$ are two ring homomorphisms such that $ \pi \i = \id_\Ga$. The employed terminology is due to Beligiannis \cite{Beligiannis}{}, who was the first to study cleft extenions in the context of arbitrary abelian categories. 
	Furthermore, we find this setup convenient as it fixes the homomorphisms $ \i $ and $ \pi $, while there are more than one retraction epimorphisms for $ \i $ and more than one  section monomorphisms for $ \pi $ in general.
	

	Let $\La $ be a ring and $K$ an ideal. We say that the quotient ring $ \La / K $ is \emph{split} or that \emph{$K$ induces a split quotient of $\La $} if there is a ring cleft extension of the form $(\La , \La / K , \i , \pi )$ where $\pi $ is the natural epimorphism. Similarly, we say that a ring epimorphism $ \pi \colon \La \epic \Ga $ \emph{splits} or that \emph{$\pi $ is a split epimorphism} if there is a ring cleft extension of the form $(\La , \Ga , \pi , \i )$.

	Within the context of Artin algebras, we adopt a specialized interpretation of these notions, defined as follows. If $\La $ and $\Ga $ are two Artin algebras over the same commutative artinian ring $k$, then the quadruple $(\La, \Ga, \i, \pi )$ is an \emph{algebra cleft extension} if it is a ring cleft extension such that $i$ and $\pi $ respect the $k$-algebra structure of the rings. Analogously, if $K$ is an ideal of $\La $ then the quotient algebra $ { \La / K } $ is \emph{split}, or \emph{$K$ induces a split algebra quotient of $\La $}, if there is an algebra cleft extension of the form $(\La , { \La / K } , \i , \pi )$ where $ \pi $ is the natural epimorphism. Furthermore, an algebra epimorphism $\pi \colon \La \epic \Ga $ \emph{splits}, or \emph{$\pi $ is a split algebra epimorphism}, if there is an \emph{algebra} cleft extension of the form $(\La , \Ga , \pi , \i )$.
	
	In what follows, a cleft extension will be a cleft extension of rings or algebras. Furthermore, we assume without loss of generality that the monomorphism $ \i $ in a cleft extension $(\La, \Ga, \i , \pi )$ is an inclusion unless stated otherwise. In particular, it holds that $\La = { \Ga \oplus \Ker \pi } $.

	\begin{defn}	
		\label{defn:1}
		A cleft extension $(\La, \Ga, \i, \pi)$ is \emph{radical-preserving with superfluous kernel} if both the kernel of $\pi $ and the Jacobson radical of $\Ga $ are contained in the Jacobson radical of $\La $.
	\end{defn}
	
	Recall that a ring homomorphism $\phi \colon A \to B $ is called \emph{radical-preserving} \cite{Giata1}{} if the image of the Jacobson radical of $ A $ under $ \phi $ is contained in the Jacobson radical of $ B $, and its kernel is called \emph{superfluous} if it is contained in the Jacobson radical of $A$. Therefore, a cleft extension $(\La, \Ga, \i, \pi)$ satisfies \cref{defn:1}{} if and only if both $ \i $ and $ \pi $ are radical-preserving with superfluous kernel. 
	According to the next lemma, it is sufficient to require that the kernel of $ \pi $ is superfluous in most cases.

	\begin{lem}
		\label{lem:cleft.2}
		The following are equivalent for a ring cleft extension $(\La, \Ga, \i, \pi)$ where $\La $ is semilocal.
		\begin{enumerate}[\rm(i)]
			\item The cleft extension is radical-preserving with superfluous kernel.
			\item The kernel of $\pi$ is superfluous.
			\item It holds that $ J(\La) = { J(\Ga) \oplus { \Ker \pi } } $.
		\end{enumerate}
		In particular, the above conditions are equivalent for all algebra cleft extensions.
	\end{lem}
	
	\begin{proof}
		The implications (i) $\Rightarrow$ (ii) and (iii) $\Rightarrow$ (i) are trivial.
		
		We begin by showing that $J ( { \La / K } ) = { J ( \La ) / K } $ for every ideal $ K $ of $ \La $ contained in $J ( \La )$, since $\La $ is semilocal. The inclusion $ { J(\La ) / K } \subseteq { J ( \La / K ) } $ is well-known, see for instance \cite[Corollary~15.8]{AndersonFuller}{}. For the inverse inclusion, note that the quotient of $ { \La / K } $ over $ { J( \La ) / K } $ is isomorphic as a $\La$-module to the semisimple module $ { \La / J( \La ) } $ and, thus, it is a semisimple ring. It follows that the inclusion $ { J ( \La / K  ) } \subseteq { J( \La ) / K } $ also holds, see \cite[Corollary 15.6]{AndersonFuller}{}. Returning to the given ring cleft extension, we have that $J ( \Ga ) = { \pi ( J ( \La ) ) } $ and the ring $\Ga $ is semilocal if $\La $ is semilocal and the kernel of $\pi $ is superfluous.
		
		(ii) $\Rightarrow $ (i): Take $x \in J ( \Ga ) $ and $y \in J ( \La ) $ such that $x = \pi ( y )$. Then $y = \i ( x ) + ( y - \i ( x ) )$, where $ y - \i ( x ) \in J( \La )$ as $\pi ( y - \i ( x ) ) = 0 $, implying that $ \i ( x ) \in J ( \La )$.
		
		(i) $\Rightarrow$ (iii): Note that the kernel of the composition of $ \pi $ with the natural epimorphism $\Ga \epic { \Ga / J ( \Ga ) } $ is equal to $ { J ( \Ga ) \oplus \Ker \pi } $ and contained in $  J ( \La ) $. It follows that $\La / ( { J ( \Ga ) \oplus \Ker \pi } ) $ is a semisimple ring as it is isomorphic to $ { \Ga / J ( \Ga ) } $, implying that $ J ( \La ) \subseteq { J ( \Ga ) \oplus \Ker \pi } $.
	\end{proof}

	We close this subsection by showing that radical-preserving cleft extensions with superfluous kernel behave well with respect to semiperfectness. We refer the reader to \cite{AndersonFuller}{} for an introduction to semiperfect rings, or to the second section of \cite{Giata1}{} for a condensed account of the necessary background.

	\begin{lem}
		\label{lem:semiperfect.clefts}
		Let $(\La, \Ga, \i, \pi)$ be a radical-preserving ring cleft extension with superfluous kernel. Then the ring $\La $ is (basic) semiperfect if and only if $\Ga $ is (basic) semiperfect. In particular, a complete or basic set of primitive orthogonal idempotents for $\Ga$ retains this status over $\La $ if the rings are semiperfect.
	\end{lem}
	
	\begin{proof}
		Let $ \{ e_j \}_j $ be a complete set of primitive orthogonal idempotents for $\Ga$. The crucial fact here is that we have the direct sum decompositions of abelian groups $\La e_j = \Ga e_j \oplus ( \Ker \pi ) e_j$ and $J( \La ) e_j = J(\Ga ) e_j \oplus ( \Ker \pi )  e_j$ due to \cref{lem:cleft.2}{}, implying that $ \topp _\La \La e_j $ and $\topp _\Ga \Ga e_j  $ are isomorphic as $\Ga $-modules. Note that the $\Ga $-structure on $ \topp _\La \La e_j $ is the one induced from the fact that $\Ker \pi ( { \topp_\La \La e_j } )  = 0$, and it coincides with the $\Ga$-structure that comes from restriction of scalars along $i$.
		As a consequence, the module ${}_\La \La e_j$ is local if and only if the module ${}_\Ga  \Ga e_j $ is local, implying that the ring $\La $ is semiperfect if and only if $\Ga $ is semiperfect. Furthermore, the remaining claims of the lemma follow from the fact that
		\[
		\La e_{j_1} \simeq  \La e_{j_2}   \Leftrightarrow    \topp _\La  \La e_{j_1}  \simeq \topp _\La  \La e_{j_2}   \Leftrightarrow	
		\topp _\Ga  \Ga e_{j_1} \simeq \topp _\Ga  \Ga e_{j_2}  \Leftrightarrow     \Ga e_{j_1} \simeq    \Ga e_{j_2}
		\]
		for any indices $j_1$, $j_2$; see also \cite[Proposition 17.18]{AndersonFuller}{}. 
	\end{proof}

	\subsection{Submodules with absorbing complements}		\label{subsec:rad.pres.clefts.2}

	Given a ring cleft extension $(\La, \Ga, \i, \pi)$, this subsection introduces and investigates $\La$-modules possessing a $ \Ga $-submodule with an absorbing $\La$-complement. We establish a strong homological connection between such a $\La$-module and its associated $\Ga$-submodules, which is instrumental for the main results in \cref{sec:gen.ar.rem,sec:ar.red.ver}{}.

	\begin{defn}
		\label{defn:cleft.1}
		Let $(\La, \Ga, \pi , \i )$ be a ring cleft extension, and let $ M $ be a $ \La $-module. For a $ \Ga $-submodule $ N $ and a $ \La $-submodule $ M ' $ of $ M $:
		\begin{enumerate}[\rm(i)]
			\item The submodule $ N $ is called \emph{generating} if its image generates $ M $ over $\La $.
			\item The submodule $M'$ is a \emph{$\La$-complement} of $ N $ if $M = N \oplus M'$ as abelian groups. 
			\item If $ M ' $ is a $\La$-complement of $ N $, we call it \emph{absorbing} if $ { \Ker \pi } N \subseteq M'$, and \emph{tight} if equality holds.
		\end{enumerate}
	\end{defn}

	As a first observation, we note that a $ \La $-module $ M $ is in the essential image of the restriction functor along $\pi $ if and only if it admits the zero submodule as an absorbing (or tight) $\La$-complement. Furthermore, a $\Ga$-submodule $N$ of $ M $ admits a tight $\La $-complement if and only if it admits an absorbing $\La $-complement and is generating. Its $\La$-complement is unique in that case and equal to $ { \Ker \pi } N $. 

	In the following lemma, we characterize the $\La$-modules in the essential image of the induction functor along $ \i $, and in the intersection of this subcategory with the essential image of the restriction functor along $ \pi $.

	\begin{lem}			\hypertarget{lem:cleft.4}{}
		\label{lem:cleft.4}
		Let $(\La, \Ga, \pi , \i )$ be a ring cleft extension.
		\begin{enumerate}[\rm(i)]
			\item A $ \La $-module $ M$ is in the essential image of the tensor functor $ { \La \otimes_\Ga - } $ if and only if $M$ admits a $\Ga$-submodule $ N $ with tight $\La$-complement, such that the map $ { \Ker \pi \otimes_\Ga N } \epic { { \Ker \pi } \, N } $ defined by $ { x \otimes n } \mapsto { x n } $ is an isomorphism.
			\item A $ \La $-module $ M$ is in the essential image of both the tensor functor $ { \La \otimes_\Ga - } $ and the restriction functor $ { {}_\La ( - ) } $ if and only if $\Ker \pi \otimes_\Ga M = 0 $.
		\end{enumerate}
	\end{lem}

	\begin{proof}
		Recall that there is a decomposition $\La = { \Ga \oplus { \Ker \pi } } $ of right $\Ga$-modules. Therefore, for any $ \Ga $-module $ N $, there is an isomorphism of abelian groups of the form $ { \La \otimes_\Ga N } \simeq N \oplus ( { { \Ker \pi } \otimes_\Ga N } ) $, where the isomorphism is given by $( \g + x ) \otimes n \mapsto ( \g n  \, , \,  x \otimes n ) $ with inverse $(n  \, , \,  \sum_i x_i \otimes n_i ) \mapsto 1_\La \otimes n + \sum_i x_i \otimes n_i$ for every $ \g \in \Ga$, every $ x , x_i \in \Ker \pi$ and $n  , n_i \in N$. The inherited $\La $-action on $ N \oplus ( { { \Ker \pi } \otimes_\Ga N } ) $ is given by
		\[
		( \g + x ) (n  \, , \,  \sum_i x_i \otimes n_i ) = ( \g  n  \, , \,  x \otimes n + \sum_i [( \g + x ) x_i] \otimes n_i ).
		\]
		We consider the above isomorphism as an identification for the rest of the proof.
		
		(i) One implication follows immediately from the above identification. Indeed, it is straightforward to verify that $N$ is a $\Ga$-submodule of $ { \La \otimes_\Ga N } $ with tight $\La $-complement $  { \Ker \pi } N =  { { \Ker \pi } \otimes_\Ga N } $. For the converse implication, assume that a $ \La $-module $ M $ admits a $\Ga $-submodule $N$ with tight $\La $-complement such that the map $ { \Ker \pi } \otimes_\Ga N \epic { { \Ker \pi } N } $ defined by $ { x \otimes n } \mapsto { x n } $ is bijective. A straightforward calculation shows now that $ M $ is isomorphic to $ \La \otimes_\Ga N $ through the $\La$-homomorphism $ \La \otimes_\Ga N \to M $ induced by the $\La$-multiplication $ N $ inherits from $ M $.

		(ii) 	Suppose ${}_\La M $ is in the essential image of the functor $ \La \otimes_\Ga - $. By the preceding argument, the module $M$ admits a $\Ga$-submodule $N$ such that $M = N \oplus { { \Ker \pi } N } $. Moreover, the map $ { \Ker \pi } \otimes_\Ga N \to { \Ker \pi } N $ defined by $ { x \otimes n } \mapsto { x n } $ is bijective.
		If $M$ is also in the essential image of the functor $ {}_\La ( - ) $, then $ { \Ker \pi } N = 0 $. This directly implies $ M = N $ and $ { \Ker \pi } \otimes_\Ga N = 0 $.
		
		Conversely, assume $ { \Ker \pi } \otimes_\Ga M = 0 $. Then the map $ \La \otimes_\Ga M \to M $ defined by $ a \otimes m \mapsto a m $ is a well-defined $\La$-isomorphism. This shows that $M$ is in the essential image of $ \La \otimes_\Ga - $.
		Furthermore, since the map $ { \Ker \pi } \otimes_\Ga M \to { \Ker \pi } M $ defined by $ { x \otimes m } \mapsto { x m } $ is surjective, we have $ { \Ker \pi } M = 0 $, which is equivalent to $ M $ being in the essential image of the functor $ {}_\La ( - ) $.
	\end{proof}

	\begin{rem}
			\label{rem:1}
		Retaining the identification introduced in the proof of \cref{lem:cleft.4}{}, $\La$-modules admitting a $\Ga$-submodule with tight $\La$-complement are precisely those isomorphic to modules of the form $ { ( \La \otimes_\Ga N ) / Z } $. Here, we let $ N $ be an arbitrary $\Ga$-module, and $ Z $ is a $\La$-submodule of $\La \otimes_\Ga N$ contained in ${ \Ker \pi } \otimes_\Ga N$.
	\end{rem}

	\begin{rem}
		The notions introduced in \cref{defn:cleft.1}{} possess a dual version, where the notion of $\Ga $-submodule with $\La $-complement is self-dual. Specifically, a $\Ga$-submodule $N $ of $ M $ is \emph{co-generating} if the map $ M \to \Hom_\Ga ( \La , M / N ) $ defined by $ m \mapsto [ \l \mapsto   \l m + N ] $ is injective. Equivalently, the module $ N $ is co-generating if the zero submodule is the only $\La$-submodule of $ M $ contained in $ N $. The dual relationship with generating $\Ga$-submodules can be seen by observing that $N$ is generating if and only if the map $ \La \otimes_\Ga N \to M $ defined by $\l \otimes n \mapsto \l n $ is surjective.
		
		Moreover, a $\La $-complement $ M ' $ of $ N $ is \emph{co-absorbing} if $ { \Ker \pi } M' = 0 $, and \emph{co-tight} if in addition $ N $ is co-generating.
	\end{rem}

	The following lemma plays an essential role in the sequel.
	Before that, we recall briefly the standard way of constructing projective covers of (finitely generated) modules over a left perfect (resp.\ semiperfect) ring $ \La $. For simplicity, the reader may assume $ \La $ is an Artin algebra throughout this section.
	
	Let $ \{ e_j \}_j $ be a basic set of primitive orthogonal idempotents for a left perfect ring $\La$, that is for every projective indecomposable $\La$-module $P$ there is a unique $j$ such that $P \simeq \La e_j $. For any non-zero $\La$-module $X$, it holds that $\topp_\La X = X / \rad_\La X $ is a non-zero semisimple $\La$-module. Therefore, there is a subset $\{ x_i \}_i \subseteq X$ such that $\topp _\La X = \oplus_{i }  \La \ol{ x_i }$ and ${}_\La \La \ol{ x_i }$ is simple for every $ i $, where $\ol{ x_i }$ denotes the image of $x_i$ in $\topp _\La X$ under the natural epimorphism.
	
	Note that we may assume without loss of generality that $x_i = e_{g_i} x_i$ for every $i$, where $e_{g_i} $ is some primitive idempotent from the above. This follows from the fact that $e_j \ol{ x_i }$ has to be non-zero for some $j$ in any case implying that $\La \ol{ e_j x_i  } = \La \ol{ x_i }  $. Such a subset of $X$ is often called a \emph{$\La$-basis}. 
	
	The unique $\La$-homomorphism $P_\La (X) = \oplus_{i } \La e_{g_i} \to X$ defined by the correspondence $e_{g_i} \mapsto x_i$ is a projective cover of ${}_\La X$, see for instance \cite[Proposition 28.13]{AndersonFuller}{}. Furthermore, the above construction also works in case $\La $ is semiperfect and $X$ is finitely generated. However, if we want to construct a minimal projective resolution of $X$, then we need the extra assumption that $\La$ is left noetherian so that submodules of finitely generated $\La $-modules are again finitely generated.
	
	For the remainder of this subsection, we fix a radical-preserving ring cleft extension $(\La, \Ga, \pi , \i)$ with superfluous kernel. Throughtout, we assume $\La$ is either left perfect or left noetherian semiperfect. Note that $\Ga $ is also left perfect or left noetherian semiperfect, respectively, as a quotient ring of $\La $; see for instance \cite[Corollaries~27.9, 28.7]{AndersonFuller}{}. Furthermore, if $\La$ is left noetherian semiperfect, all modules discussed are assumed to be finitely generated.

	\begin{lem}
		\label{lem:cleft.1}
		Let $ M $ be a $ \La $-module, and let $ N $ be a $\Ga$-submodule with an absorbing $\La $-complement. Let $P_\La (M) \epic M$ and $P_\Ga (N) \epic N$ be projective covers. Then there is a commutative diagram of $\Ga$-modules of the form
		\phantomsection		\hypertarget{eq:1'}{}
		\begin{equation}		
			\label{eq:1}
			\begin{tikzcd}
				0 & {{}_\Gamma \Omega^1_\Lambda(M)} & {{}_\Gamma P_\La (M)} & {{}_\Gamma M} & 0 \\
				\\
				0 & {\Omega^1_\Gamma(N)} & {P_\Ga (N)} & N & 0
				\arrow[from=1-4, to=1-5]
				\arrow[from=3-4, to=3-5]
				\arrow[curve={height=-6pt}, hook', from=3-4, to=1-4]
				\arrow[curve={height=-6pt}, two heads, from=1-4, to=3-4]
				\arrow[from=3-3, to=3-4]
				\arrow[from=1-3, to=1-4]
				\arrow[curve={height=-6pt}, hook', from=3-3, to=1-3]
				\arrow[curve={height=-6pt}, two heads, from=1-3, to=3-3]
				\arrow[from=1-1, to=1-2]
				\arrow[from=3-1, to=3-2]
				\arrow[from=3-2, to=3-3]
				\arrow[curve={height=-6pt}, hook', from=3-2, to=1-2]
				\arrow[curve={height=-6pt}, two heads, from=1-2, to=3-2]
				\arrow[from=1-2, to=1-3]
			\end{tikzcd}
		\end{equation}
		where the rows are exact, and each vertical monomorphism is split and admits an absorbing $\La$-complement in a canonical way.
	\end{lem}

	\begin{proof}
		Let $\La$ be left perfect and let $ \{ e_j \}_j $ be a basic set of primitive orthogonal idempotents for $\Ga$. Note that $ \{ e_j \}_j $ is a set with the same properties over $\La$ according to \cref{lem:semiperfect.clefts}{}.
		Let $ K = \Ker \pi $ and let $M'$ be the absorbing $\La$-complement of $N$ in $M$. Then we have that
			\[
				\rad _\La M = J(\La ) M = J( \Ga ) N \oplus ( K N + J( \La ) M' ) = \rad _\Ga N \oplus ( K N + \rad_\La M')
			\]
		since $J(\La ) = { J( \Ga ) \oplus K } $ according to \cref{lem:cleft.2}, and $ M = N \oplus M' $.

		Now let $\{ n_i\}_i $ be a fixed $\Ga$-basis of $ N $. We claim that the module $\La \ol{ n_i }$, that is the $\La$-submodule of $\topp _\La M$ generated by $ \ol{ n_i } = n_i + \rad _\La M$, is simple for every $i$, and that the sum of all these submodules is direct.
		For the first claim, note that the $ \Ga $-structure that ${}_\La \La \ol{ n_i } $ inherits due to the fact that $ K  \La \ol{ n_i } = 0$ is the same as the $ \Ga $-structure induced by restriction of scalars along $ \i $. In particular, we have that ${}_\La \La \ol{ n_i } $ is simple if and only if ${}_\Ga \La \ol{ n_i }$ is simple. But the latter module is simple as it is isomorphic to ${}_\Ga \Ga \ol{ n_i }$ through the bijection defined by $ { { \g  n_i } + \rad_\La M } \mapsto { { \g n_i } + \rad_\Ga N } $, proving thus our first claim.
		
		For the second claim, let $ a_i  $ be non-zero elements of $\La $ for a finite number of indices $i$ such that $ \sum_i a_i  n_i \in \rad_\La M$, and write $a_i = \g_i + x_i $ for $ \g_i \in \Ga $ and $ x_i \in K $. It follows from the above decomposition of $\rad_\La M $ that $\sum_i \g_i  (  n_i + \rad_\Ga N) = 0$, implying that $\g_i n_i \in \rad_\Ga N$ for every $i$. Therefore, it holds that $ a_i n_i \in \rad_\La M$ for every $ i $, which completes the proof of our second claim.

		Next, we fix a $\La$-basis ${ m_l }_l $ of $ M' $. Then $\topp_\La M = \sum_i \La \ol{ n_i } + \sum_l \La \ol{ m_l }$ and, therefore, we have an equality of the form
		\[
		{ \topp_\La M } = { { \big( \! \oplus_i { \La \ol{ n_i } } \big) } \bigoplus { \big( \! \oplus_{l'} { \La \ol{m_{l'}} } \big) } }
		\]
		where $l'$ ranges over an appropriate subset of the initial index set, see for instance \cite[Proposition 9.3]{AndersonFuller}{}. Let $e_{g_i}$ and $e_{ v_{l'}} $ denote the primitive idempotents such that $n_i = e_{g_i } n_i$ and $m_{l'} = e_{v_{l'}} m_{l'}$ for every $i$ and $l'$, respectively. Then the map $P_\Ga (N) = { \oplus_i { \Ga e_{g_i } } } \epic N$ defined by $( \g_i e_{g_i } )_i \mapsto { \sum_i { \g_i  n_i } }$, where almost all $\g_i$ are zero, is a projective cover of ${}_\Ga N$. Similarly, the map $P_\La (M) = \big( \! \oplus_{i } \La e_{g_i} \big) \oplus \big( \! \oplus_{ l' } \La e_{v_{l'}} \big) \epic M$ defined by $( a_i e_{g_i} )_i \cup ( b_{l'} e_{v_{l'}})_{l'} \mapsto \sum_i a_i  n_i + \sum_{l'} b_{l'}  m_{l'} $ is a projective cover of ${}_\La M$.

		A key property for the final part of the proof is that $\Ga e_j $ is a $\Ga$-submodule of $ \La e_j $ with tight $\La$-complement $ K e_j $ for every idempotent $e_j$. Consequently, the natural injection $ P_\Ga (N) \monicc P_\La (M) $ is a split $\Ga$-monomorphism with absorbing $\La$-complement $( \oplus_i K e_{g_i} ) \oplus ( \oplus _{ l' } \La e_{v_{l'}} ) $. Moreover, both the split monomorphism and its retraction (induced by the given absorbing $\La$-complement of $ P_\Ga (N) $) make the respective squares commutative in diagram \hyperlink{eq:1'}{(\ref{eq:1})}.
		
		By the universal property of kernels, there exist unique $ \Ga $-homomorphisms between $\Omega^1_\Ga (N)$ and ${}_\Ga \Omega^1_\La (M)$ that make the respective squares commutative. Furthermore, the composition $\Omega^1_\Ga (N) \to {}_\Ga \Omega^1_\La (M) \to \Omega^1_\Ga (N)$ is the identity on $\Omega^1_\Ga (N)$, showing that the vertical map for the syzygies is also a split monomorphism.

		To show that the image of $ \Omega^1_\Ga (N) $ in $ \Omega^1_\La (M) $ also admits an absorbing $\La$-complement, we characterize the kernel of the retraction from ${}_\Ga \Omega^1_\La (M) $ to $ \Omega^1_\Ga (N)$. This kernel consists of elements $( x_i e_{g_i} )_i \cup ( b_l e_{v_l} )_l \in P_\La (M) $ such that $ x_i e_{g_i} \in K e_{g_i} $ and $ b_l e_{v_l} \in \La e_{v_l} $, and $\sum_i x_i  n_i + \sum_l b_l m_l = 0$. This kernel forms an absorbing $\La$-complement of the image of $ \Omega^1_\Ga (N) $ in $ \Omega^1_\La (M) $ because $ \La e_j = \Ga e_j \oplus K e_j $ for every idempotent $e_j$, completing the proof.
	\end{proof}

	\begin{cor}
			\label{cor:1}
		Let $ M $ be a $ \La $-module, and let $ N $ be a $\Ga$-submodule with an absorbing $\La $-complement. Then $\pd _\Ga N \leq \pd _\La M $.
	\end{cor}

	\begin{proof}
		The corollary follows from an inductive application of \cref{lem:cleft.1}{}.
	\end{proof}

	\begin{cor}
		\label{cor:cleft.2}
		Let $ N $ be a $ \Ga $-module, and let $ Z $ be a $ \La $-submodule of $ { \La \otimes_\Ga N } $ contained in $ { \Ker \pi } \otimes_\Ga N $. Then
			\[
				\pd N_\Ga \leq \pd _\La { ( \La \otimes _\Ga N ) / Z } .
			\]
		In particular, it holds that $ \pd _\Ga N \leq \min \{  { \pd _\La N } , \, { \pd _\La \La \otimes_\Ga N } \} $.
	\end{cor}
	
	\begin{proof}
		The first statement follows directly from \cref{cor:1}{} and \cref{rem:1}{}. The second one is a direct application of the first for $ Z = 0 $ and $ Z = { \Ker \pi } \otimes_\Ga N $.
	\end{proof}

	The following lemma outlines conditions under which projective dimensions are preserved when considering $\Ga$-modules as $\La$-modules via induction or restriction.

	\begin{lem}
		\label{lem:cleft.3}
		For a minimal projective resolution $ \mathbb{P} $ of a $ \Ga $-module $ N $, the following are equivalent.
		\begin{enumerate}[\rm(i)]
			\item The complex $\La \otimes_\Ga \mathbb{P}$ is exact.
			\item The group $\Tor_i^\Ga ( \Ker \pi , N ) $ is trivial for every $i \geq 1$.
			\item The complex $\La \otimes_\Ga \mathbb{P}$ is a minimal projective resolution of $ \La \otimes_\Ga N$.
		\end{enumerate}
		Furthermore, it holds that $ \pd _\Ga N = \pd _\La \La \otimes_\Ga N $ in that case.
	
		Moreover, the complex $\La \otimes_\Ga \mathbb{P}$ is a minimal projective resolution of $ {}_\La N$ if and only if the group $\Tor_i^\Ga ( \Ker \pi , N ) $ is trivial for every $i \geq 0$, in which case $\pd _\Ga N = \pd _\La N $.
	\end{lem}
	
	\begin{proof}
		The equivalence between (i) and (ii) follows from the isomorphisms
			\[
				\Tor_i^\Ga ( \La , N ) \simeq \Tor_i^\Ga ( \Ga , N ) \oplus \Tor_i^\Ga ( \Ker \pi , N ) \simeq \Tor_i^\Ga ( \Ker \pi , N )
			\]
		for every $i \geq 1 $, which hold since there is an isomorphism of right $\Ga$-modules $ \La_\Ga \simeq { \Ga_\Ga \oplus { \Ker \pi }_\Ga  } $ and the regular module $ \Ga_\Ga $ is flat.
		
		Assume that the complex $\La \otimes_\Ga \mathbb{P}$ is exact. Since induction functors preserve projectivity, the complex $\La \otimes_\Ga \mathbb{P}$ is a projective resolution of $ \La \otimes_\Ga N $. For minimality, we use that the monomorphism $\i \colon \Ga \monicc \La$ is radical-preserving (with superfluous kernel), as $( \La , \Ga , \pi , \i )$ is a radical-preserving cleft extension with superfluous kernel. According to \cite[Proposition~3.10]{Giata1}{}, this implies that the associated induction functor $ \La \otimes _\Ga - $ preserves non-trivially projective covers. In other words, if $ f \colon P \epic N $ is a projective cover over $ \Ga $, then $ \La \otimes_\Ga f \colon \La \otimes_\Ga P \epic \La \otimes_\Ga N $ is a projective cover over $ \La $. We deduce that the complex $\La \otimes_\Ga \mathbb{P}$ is a minimal projective resolution of $\La \otimes_\Ga N$. Moreover, we have that $ \pd _\Ga N = \pd _\La \La \otimes_\Ga N $ as the lengths of the complexes $ \La \otimes_\Ga \mathbb{P} $ and $ \mathbb{P}$ are equal.

		For the last part of the lemma, consider the short exact sequence of $\La$-$\Ga$-bimodules $0 \to { \Ker \pi } \to \La \xrightarrow{ \pi \, }  \Ga \to 0 $. Since tensoring with $ N $ over $\Ga$ is right exact, and $\Ga_\Ga$ is flat (so $\Tor_1^\Ga ( \Ga , N ) = 0 $), the induced long exact sequence is of the form
			\[
				0 \to { \Ker \pi } \otimes_\Ga N \to \La \otimes_\Ga N \to \Ga \otimes_\Ga N \to 0.
			\]
		If $\Tor_i^\Ga ( \Ker \pi , N ) = 0 $ for every $i \geq 0$, then in particular $ { \Ker \pi } \otimes_\Ga N = 0 $. Our claim follows now from the first part of the proof and the fact that ${}_\La \Ga \otimes_\Ga N \simeq {}_\La N $.
		
		For the converse, assume $\La \otimes_\Ga \mathbb{P}$ is a minimal projective resolution of ${}_\La N$
		In particular, we have that $\La \otimes\Ga N \simeq {}_\La N $. From \hyperlink{lem:cleft.4}{\cref{lem:cleft.4}{}.(ii)}{}, this isomorphism implies that ${ { \Ker \pi } \otimes_\Ga N } = 0$. This, combined with the first part of the proof, shows that $\Tor_i^\Ga ( \Ker \pi , N ) = 0 $ for all $ i \geq 0 $.
	\end{proof}

	Building on \cref{lem:cleft.1}{}, we establish an even stronger homological connection between a $ \La $-module and a $\Ga$-submodule with an absorbing $\La$-complement.
	
	\begin{prop}
		\label{prop:cleft.1}
		Let $ M $ be a $ \La $-module, and let $ N$ be a $\Ga$-submodule with an absorbing $\La$-complement. For any right $\Ga$-module $ L $, the group $ \Tor^\Ga_i ( L , N ) $ is a direct summand of $\Tor^\La_i (L , M) $ for every $i \geq 0$.
	\end{prop}
	
	\begin{proof}
		We begin by fixing a truncated minimal projective resolution
		\[
		\ol{\mathbb{Q}}_N \colon   \, \, \, \, \, \, \ldots  \to  Q_j \xrightarrow{g_j}  Q_{j-1} \to \ldots \to Q_1 \xrightarrow{g_1} Q_0  \to 0 
		\]
		of ${}_\Ga N$. Furthermore, we use $ \ol{\mathbb{Q}}_N$ to construct a truncated minimal projective $ \La $-resolution of $ M $
		\[
		\ol{\mathbb{P}}_M \colon   \, \, \, \, \, \, \ldots  \to  P_j \xrightarrow{f_j}  P_{j-1} \to \ldots \to P_1 \xrightarrow{f_1} P_0 \to 0 
		\]
		as in the proof of \cref{lem:cleft.1}{}.
		
		Recall that for every $ j \geq 0 $ there is a split $\Ga$-monomorphism $ i_j \colon Q_j \monicc {}_\Ga P_j $. The kernel of its canonical retraction $ \pi_j  $, described in the proof of \cref{lem:cleft.1}{}, is an absorbing $\La$-complement of (the image of) $ i_j$ containing $ { \Ker \pi }  P_j $. Furthermore, both $(i_j)_j$ and $ ( \pi_j)_j $ are chain maps between the complexes $\mathbb{Q}_N$ and ${}_\Ga \mathbb{P}_M$.

		Our aim is to show that $L \otimes_\Ga \ol{\mathbb{Q}}_N$ is a direct summand of $L \otimes_\La \ol{\mathbb{P}}_M$ as a complex of abelian groups, where $ L $ is viewed as a right $ \La $-module through restriction of scalars along $ \pi $ as per our convention.
		We do that by establishing the existence of split monomorphisms of abelian groups $\tilde{ i }_j \colon L \otimes_\Ga Q_j \monicc L \otimes_\La P_j $ with retractions $\tilde{\pi}_j $ for every $j \geq 0 $, such that both $( \tilde{ i }_j )_j $ and $ ( \tilde{ \pi }_j )_j $ are chain maps between the abelian group complexes $L \otimes_\Ga \mathbb{Q}_N $ and $ L \otimes_\La \mathbb{P}_M $. 
		
		First, observe that the correspondence $l \otimes q_j \mapsto l \otimes i_j (q_j) $, where $l \in L$ and $q_j \in Q_j$, defines an abelian group homomorphism $\tilde{i}_j \colon { L \otimes_\Ga Q_j } \to { L \otimes_\La P_j } $ for every $j \geq 0$. Indeed, we have that 
		\[
		{ ( l  \g ) \otimes q_j } \mapsto { (l  \g ) \otimes i_j (q_j) } = l \otimes { ( \g i_j (q_j) ) } = l \otimes { i_j ( \g q_j) }
		\]
		for every $\g \in \Ga$. Similarly, the correspondence $ { l \otimes p_j } \mapsto { l \otimes \pi_j ( p_j ) } $, where $l \in L $ and $p_j \in P_j$, defines an abelian group homomorphism $\tilde{\pi}_j \colon { L \otimes_\La P_j } \to { L \otimes_\Ga Q_j } $ for every $j \geq 0$. Indeed, it holds that
		\begin{multline*}
			{ ( l  ( { \g + x } ) ) \otimes p_j } = { ( l \g ) \otimes p_j } \mapsto		\\
			{ ( l \g ) \otimes \pi_j ( p_j ) } =  l \otimes ( \g  \pi_j ( p_j ) ) =  l \otimes { \pi_j ( \g p_j ) } = l \otimes \pi_j ( (\g + x ) p_j )
		\end{multline*}
		for every $\g \in \Ga$ and $x \in \Ker \pi$, as $L  { \Ker \pi } = 0$ and $\pi_j ( { \Ker \pi }  P_j ) = 0$ for every $j \geq 0$.
		
		The equality $ \tilde{ \pi}_j \circ \tilde{ i }_j = \id_{L \otimes_\Ga Q_j }$ is an immediate consequence of the definitions. Moreover, the commutativity of the diagram
			\[\begin{tikzcd}
				{L \otimes_\Gamma Q_j } & {L \otimes_\Gamma Q_{j - 1}} \\
				{L \otimes_\Lambda P_j} & {L \otimes_\Lambda P_{j - 1}}
				\arrow["{1 \otimes g_j}", from=1-1, to=1-2]
				\arrow["{\tilde{i}_j}"', curve={height=6pt}, hook, from=1-1, to=2-1]
				\arrow["{\tilde{i}_{j- 1}}"', curve={height=6pt}, hook, from=1-2, to=2-2]
				\arrow["{\tilde{\pi}_j}"'{pos=0.44}, curve={height=6pt}, two heads, from=2-1, to=1-1]
				\arrow["{1 \otimes f_j}", from=2-1, to=2-2]
				\arrow["{\tilde{\pi}_{j - 1}}"'{pos=0.44}, curve={height=6pt}, two heads, from=2-2, to=1-2]
			\end{tikzcd}\]
		follows from the commutativity of the respective squares of $\Ga$-modules before tensoring with $L$.
	\end{proof}	
	
	We close this subsection with an application of \cref{prop:cleft.1}{} for the special subclass of $\La$-modules admitting a $\Ga$-submodule with tight $\La $-complement.
	
	\begin{cor}[cf.\!\mbox{\cite[Corollary 4.6]{Beli2}{}}]
		Let $ L_\Ga $ and $ {}_\Ga N$ be any modules, and let $Z$ be any submodule of $\La \otimes_\Ga N$ contained in $ { \Ker \pi } \otimes_\Ga N $. Then the group $\Tor_i^\Ga ( L , N ) $ is a direct summand of $\Tor_i^\La ( L ,  { \La \otimes_\Ga N } / Z ) $ for every $ i \geq 0 $. 
	\end{cor}
	
	\begin{proof}
		Direct consequence of \cref{prop:cleft.1}{} and \cref{rem:1}{}.
	\end{proof}

	\subsection{A novel reduction for the finitistic dimensions of artinian rings}		\label{subsec:rad.pres.clefts.3}

	Understanding the homological dimensions of rings and their quotients is a central theme in homological algebra, offering profound insights into ring structure and module theory. This subsection develops a powerful reduction result that relates the finitistic (and global) dimensions of a left artinian ring $ \La $ to those of a split quotient $ \La / K $ under specific conditions. As in the previous subsection, the reader may assume that $ \La $ is an Artin algebra throughout, or even a finite dimensional algebra, for the sake of simplicity.

	\begin{thm}
		\label{thm:main.III}
		Let $K $ be an ideal of a left artinian ring $\La $ contained in $J (\La )$ such that the natural epimorphism $\La \epic \La / K$ splits. Assume furthermore that $\pd K_\La < \infty $ and at least one of $\pd_\La K < \infty $ and $K ^2 = 0 $ holds. Then
		\[
		\fpd \La < \infty \, \, \iff \, \, \fpd { \La / K } < \infty 
		\]
		and the equivalence remains valid if $\fpd $ is replaced by $\Fpd $ or $\gd$.
	\end{thm}

	\cref{thm:main.III}{} is a direct application of the following proposition. Indeed, a ring is left noetherian left perfect if and only if it is left artinian; see \cite[Lemma~2.4]{Giata1}{}.

	\begin{prop}
		\label{prop:main.III}
		Let $K $ be an ideal of a ring $\La $ contained in $J (\La )$ such that the natural epimorphism $\La \epic { \La / K } $ splits. Assume that $\La $ is left noetherian semiperfect.
		\begin{enumerate}[\rm(i)]
			\item If the projective dimension of $ K $ as a left $\La $-module is finite, then
			\[
			\fpd \La / K  \, \leq \,  \fpd \La  \, \leq \,  \fpd \La / K + \pd K_\La + 1.
			\]
			\item If the square of $ K $ is zero, then
			\[
			\fpd \La / K  - \pd K_\La  \, \leq \,  \fpd \La  \, \leq \,  \fpd \La / K + \pd K_\La + 1.
			\]
		\end{enumerate}
		Furthermore, the above inequalities are valid if $\fpd $ is replaced by $\Fpd $ and $\La $ is left perfect, and by $\gd $ in both cases.
	\end{prop}

	We postpone the proof of \cref{prop:main.III}{} to the end of this section, as it rests on the next two lemmas.

	\begin{lem}[cf.\!\mbox{\cite[Corollary 1.6]{SteffenDiracca}{}}]
		\label{cor:cleft.1}
		Let $(\La , \Ga , \pi , \i)$ be a radical-preserving ring cleft extension with superfluous kernel. Then $ \gd \Ga  \, \leq \,  \gd \La $ if $\La $ is left perfect or left noetherian semiperfect.
		If we furthermore assume that $ \pd _\La { \Ker \pi } < \infty $, then
		\begin{enumerate}[\rm(i)]
			\item $\fpd \Ga \leq \fpd \La$ if $\La$ is left noetherian semiperfect, and
			\item $\Fpd \Ga \leq \Fpd \La$ if $\La$ is left perfect.
		\end{enumerate}
	\end{lem}

	\begin{proof}
		The first part of the lemma concerning the global dimension follows immediately from \cref{cor:cleft.2}{}. For the case when $ \La $ is left noetherian semiperfect we also use the fact that the global dimension of a ring may be attained on finitely generated modules. Now let $\La $ be left perfect and consider a $ \Ga $-module $ N$. Denote by $\mathbb{P }$ a minimal projective resolution of $ N $, which is available because $\Ga $ is also left perfect as a quotient ring of $\La $. The exact complex ${}_\La \mathbb{P}$ implies that $\pd _\La N \leq \pd _\Ga N + \pd _\La \Ga$ by a well-known homological argument. If $\pd {}_\La \Ker \pi < \infty$, then $\pd {}_\La \Ga = \pd {}_\La \Ker \pi + 1 < \infty$ as the kernel of $\pi $ is contained in $ \rad_\La \La = J( \La ) $; in particular, we have that $ \pi $ is a projective cover of $ \Ga $ viewed as a left $ \La $-module. Therefore, if the projective dimension of ${}_\Ga N $ is finite, then the same holds for $ {}_\La N $. \cref{cor:cleft.2}{} implies now that $ \pd _\Ga N \leq \pd _\La N \leq \gd \La  $ and inequality (ii) follows readily. The proof of inequality (i) is analogous, modulo the observation that $ N $ is finitely generated over $ \La $ if the same holds over $ \Ga $.
	\end{proof}

	\begin{lem}
		\label{cor:cleft.3}
		Let $K $ be an ideal of a ring $\La $ contained in $J (\La )$ such that the natural epimorphism $\pi \colon \La \epic { { \La / K } = \Ga } $ splits. If $\La $ is left noetherian semiperfect, then
		\[
		\fpd \Ga - \fd K_\Ga \leq  \fpd \La  \leq  \fpd \Ga + \fd K_\La + 1 
		\]
		where $K$ is considered as a right $\Ga$-module through restriction of scalars along a section monomorphism of $\pi $.
		
		Furthermore, if $\La $ is left perfect then the analogous inequalities hold for the big finitistic dimensions of the rings. In both cases, the analogous inequalities hold for the global dimensions of the rings.
	\end{lem}
	
	\begin{proof}
		Observe that $ \pi $ is radical-preserving as it is surjective, and its kernel is superfluous by assumption. Let $ \i \colon \Ga \monicc \La $ be a section ring monomorphism of $\pi $, and observe that $ \i $ is also radical-preserving (with superfluous kernel) according to \cref{lem:cleft.1}{}. See also \cref{defn:1}{} and the following discussion. A direct application of \cite[Theorem~3.11]{Giata1}{} for $ \i $ and $ \pi $ yields now the inequalities
			\[
				\fpd \Ga \leq \fpd \La + \fd \La_\Ga \, \, \, \textrm{ and } \, \, \, \fpd \La \leq \fpd \Ga + \fd \Ga_\La
			\]
		respectively, assuming that $\La $ is left noetherian semiperfect.
		It remains to note that $\fd \Ga_\La \leq \fd K_\La +1 $ as the sequence $0 \to K \to \La \to \Ga \to 0 $ is a short exact sequence of right $\La$-modules, and $\fd K_\Ga = \fd \La_\Ga $ due to the direct sum of right $ \Ga$-modules $\La_\Ga \simeq \Ga_\Ga \oplus K_\Ga $.
		
		The proof for the other two homological dimensions also follows from \cite[Theorem~3.11]{Giata1}{} in a similar way.
	\end{proof}

	We are now ready to prove \cref{prop:main.III}{}.

	\begin{proof}[Proof of \cref{prop:main.III}{}]
		In both cases, we have that $ \pi $ is a radical-preserving homomorphism with superfluous kernel. Therefore, the right-hand side inequalities follow from \cite[Theorem~3.11]{Giata1}{} applied to $ \pi $ for all three homological dimensions; see also the proof of \cref{cor:cleft.3}{}.
		
		As for the left-hand side inequalities, observe that they follow from \cref{cor:cleft.1}{} in the first case, for all three homological dimensions.
		In the second case, note that $ K^2 = 0 $ implies $ K $ may be seen as a right $ \Ga $-submodule of itself admitting the zero submodule as an absorbing $\La $-complement; see also \cref{defn:cleft.1}{}. In particular, we have that $ \pd K_\Ga \leq \pd K_\La $ according to \cref{cor:cleft.2}{}, since the $ \La $-structure of $ ( K _\Ga )_\La $ is the same as the initial structure that $ K $ inherits as a (right) ideal of $ \La $. Applying \cref{cor:cleft.3}{}, we deduce that the left-hand side inequalities hold in the second case for all three homological dimensions.
	\end{proof}

	\medskip
	
	\section{Generalized arrow removal operation}	\label{sec:gen.ar.rem}

	\smallskip

	This section establishes conditions on a set of arrows $A$ in a bound quiver algebra $\La = k Q / I$ ensuring that
	\begin{equation*}
		\fpd \La < \infty  \,\, \iff \,\,  \fpd { \La / \la A + I \ra } < \infty
	\end{equation*}
	where $\la A + I \ra $ denotes the ideal of $\La $ generated by the arrows in $A$. The novelty of our reduction technique is illustrated with concrete example algebras (\cref{subsec:3.3}{}) that are highly non-trivial in the sense of \cref{defn:irredu}{}.
	
	\subsection{Main result}		\label{subsec:3.1}
	We begin by characterizing when a set of arrows $A$ is such that the natural epimorphism $\La \epic { \La / \la A + I \ra } $ admits a section algebra monomorphism with image the subalgebra of $\La $ generated by all trivial paths and all arrows in $Q_1 \setminus A$. \hyperlink{lem:ar.rem.2}{\cref{lem:ar.rem.2}{}} generalizes previous partial results such as \cite[Proposition 5.4]{SteffenDiracca}{} and the third part of \cite[Proposition 4.4]{arrowrem1}{}. It may also be found in \cite[Subsection~2.4]{ACT}{} in a different but equivalent form.
	
	Before proceeding with the lemma, we fix some notation about bound quiver algebras that will be used throughout the rest of the paper. For a comprehensive source of the standard notation and well-established facts about bound quiver algebras we refer the reader to \cite{ASS}{}, particularly Chapters I to III.
	
	Throughout, we let $ Q $ denote a finite quiver where $ A $ is a set of arrows, and $ k $ denotes an arbitrary field. We denote by $\BQ $ the set of all paths in $Q$, which is a $k $-basis of the path algebra $ k Q $. As usual, a \emph{relation} is a linear combination of paths of length at least two, where all paths have the same source and the same target.
	
	A path $ p $ \emph{occurs} in a non-zero element $ \sum_{q \in \BQ } \l_q \cdot q \in k Q $ if the coefficient $ \l_{p} \in k $ is non-zero.
	Furthermore, we say that $ p $ \emph{passes through} an arrow $ \a $ if it is divided by it, i.e.\ if $ p = p_1 \a p_2 $ for subpaths $ p_1 $ and $ p_2 $; otherwise, we say that $ p $ \emph{avoids} $ \a $. An element of $ k Q $ avoids an arrow if all paths occurring in the element do so. 
	Extending the terminology to subsets $ T \subseteq k Q $, we say that $ p $ occurs in $ T $ if it occurs in some element of $ T $, and $ T $ avoids $ \a $ if all elements of $ T $ avoid that arrow.

	We denote by $ \BQA $ the set of paths in $Q$ passing at least once through some arrow in $A$, and $\BQnotA $ denotes the set of paths avoiding all arrows in $ A $. Note that $\BQ = \BQA \dot \cup \BQnotA $.
	For every element $z \in k Q$, we use $ z_A$ and $z_{\nott \! A}$ to denote the unique elements in the subspaces $ {}_k \la \BQA \ra$ and ${}_k \la \BQnotA \ra$, respectively, such that $z = z_A + z_{\nott \! A}$. Moreover, we extend the notation for arbitrary subsets $ T \subseteq k Q$, that is $ T_A = \{ z_A  \colon  z \in T \}$ and $ T_{\nott \! A} $ is defined analogously. Furthermore, if $ A $ comprises a single arrow, we replace $ A $ with the arrow itself in the above symbols.

	\begin{lem}[\mbox{cf.\ \cite[Subsection~2.4]{ACT}{}}]		\label{lem:ar.rem.2}	\phantomsection	\hypertarget{lem:ar.rem.2}
		Let $\La = k Q / I$ be a bound quiver algebra, where $ A $ is a set of arrows. The following are equivalent.
		\begin{enumerate}[\rm(i)]
			\item The natural epimorphism $\pi \colon  \La \epic \La / \la A + I \ra $ admits a section algebra monomorphism with image the subalgebra of $\La $ generated by all trivial paths and all arrows in $Q_1 \setminus A$.
			\item The sum of $k $-vector spaces $\La = \Ga' \oplus \la A+ I \ra$ is direct.
			\item The subspaces $I_A $ and $ I_{\nott \! A} $ are contained in $ I $.
			\item The sum of $k$-vector spaces $I = (I \cap {}_k \la \BQA \ra) \oplus (I \cap {}_k \la\BQnotA \ra)$ is direct.
			\item The ideal $I$ is generated by a finite set of relations $S$ such that $S = S_A \dot \cup S_{\nott \! A} $.
		\end{enumerate}
	\end{lem}
	
	\begin{proof}
		Let $\Ga' $ denote the subalgebra of $\La $ generated by all trivial paths and all arrows in $Q_1 \setminus A$.
		Then $\Ga' = {}_k \la \BQnotA + I \ra $ and $ \la A+ I \ra = {}_k \la \BQA + I \ra $, implying that $\La = \Ga' + \la A+ I \ra$.
		
		(i) $\Rightarrow$ (ii): It follows from the fact that $\Ker \pi = \la A + I \ra $.
		
		(ii) $\Rightarrow$ (i): We assume that the sum of $k$-vector spaces $\La = \Ga' \oplus \la A+ I \ra$ is direct. Then for every $W \in \La / \la A+ I \ra $ there is a unique representative $w \in \Ga' $. It is easy to verify that $W \mapsto w$ defines an algebra monomorphism $ \i \colon { \La / \la A+ I \ra } \monicc \La $ with image equal to $ \Ga' $ and such that $\pi \i = \id_{\La / \la A+ I \ra} $.
		
		(ii) $\Rightarrow$ (iii): For every $z \in I$, we have that $ (z_A + I) + (z_{\nott \! A} + I) = I$. Since $z_A + I \in \la A + I \ra$ and $ z_{\nott \! A} + I \in \Ga'$, and the sum $ \La =  \Ga' \oplus \la A + I \ra $ is direct, it follows that $z_A , z_{\nott \! A} \in I$.
		
		(iii) $\Rightarrow$ (ii): If $Z \in \Ga' \cap \la A+I \ra$, then $Z = ( \sum_p \lambda_p \cdot p ) +I = ( \sum_q \mu_q \cdot q ) +I$ where finitely many coefficients $\lambda_p, \mu_q \in k$ are non-zero, and the paths $p$, $q$ range over $\BQA$ and $ \BQnotA$ respectively. We have that $ w =  \sum_p \lambda_p \cdot p  -   \sum_q \mu_q \cdot q \in I$ implying that $\sum_p \lambda_p \cdot p = w_A $ and $\sum_q \mu_q \cdot q = w_{\nott \! A} $ are in $ I$. In particular, we have $ Z = w_A + I = I $.
		
		(iii) $\Rightarrow$ (iv): In general, the sum on the right-hand side of the equality in (iv) is direct and contained in $I$. Condition (iii) ensures that the two subspaces are actually equal.
		
		(iv) $\Rightarrow$ (iii): Let $z , z_1, z_2 \in I$, where $z_1  $, $z_2 $ are the unique elements in $I \cap {}_k \la \BQA \ra $ and $ I \cap {}_k \la \BQnotA \ra$, respectively, such that $z = z_1 + z_2 $. Then $z_1 = z_A$ and $z_2 = z_{ \nott \! A }$, implying that $z_A, z_{ \nott \! A } \in I$.
		
		(iii) $\Rightarrow$ (v): Let $ T $ be a finite generating set of relations for $I$. Then $ T_A \cup T_{\nott \! A } $ is contained in $I$ and, therefore, it is a generating set for $I$ with the desired property.
		
		(v) $\Rightarrow$ (iii): Every element $ z \in I$ can be written as a finite sum of the form
		\[
		z = \sum_j \l_j \cdot r_1^j z_j r_2^j
		\]
		for coefficients $\l_j \in k$, where the elements $z_j$ are in the special generating set $ T $, and $r_1^j $, $r_2^j $ are paths such that $t( r_1^j ) = s( z_j ) $ and $s( r_2^j ) = t( z_j ) $. Observe that either $r_1^j z_j r_2^j \in   {}_k \la \BQnotA \ra$, when $z_j = {z_j}_{\nott \! A}$ and $r_1^j$, $r_2^j $ avoid the arrows in $A$, or else $r_1^j z_j r_2^j \in {}_k \la \BQA \ra $. In particular, we have $z_A , z_{ \nott \! A} \in I $ as $r_1^j z_j r_2^j \in I$ for all $j$.
	\end{proof}

	\begin{defn}
		\label{defn:ar.rem.6}
		Let $\La = k Q / I$ be a bound quiver algebra. A set of arrows $A$ is \emph{pre-removable} if it satisfies any of the equivalent conditions of \hyperlink{lem:ar.rem.2}{\cref{lem:ar.rem.2}{}}{}.
		In particular, the set $ A $ is pre-removable if it possesses a generating set $ T $ of relations such that $ T = T_A \cup T_{\nott \! A } $.
	\end{defn}
	
	\begin{rem}
		In the context of bound quiver algebras, the concept of pre-removability is actually equivalent to the concept of split algebra quotients over superfluous ideals, i.e.\ ideals contained in the Jacobson radical; see \cref{prop:ar.rem.3}{} for more.
	\end{rem}

	\begin{rem}
		\label{rem:ar.rem.2}
		If a set of arrows $A$ is pre-removable in $\La = k Q / I$, then the natural epimorphism induced by the ideal $ \la A + I \ra $ admits a section algebra monomorphism $ \i \colon  \La / \la A + I \ra \monicc \La $, where $ \i ( \ol{ p + I } ) = p + I $ for every path $p$ in $Q$ avoiding the arrows of $A$. See the proof of \hyperlink{lem:ar.rem.2}{\cref{lem:ar.rem.2}{}} for details.
	\end{rem}

	\begin{exam}		\label{exam:ar.rem.-1}		\hypertarget{exam:ar.rem.-1}
		For every bound quiver algebra $\La = k Q / I $ the set of all arrows is pre-removable, which is another way to state that the Wedderburn-Malcev Theorem holds for $\La $. The section algebra monomorphism $ \i \colon \La / J ( \La ) \monicc \La $ of \cref{rem:ar.rem.2}{} is given by the correspondence $\ol{ e_i } \mapsto e_i $ for every vertex $i \in Q_0$. 
	\end{exam}

	The following corollary will be employed for showing that the classical arrow removal operation \cite{arrowrem1}{} is a special instance of \cref{thm:GAR}{}; see also \cref{exam:1}{}.
	
	\begin{cor}
		\label{cor:ar.rem.5}
		Let $I$ be an admissible ideal of the path algebra $k Q$ and assume that there is a generating set $ T $ for $I$ avoiding all arrows in some set $A \subseteq Q_1$. Then for every generating set $ T' $ for $I$, the set $ T'_{\nott \! A} $ is also generating.
	\end{cor}
	
	\begin{proof}
		The fact that $z = z_{\nott \! A} $ for every $ z \in T $ implies that the set $ T'_A \cup T'_{\nott \! A } $ is contained in $ I $ according to \hyperlink{lem:ar.rem.2}{\cref{lem:ar.rem.2}{}}; in particular, it is generating. But the ideal of $k Q $ generated by $ T'_{ \nott \! A } $ contains  $ I \cap {}_k \la \BQnotA \ra $ and thus $ T $. We deduce that $ T'_{ \nott \! A } $ is generating.
	\end{proof}
	
	The main theorem of this subsection concerns removable sets of arrows defined as follows.
	
	\begin{defn}
		\label{defn:ar.rem.1}
		Let $\La  = k Q / I$ be a bound quiver algebra and $A $ a pre-removable set of arrows. Then $A $ is:
		\begin{enumerate}[\rm(i)]
			\item \emph{two-sided removable} if the ideal $\la A + I \ra $ has finite projective dimension both as a right and left $\La$-module;
			\item \emph{non-repetitive} if ${\la A + I \ra} ^2 = 0$;
			\item \emph{only left removable} if $A$ is non-repetitive and the ideal $\la A + I \ra $ has finite projective dimension only as a right $\La$-module;
			\item \emph{(left) removable} if $A$ is two-sided or only left removable.
		\end{enumerate}
		Moreover, an algebra $\Ga $ is a \emph{generalized arrow removal of $\La $} if it is isomorphic to the quotient algebra of $ \La  $ induced by a removable set of arrows.
	\end{defn}

	We make a few comments on \cref{defn:ar.rem.1}{} before proceeding with the main result of the subsection. First, we have allowed $A$ to be the empty set
	as this will simplify the exposition of \cref{sec:ar.red.ver}{}. Furthermore, the terms of \cref{defn:ar.rem.1}{} will also be used for arrows instead of the respective singleton sets. Last, we say that a set of arrows $A \subseteq Q_1 $ is, for instance, only left removable \emph{in} $\La $ when we want to make clear the specific bound quiver algebra of $Q $ we are working with.

	\begin{thm}
		\label{thm:GAR}
		Let $A $ be a removable set of arrows in $\La = k Q / I$. Then
		\[
		\fpd \La < \infty  \,\, \iff \,\,  \fpd { \La / \la A + I \ra } <\infty
		\]
		and the equivalence remains valid if $\fpd $ is replaced by $\Fpd $ or $\gd$.
	\end{thm}
	
	\begin{proof}
		The theorem is a direct consequence of \cref{defn:ar.rem.1}{} and \cref{thm:main.III}{}. Indeed, the algebra $\La $ is a noetherian perfect ring and the natural epimorphism $\La \epic \La / \la A + I \ra $ splits according to \hyperlink{lem:ar.rem.2}{\cref{lem:ar.rem.2}{}}.
	\end{proof}

	\begin{exam}
		\label{exam:ar.rem.0}
		For any bound quiver algebra $\La = k Q / I $, it holds that $Q_1 $ (the set of all arrows) is (two-sided) removable in $\La $ if and only if the global dimension of $\La $ is finite. Indeed, if $Q_1$ is removable in $\La $ then \cref{thm:GAR}{} implies that the global dimension of $ \La $ is finite as the same holds for the semisimple algebra $\La / J( \La )$. Conversely, if the global dimension of $ \La $ is finite then the ideal $J (\La )$ has finite projective dimension as a left and right $\La$-module. Therefore, the set $Q_1$ is two-sided removable as it is always pre-removable (\hyperlink{exam:ar.rem.-1}{\cref{exam:ar.rem.-1}{}}).
	\end{exam}

	\begin{exam}
			\label{exam:1}
		Let $ A $ be a set of arrows in a bound quiver algebra $\La = k Q / I $ such that the quotient algebra $\La / \la A + I \ra $ is an arrow removal of $\La $ in the sense of \cite{arrowrem1}{}. Then $ A $ is two-sided removable according to \cref{prop:ar.rem.1}{}, since a projective $ \La $-bimodule is also projective as a left and right $ \La $-module. We call the set of arrows \emph{$0$-biremovable} in this case, since the projective dimension of the generated ideal as a $ \La $-bimodule is zero. Similarly, an arrow will be called $0$-biremovable if the respective singleton set is $0$-biremovable.
		
		Note that an arrow is $0$-biremovable if and only if there is a generating set for $I$ avoiding the arrow according to \cite[Proposition~4.5]{arrowrem1}{}; see also \cref{cor:ar.rem.1}{}. In particular, every arrow in a $0$-biremovable set is itself $0$-biremovable. 
	\end{exam}

	\subsection{Non-triviality conditions}		\label{subsec:3.2}
	Before introducing concrete examples of arrow reductions via \cref{thm:GAR}{}, we consider the next conditions. If a bound quiver algebra $\La = k Q / I $ does not satisfy one or more of them, then the finiteness of its (big) finitistic dimension is either immediate or follows from the indicated paper.
	\begin{enumerate}[\rm(i)]
		\item $\La $ has non-zero finitistic dimensions.
		\item $\La$ is not monomial (\!\!\cite{GKK}{}). 
		\item The Loewy length of $\La$ is greater than $ 3 $ (\!\!\cite{GH}{}).
		\item Every arrow of $ Q $ occurs in every generating set for $I$ (\!\!\cite{arrowrem1}{}).
		\item $\La$ is triangular reduced (\!\!\cite{FGR}{}).
		\item The projective dimension of every simple $\La$-module is greater than $1$ (\!\!\cite{FS}{}).
		\item The injective dimension of every simple $\La$-module is infinite (\!\!\cite{arrowrem1}{}, depending on the little finitistic dimension of a specific Peirce corner algebra of $\La $).
	\end{enumerate}
	
		\begin{defn}
				\label{defn:irredu}
			We say that a bound quiver algebra is \emph{irreducible} if it satisfies conditions (i) to (vii).
		\end{defn}
	
	Recall that a bound quiver algebra $ \La $ is called \emph{triangular reduced} \cite{arrowrem1}{} if, for every idempotent $e \neq 0 , 1$ of $\La $, both $e \La ( 1 - e ) $ and $(1 - e ) \La e $ are non-trivial subspaces of $\La$. It can be easily verified that $\La $ is triangular reduced if and only if
	it is not isomorphic to a triangular matrix algebra
	$ \big(\begin{smallmatrix}
		A & M \\
		0 & B
	\end{smallmatrix}\big)$,
	where $A$ and $B$ are bound quiver algebras and $M$ is an $A$-$B$-bimodule such that $k$ acts centrally making it a finite-dimensional vector space.
	
	The only argument that needs perhaps some attention is that the algebra $ e \La e $ is isomorphic to a bound quiver algebra for every non-zero idempotent $e \in \La $. Indeed, it holds that $e$ is always the sum of $m$ primitive orthogonal idempotents for some integer $m > 0 $. Consequently, the algebra $e \La e / e J( \La ) e $ is isomorphic to the direct product of $m$ copies of $k$ and the desired property follows from \cite[Proposition~I.6.2, Corollary~I.1.4]{ASS}{}.
	
	The above observation leads to the characterization of triangular reduced bound quiver algebras as follows. One direction can be found in \cite[Proposition~5.3]{arrowrem1}{}, implicitly proven for the first time in \cite{GM}{}. Recall that a quiver $ Q $ is \emph{strongly connected} if for every pair of vertices $v, v' \in Q_0 $ there is a path with source $v $ and target $v'$.

	\begin{lem}
		\label{lem:ar.rem.9}
		A bound quiver algebra is triangular reduced if and only if the underlying quiver is strongly connected.
	\end{lem}
	
	\begin{proof}
		Let $\La = k Q / I $ be triangular reduced and assume that there are (distinct) vertices $i, j \in Q_0$ such that there is no path in $ Q$ with source $i$ and target $j$. Define $Q_0^j$ to be the set of vertices $ v \in Q_0 $ such that there is a path with source $ v $ and target $ j$ in $Q$, $j$ included, and let $Q_0^i $ be the complement $ Q_0 \setminus Q_0^j$. Note that the idempotent $e^i = \sum_{ v \in Q_0^i} e_v $ is non-trivial as neither of $Q_0^i$ and $Q_0^j$ is empty. Therefore, the subspace $ e^i \La ( 1 - e^i ) $ is non-zero due to $ \La $ being triangular reduced. In particular, there is a path $p $ with source  $ v_1 \in Q_0^i$ and target $ v_2 \in Q_0^j$. Since $ v_2 \in Q_0^j$, there is a path $ q $ with source $ v_2 $ and target $ j $, and the path $ pq $ implies that $v_1 \in Q_0^j$, a contradiction.
		
		If $\La $ is not triangular reduced, then it is isomorphic to a triangular matrix algebra as in \cref{exam:ar.rem.6}{}. Hence, the quiver $Q$ is not strongly connected.
	\end{proof}

	\subsection{Examples}		\label{subsec:3.3}
	We hope the reader will be persuaded by the following examples that our techniques effectively yield a novel tool for reducing the validity of the $(\mathsf{FDC})$ for concrete algebras to the validity of the $(\mathsf{FDC})$ for algebras of smaller dimension. All examples are highly non-trivial in the sense of \cref{defn:irredu}{}.
	
	For a bound quiver algebra $ \La = k Q / I $ and a vertex $ i $, we write $ S_\La ( { i } ) $ and $ S_{\La^{\! \mathrm{op}}} ( { i } ) $ to denote the simple left and right $ \La $-module corresponding to that vertex.
	

	\begin{exam}
		\label{exam:ar.rem.1}
		
		Let $Q$ be the quiver below, and let $k$ be any field. We consider the bound quiver algebra $ \La_1 = k Q / I_1 $ where $I_1 $ is the ideal generated by the relations $R_1 = \{  \a \e - \d \a   , \,  \d ^2   , \,  \e ^2   , \,  \z ^2   , \,  \b \g   ,   \z \g  , \,  \g \a \} $.

		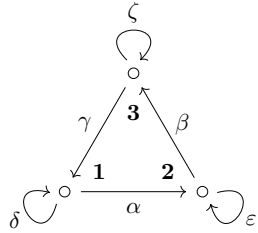
\begin{figure}[hbt!]
			\vspace*{-0.15cm}
			\centering
			\resizebox{11em}{!}{
				\begin{tikzpicture}
					
					
					\draw  ($(0,0)+(90:1.2)$) circle (.08);
					\node at ($(0,0)+(90:0.6)$) {$\mathbf{3}$};
					
					\draw  ($(0,0)+(210:1.2)$) circle (.08);
					\node at ($(0,0)+(210:0.6)$) {$\mathbf{1}$};
					
					\draw  ($(0,0)+(330:1.2)$) circle (.08);
					\node at ($(0,0)+(330:0.6)$) {$\mathbf{2}$};

					\draw[->,shorten <=7pt, shorten >=7pt] ($(0,0)+(210:1.2)$) -- ($(0,0)+(330:1.2)$);
					
					\draw[->,shorten <=7pt, shorten >=7pt] ($(0,0)+(330:1.2)$) -- ($(0,0)+(90:1.2)$) ;
					
					\draw[->,shorten <=7pt, shorten >=7pt] ($(0,0)+(90:1.2)$)  -- ($(0,0)+(210:1.2)$);

					\node at ($(0,0)+(150:0.85)$) {$\gamma$};
					
					\node at ($(0,0)+(270:0.85)$) {$\alpha$};
					
					\node at ($(0,0)+(30:0.85)$) {$\beta$};

					\draw[->,shorten <=4pt, shorten >=4pt] ($(0,0)+(210:1.3)$).. controls +(210+45:1) and +(210-45:1) .. ($(0,0)+(210:1.3)$);
					\node at ($(0,0)+(210:2.07)$) {$\delta$};
					
					\draw[->,shorten <=4pt, shorten >=4pt] ($(0,0)+(330:1.3)$).. controls +(330+45:1) and +(330-45:1) .. ($(0,0)+(330:1.3)$);
					\node at ($(0,0)+(330:2.05)$) {$\varepsilon$};
					
					\draw[->,shorten <=4pt, shorten >=4pt] ($(0,0)+(90:1.3)$).. controls +(90+45:1) and +(90-45:1) .. ($(0,0)+(90:1.3)$);
					\node at ($(0,0)+(90:2.13)$) {$\zeta$};
					
				\end{tikzpicture}
			}
			\caption{Quiver $Q$}
			\label{fig1}
			\vspace*{-0.35cm}
		\end{figure}

		The arrow $\b$ is two-sided removable and non-repetitive in $\La_1$. Indeed, it holds that $\pd {\la \b + I_1 \ra} _{ \La_1 } = 2 $ as there is a minimal projective resolution of the form
		\[
		0 \to (e_2 \La_1 )^{ 4 } \to (e_1 \La_1 )^{ 4} \to ( e_3 \La_1 )^{ 4} \to { \la \b + I_1 \ra} \to 0
		\]
		and $  \la \b + I_1 \ra  \simeq (  \La_1  e_2 ) ^{ 2}$ is projective as a left $\La_1$-module.
		To see that $\b $ is non-repetitive note that $\b \g  $ and $\b \z \g $ are zero in $\La_1 $.
		Similarly, one can show that $\a $ and $\g $ are also removable in $\La_1 $.
		
		The algebra $\La_1 $ is irreducible in the sense of \cref{defn:irredu}{}. The last three conditions follow immediately from the shape of $Q$; see \cref{lem:ar.rem.9}{} and \cite{ILP}{}. The Loewy length of $\La_1 $ is $5$, where $\a \e \b \z $ is the longest non-zero path, and $\La_1$ is not monomial as the paths $\a \e $ and $ \d \a$ are non-zero and equal in $\La_1$. Furthermore, the finitistic dimensions of $\La_1 $ are non-zero as there is no right ideal of $\La_1$ isomorphic to $S_{\La_1^{\! \mathrm{op}}} ( { 2 } )$; see \cite[Lemma~6.2]{Bass}{}. Finally, one can readily verify that all paths occurring in $R_1$ occur in every generating set for $I_1$. In particular, the same holds for all arrows of $Q$ implying that no non-empty set of arrows is $0$-biremovable in $\La _1 $.
		
		We also note that the algebra $\La_1 $ is not Iwanaga-Gorenstein due to \hyperlink{lem:Gor}{\cref{lem:Gor}{}} applied for the loop $\z $ and the arrow $\g$.
	\end{exam}

	\begin{exam}
		\label{exam:ar.rem.2}
		
		Let $Q$ be the quiver of \cref{exam:ar.rem.1}{}, and let $k $ be any field. We consider the bound quiver algebra $ \La_2 = k Q / I_2 $ where $I_2$ is the ideal generated by the relations $R_2 = \{  \a \e - \d \a  , \,  \d ^2  , \,  \e ^2 , \,  \z ^2 , \, \b \g \a   , \,  \b \z \} $.
		
		The arrow $\g$ is two-sided removable and repetitive in $\La_2$. Indeed, it holds that $\pd _{ \La_2 }  \la \g + I_2 \ra = 1 = \pd { \la \g + I_2 \ra }_{ \La_2 }$ as there are minimal projective resolutions
		\[
		0 \to ( e_2 \La_2 )^{  4 } \to ( e_1 \La_2 )^{  6 } \to  \la \g + I_2 \ra  \to 0
		\]
			\begin{center}
			and
		\end{center}
		\[
		0 \to ( \La_2 e_2 )^{  4 } \to ( \La_2 e_3 )^{  6 } \to   \la \g + I_2 \ra  \to 0  .
		\]
		Furthermore, the path $ \g \a \e \b \g $ is non-zero. It can also be shown that $ \la \g + I_2 \ra $ has projective dimension $ 1 $ as a $ \La_2 $-bimodule due to the minimal projective resolution of the form
		\[
		0 \to \La_2 e_2 \otimes_k e_2 \La_2 \to \La_2 e_3 \otimes_k e_1 \La_2 \to  \la \g + I_2 \ra  \to 0
		\]
		where $e_3 \otimes e_1 $ maps to $\g + I_2 $ and $e_2 \otimes e_2 $ to $ ( \b + I_2 ) \otimes ( \a + I_2 ) $. Therefore, we could say that $\g $ is $1$-biremovable following the terminology introduced in \cref{exam:1}{}.
		
		Similarly, one can show that $\a $ is also removable in $\La_2$, whereas $\b $ is not.
		
		Last, it can be verified that the algebra $\La_2$ is irreducible as in \cref{exam:ar.rem.1}{}. We only mention that
		the Loewy length of $\La_2$ is $8$, where $\z \g \a \e \b \g \d $ is the longest non-zero path, and the finitistic dimensions of $\La_2 $ are non-zero as there is no right ideal of $\La_2$ isomorphic to $S_{\La_2^{\! \mathrm{op}}} ( {2} )$.
	\end{exam}

	\begin{exam}
		\label{exam:ar.rem.3}
		
		Let $Q$ be the quiver of \cref{exam:ar.rem.1}{}, and let $k$ be any field. We consider the bound quiver algebra $ \La_3 = k Q / I_3 $ where $I_3 $ is the ideal generated by $R_3 = \{  \a \e - \d \a   , \,  \d ^2   , \,  \e ^2   , \,  \z ^2   ,   \e \b   , \,  \b \g   , \,  \b \z \g \a \} $.
		
		The arrow $\b $ is only left removable in $\La_3$. Indeed, it holds that $\pd {\la \b + I_3 \ra}_{ \La_3 } = 1$ as there is a minimal projective resolution of the form
		\[
		0 \to ( e_1 \La_3 )^{ 4} \oplus ( e_2 \La_3 )^{ 4} \to ( e_3 \La_3 )^{ 4} \to \la \b + I_3 \ra\to 0 .
		\]
		Furthermore, the ideal $  \la \b + I_3 \ra $ has infinite projective dimension as a left $\La_3$-module since it is isomorphic to the direct sum of $4$ copies of $X = \La_3 ( \b + I_3 ) $, which is an $\Omega$-$1$-periodic module. The arrow $\b $ is non-repetitive as the paths $\b \g  $, $\b \z \g \a $ and $\b \z \g \d \a = \b \z \g \a \e $ are zero in $\La_3$.
		
		Similarly, one can show that $\a $ is also removable in $\La_3 $, whereas $\g $ is not.
		
		Again, the algebra $\La_3$ is irreducible; see also \cref{exam:ar.rem.1,exam:ar.rem.2}{}. We only mention that
		the Loewy length of $\La_3$ is $8$, where $\z \g \a \b \z \g \d $ is the longest non-zero path, and the finitistic dimensions of $\La_3 $ are non-zero as there is no right ideal of $\La_3$ isomorphic to $S_{\La_3^{\! \mathrm{op}}} ( {3} )$.
	\end{exam}

	All three of the above algebras have finite (big) finitistic dimension. We prove this fact by computing their arrow reduced version, introduced in the next section.

	\medskip
	
	\section{Arrow reduced version of an algebra}	\label{sec:ar.red.ver}
	
	\smallskip
	
	This section introduces two distinct but related methods for simplifying bound quiver algebras by removing arrows -- the arrow reduced version and the arrow irredundant version -- and establishes their fundamental homological properties. The first method is aimed at the preservation of finiteness for the finitistic (and global) dimensions in a context that is as general as possible. The second method restricts the arrows to be removed so that a broader range of homological properties is preserved, such as Iwanaga-Gorensteinness, the finite generation condition for the Hochschild cohomology (\!\!\cite{EHTSS}{}) and the singularity category of Buchweitz.

	\subsection{Arrow reduced version}		\label{subsec:4.1}
	We begin with a technical lemma ensuring that the quotient algebra over an ideal generated by a pre-removable set of arrows possesses a canonical representation as a bound quiver algebra. Recall that $ \BQA $ denotes the set of paths in $ Q $ that pass through at least one arrow in $ A $ at least once, and $ \BQnotA $ denotes the set of all paths avoiding all arrows in $ A $. See also the beginning of \cref{subsec:3.1}{} for an overview of the relevant notation.
	
	\begin{lem}
			\label{lem:ar.red.1}
		Let $A $ be a set of arrows in a bound quiver algebra $\La = k Q / I$, and let $Q'$ be the quiver that results from $Q$ if we remove from it all the arrows in $A$. Then:
			\begin{enumerate}[\rm(i)]
				\item There is a canonical algebra epimorphism $ \phi_A \colon k Q' \epic \La / \la A + I \ra $ defined by $x \mapsto \ol{x + I}  $ for every vertex or arrow $x$ of $Q'$.
				\item The kernel of $ \phi_A $ consists of the elements $z' \in k Q'$ such that $z' + z_A \in I$ for some $z_A \in {}_k \la \BQA \ra $. In particular, it contains the ideal $I' = I \cap {}_k \la \BQnotA \ra $ of $ k Q' $ which is admissible.
				\item The epimorphism $ \phi $ induces an algebra isomorphism $ k Q' / I' \simeq \La / \la A + I \ra $ if and only if $ A $ is pre-removable.
			\end{enumerate}
	\end{lem}

	\begin{proof}
		(i) It is clear that the image of the vertices of $ Q' $ through $ \phi_A $ (or rather of the corresponding trivial paths) is a set of orthogonal idempotent elements whose sum is equal to the unit of the target algebra. Furthermore, for every arrow $ \b \colon i \to j $ of $ Q' $ it holds that $ \phi_A ( \b ) = \phi_A ( e_i ) \phi_A( \b ) \phi_A( e_j ) $. Therefore, the given correspondence defines uniquely an algebra homomorphism $ \phi_A $, see \cite[Theorem~II.1.8]{ASS}{}. The surjectivity of $ \phi_A $ is straightforward as all paths of $ Q $ avoiding the arrows in $ A $ are contained in its image.
		
		(ii) It is clear that $ \phi_A ( z' ) = ( z' + I ) + \la A + I \ra $ for every element $ z' \in k Q ' $, since $ \phi_A $ is a ring homomorphism that respects the $ k $-structures of the algebras. Therefore, we have $ z' \in \Ker \phi_A $ if and only if $ z' + I \in \la A + I \ra $. Equivalently, we have that $ z ' \in \Ker \phi_A $ if and only if there is an element $ z_A \in {}_k \la \BQA \ra $ such that $ z ' - z_A \in I $ since $ \la A + I \ra = {}_k \la \BQA + I \ra $.
		
		For an element $ z' \in I' $, we have that $ z ' \in \Ker \phi_A $ as $ z ' - z_A \in I $ for $ z_A = 0 $. Furthermore, it is clear that all paths occurring in $ I ' $ have length at least two, and all paths of sufficiently large length in $ Q' $ are in $ I ' $, as they are also paths in $ Q $ and therefore in $ I $.
		
		(iii) We have to show that $ \Ker \phi_A = I ' $ if and only if $ A $ is pre-removable. For one direction, assume that $ \Ker \phi_A = I ' $ and $ z \in I $. It suffices to show that $ z_{\nott \! A} \in I $, where $ z_{ \nott \! A } $ denotes here the unique element of $ {}_k \la \BQnotA \ra $ such that $ z = z_{ \nott \! A } + z_A $ for some (unique) element $ z_A \in {}_k \la \BQA \ra $. We have that $ I = z + I = ( z_{ \nott \! A } + I ) + ( z_A + I ) $, implying that $ z_{ \nott \! A } + I = - z_A + I $. Therefore, we have that $ z_{ \nott \! A } \in \Ker \phi_A = I ' \subseteq I $.
		
		Conversely, assume that $ A $ is pre-removable. Since we have already established that $ I ' \subseteq \Ker \phi_A $, it remains to show the opposite inclusion. Let $ z' \in \Ker \phi_A $, and observe that there is an element $ z_A \in {}_k \la \BQA \ra $ such that $ z' + I = z_A + I $ in this case. Therefore, we have $ w = z' - z_A \in I $ implying that $ z' = w_{ \nott \! A } \in I $ according to the decomposition of $ I $ in \hyperlink{lem:ar.rem.2}{\cref{lem:ar.rem.2}{}}{}, as all paths occurring in $ z_A $ pass through some arrow in $ A $. In particular, we have that $ z' \in I ' $.
	\end{proof}
	
	Based on the above lemma, we introduce the notion of canonical representation for a quotient algebra over the ideal generated by a pre-removable set of arrows.
	
	\begin{defn}
			\label{defn:2}
		Let $ A $ be a pre-removable set of arrows in a bound quiver algebra $ \La = k Q / I $. The \emph{canonical representation} of the quotient $\La / { \la A + I \ra } $ as a bound quiver algebra is $ k Q' / I' $, where $ Q ' $ is the quiver that results from $ Q $ if we remove the arrows in $ A $ and $I' = I \cap {}_k \la \BQnotA \ra $, and we often identify the two algebras. Moreover, we say that a set of arrows $B \subseteq Q_1 \setminus A $ satisfies any of the conditions of \cref{defn:ar.rem.1}{} in the quotient algebra $\La / { \la A + I \ra } $, if it does so as an arrow of $ k Q' / I' $.
	\end{defn}

	The next result enables us to define eventually removable sets of arrows, which generalize removable sets.

	\begin{lem}		\label{lem:ar.rem.1}		\hypertarget{lem:ar.rem.1}
		Let $\La = k Q / I$ be a bound quiver algebra and consider two sets of arrows $A , B \subseteq Q_1 $.
		\begin{enumerate}[\rm(i)]
			\item Let $A$ and $B$ be disjoint, and such that $A$ is pre-removable in $\La $ and $B$ is pre-removable in $\La / \la A + I \ra $. Then their union is pre-removable in $\La $. Moreover, the canonical representation of the quotient algebra $\La / \la ( A  \cup B ) + I \ra $ as a bound quiver algebra is the same as the canonical representation of the quotient of $ \La / \la A + I \ra $ over the ideal generated by the arrows in $B$.
			\item Let $A $ and $ B  $ be pre-removable in $\La$. Then their difference $B \setminus A$ is pre-removable in $\La / \la A + I \ra $. Moreover, their union is pre-removable in $\La $.
		\end{enumerate}
	\end{lem}
	
	\begin{proof}
		 We assume for the rest of the proof that $A$ is pre-removable in $\La $. Furthermore, we write $k Q' / I'$ to denote the canonical representation of $\La / \la A + I \ra $ as a bound quiver algebra, that is $Q' = Q_0 \cup ( Q_1 \setminus A ) $ and $I' = I \cap {}_k \la \BQnotA \ra $.
		
		(i) It is straightforward to see that $z_{ \nott ( A  \cup B ) } = ( z _{ \nott \! A} )_{\nott \! B}$ for any element $z \in k Q $, even when $ A $ and $ B $ are not disjoint. Therefore, it suffices to show that $ ( z _{ \nott \! A} )_{\nott \! B} \in I$ for any $z \in I $. The pre-removability of $A$ in $\La $ implies that $z_{\nott \! A } \in I' $ and the pre-removability of $B$ in $\La / \la A + I \ra $ implies that $( z _{ \nott \! A} )_{\nott \! B} \in I' \subseteq I$. Now let $k Q'' / I ''$ denote the canonical representation of the quotient of $\La / \la A + I \ra $ over the ideal generated by the arrows of $B$. Then $Q'' = Q_0 \cup(  ( Q_1 \setminus A ) \setminus B ) = Q_0 \cup ( Q_1 \setminus ( A  \cup B ) ) $ and $I'' = (I \cap {}_k \la \BQnotA \ra ) \cap {}_k \la \BQQnotB \ra = I \cap {}_k \la \BQ^{\nott ( A \cup B )}  \ra $, since paths in $ Q ' $ avoiding $ B $ are precisely paths in $ Q $ avoiding $ A \cup B $.

		(ii) We assume that $ B' = B \setminus A $ is non-empty as, otherwise, there is nothing to prove. For every $z \in I' = I \cap {}_k \la \BQnotA \ra $, the pre-removability of $B$ in $\La $ implies that $  z_{B '} \in I' $ as $z_{B '} = z_B $ avoids the arrows in $A$. We conclude that $B' $ is pre-removable in $\La / \la A + I \ra $. Furthermore, the first part of the proof applies for $A $ and $B' $, implying that $A \cup B = A \dot \cup { B' }$ is pre-removable in $\La $.
	\end{proof}

	We now define eventually removable sets of arrows, a notion that will allow us to formulate the main results of this subsection in a compact way.

	\begin{defn}
		\label{defn:ar.rem.3}
		Let $\La  = k Q / I$ be a bound quiver algebra. A set of arrows $ A $ is \emph{eventually removable} in $\La $ if there is an ordered partition $A = A_1 \dot \cup A_2 \dot \cup \ldots \dot \cup A_m $ such that $A_j $ is removable in $\La / { \la ( A_1  \dot  \cup \ldots \dot \cup A_{j-1} ) + I \ra } $ for every $j =1, 2, \ldots , m$. Every such ordered partition of $A$ is called \emph{admissible}.
	\end{defn}

	\begin{rem}
			\label{rem:2}
		The above definition makes sense as the set
		$A_1 \dot \cup \ldots \dot \cup A_{j-1} $ is pre-removable in $\La $ for every $j = 1 , 2 , \ldots , m $, by inductive application of \hyperlink{lem:ar.rem.1}{\cref{lem:ar.rem.1}{}.(i)}. Therefore, the quotient algebra $\La / { \la ( A_1  \dot  \cup \ldots \dot \cup A_{j-1} ) + I \ra } $ has a canonical representation as a bound quiver algebra, and the condition that $A_j $ is removable has a precise meaning.
	\end{rem}

	\begin{lem}
		\label{lem:ar.rem.7}
		Let $\La = k Q / I$ be a bound quiver algebra. Let $A , B \subseteq Q_1 $ be two sets of arrows such that $A  $ is pre-removable in $\La $ and $B  $ is two-sided (resp.\ only left) removable in $\La $. Then $B \setminus A$ is two-sided (resp.\ only left) removable in $ \La / { \la A + I \ra } $.
	\end{lem}

	\begin{proof}
		We assume that the set $B' = B \setminus A$ is non-empty. Let $\Ga $ denote the canonical representation $ k Q' / I' $ of $\La / \la A + I \ra $. As $ B $ is in particular pre-removable in $ \La $, it follows from \hyperlink{lem:ar.rem.1}{\cref{lem:ar.rem.1}{}.(ii)} that $ B' $ is pre-removable in $\Ga $.
		Furthermore, it holds that $B$ is non-repetitive in $\La $ if and only if every path of $Q$ that passes at least twice through an arrow in $B$ is in $I$. Therefore, it is clear that $B'$ is non-repetitive in $\Ga $ if $B$ is non-repetitive in $\La $.
		
		Now let $ \tilde{ \pi } \colon \La \epic \Ga $ denote the composition of the natural epimorphism induced by the ideal $ { \la A + I \ra } $, followed by the inverse of the isomorphism $\Ga \xrightarrow{\sim} \La / { \la A + I \ra } $ established in the proof of \cref{lem:ar.red.1}{}. It holds that $\tilde{ \pi } $ is the algebra epimorphism induced by the projection of quiver $Q$ onto $Q' = Q_0 \cup ( Q_1 \setminus A ) $, sending every vertex and every arrow of $Q$ to itself except for the arrows in $A $ that are sent to zero. Furthermore, let $ \tilde{\i} \colon \Ga \monicc \La $ denote the algebra monomorphism induced by the inclusion of quiver $Q'$ into $Q$, using \cite[Theorem~II.1.8]{ASS}{} as in the proof of \cref{lem:ar.red.1}{}. Then the quadruple $( \La , \Ga , \tilde{\i} , \tilde{ \pi } )$ is a radical-preserving algebra cleft extension with superfluous kernel, see \cref{defn:1}{}. We show next that the ideal $ \la B' + I' \ra $ is a left and right $ \Ga $-submodule of the ideal $ \la B + I \ra $ admitting an absorbing $\La$-complement, see \cref{defn:cleft.1}{}, up to a $ \Ga $-bimodule monomorphism.

		Let $ \la B + I \ra_{\nott \! A} $ and $ \la B + I \ra_{A} $
		denote the subspaces of $ \la B + I \ra $ generated by the paths in $ \BQB \cap \BQnotA $ and $ \BQB \cap \BQA $, respectively. Since $ A $ is pre-removable in $ \La $, we have that $I = I_A \oplus I_{\nott A}$, implying the direct sum decomposition $\la B + I \ra = \la B + I \ra_{\nott \! A} \oplus \la B + I \ra_{A}$.

		Next, we observe that $ \i $ is a $ \Ga $-bimodule homomorphism if $ \La $ is viewed as a $ \Ga $-bimodule by restricting scalars along $ \i $. Since the ideal $ \la B' + I' \ra $ is mapped onto $ \la B + I \ra_{\nott \! A} $, we may view $ \la B' + I' \ra $ as a left and right $ \Ga $-submodule of $ \la B + I \ra $ by identifying it with $ \la B + I \ra_{\nott \! A} $.
		Furthermore, the subspace $ \la B + I \ra_{A} $ is a $\La$-subbimodule of $\la B + I \ra $ and therefore a $ \La $-complement of $ \la B' + I' \ra $ from both sides. Finally, multiplying an element in $\la B + I \ra_{\nott \! A } $ with an element in $  \la A + I \ra $ from either side yields an element in $ \la B + I \ra_{A}$, implying that $ \la B + I \ra_{A} $ is an absorbing $ \La $-complement. It remains to apply \cref{cor:1}{} twice, once for each side, to deduce that $\pd \la B' + I' \ra_\Ga \leq \pd \la B + I \ra_\La $ and $\pd _\Ga \la B' + I' \ra \leq \pd_\La \la B + I \ra $, completing thus the proof.
	\end{proof}

	We are now ready to prove the central result of this subsection.

	\begin{prop}
		\label{prop:ar.rem.2}
		There exists a unique maximal eventually removable set of arrows for every bound quiver algebra.
	\end{prop}
	
	\begin{proof}
		Let $\La = k Q / I$ be a bound quiver algebra and let $A = A_1 \dot \cup A_2 \dot \cup \ldots \dot \cup A_m $ and $B = B_1 \dot \cup B_2 \dot \cup \ldots \dot \cup B_t $ be two maximal eventually removable sets of arrows in $\La$, where the given partitions are admissible. We assume that $ t \leq m$ and prove the corollary by induction on $m $. If $m = 1$, then each of $A \setminus B $ and $B \setminus A $ is removable in $\La / { \la B + I \ra }$ and $\La / { \la A + I \ra }$, respectively, due to \cref{lem:ar.rem.7}{}. It follows from the maximality of $A$ and $B$ that $A \setminus B $ and $B \setminus A $ are empty, and thus $ A = B $.

		Assume now that  $ m \geq 2 $, and let $A_i' = A_i \setminus B_1 $ and $B_j' = B_j \setminus A_1$ for $i = 2, 3, \ldots , m$ and $j =2 , 3 , \ldots , t$. We want to show that the sets $A' = A_2' \dot \cup A_3' \dot \cup \ldots \dot \cup A_m' $ and $B' = B_2' \dot \cup B_3' \dot \cup \ldots \dot \cup B_t' $ are two maximal eventually removable sets of arrows in the quotient algebra $ \La ' = \La / {  \la ( B_1 \cup A_1 ) + I \ra  } $. Note that the set $B_1 \cup A_1 $ is pre-removable in $\La $ according to \hyperlink{lem:ar.rem.1}{\cref{lem:ar.rem.1}{}.(ii)}, allowing the identification of $ \La ' $ with its canonical representation as a bound quiver algebra (\cref{defn:2}{}).
		
		We begin by showing that the set of arrows $A'$ is eventually removable in $\La '$ via the given partition.
		For any integer $ \mu $ such that $1 \leq \mu \leq m$, the set $A_1 \cup A_2 \cup \ldots  \cup A_\mu $ is pre-removable in $\La $ according to \cref{rem:2}{}.
		Furthermore, the set $B_1$ is removable in $\La $, implying that $B_1 \setminus  ( A_1   \cup \ldots  \cup A_\mu ) $ is removable in $\La '' = \La / \la ( A_1 \cup \ldots  \cup A_\mu  ) + I   \ra $ due to \cref{lem:ar.rem.7}{}. Since $A_{ \mu + 1 } $ is also removable in $\La''$, \cref{lem:ar.rem.7}{} implies that $A_{ \mu + 1 }' = A_{ \mu + 1 } \setminus ( B_1 \setminus ( A_1 \cup \ldots \cup A_\mu  )) $ is removable in the quotient of $\La''$ over the ideal generated by the arrows of $B_1 \setminus  ( A_1 \cup \ldots \cup A_\mu ) $. Note that this is equivalent to $A_{ \mu + 1 }' $ being removable in $\La / \la ( B_1 \cup A_1 \cup A_2' \cup \ldots \cup A_\mu' ) + I \ra $ due to \hyperlink{lem:ar.rem.1}{\cref{lem:ar.rem.1}{}}, as $ B_1 \cup A_1 \cup A_2' \cup \ldots \cup A_\mu' = B_1 \cup A_1 \cup A_2 \cup \ldots \cup A_\mu$.

		We show next that $A'$ is maximal as an eventually removable set of arrows in $\La '$. The set $B_1 \setminus A $ is removable in $\La / \la A + I \ra $ according to \cref{lem:ar.rem.7}{}, as $ A $ is pre-removable in $\La $. Therefore, the maximality of $A$ implies that $B_1 \subseteq A$. In particular, the canonical representation of the quotient of $\La ' $ over the ideal generated by $A'$ is the same as the canonical representation of $\La / \la A + I \ra $ according to \hyperlink{lem:ar.rem.1}{\cref{lem:ar.rem.1}{}}. The maximality of $ A' $ in $\La'$ follows now from the maximality of $ A $ in $\La $.

		Analogous arguments show that $B'$ is also a maximal eventually removable set of arrows in $\La '$ via the given partition, and that $A_1 $ is contained in $ B$. Furthermore, since the lengths of the given admissible partitions of $A'$ and $B'$ are $m-1 $ and $t-1$, respectively, it follows from the inductive hypothesis that $A' = B' $. Therefore, we have $A   =  A' \cup B_1 \cup A_1  =   B' \cup A_1 \cup B_1 = B$, where we have used equalities $A = A \cup B_1 $ and $B = B \cup A_1 $.
	\end{proof}

	\cref{prop:ar.rem.2}{} suggests the following definition.
	
	\begin{defn}
		Let $\La = k Q / I $ be a bound quiver algebra.
			\begin{enumerate}[\rm(i)]
				\item An arrow of $Q$ is \emph{eventually removable in $\La$} if it belongs to the (unique) maximal eventually removable set of arrows, denoted as $A^{\textsc{er}}_\La$.
				\item The \emph{arrow reduced version of $\La$}, denoted as $\La_{ \textsc{arv} }$, is the bound quiver algebra occurring as the canonical representation of the quotient $\La / \la A^{\textsc{er}}_\La + I \ra  $.
				\item An algebra is \emph{arrow reduced} if it coincides with its arrow reduced version, i.e.\ if it possesses no non-empty removable sets of arrows.
			\end{enumerate}
	\end{defn}

	The motivation for the definition of the arrow reduced version of a bound quiver algebra rests on the following theorem.

	\begin{thm}
		\label{thm:ar.rem.main}
		For every bound quiver algebra $\La  $, it holds that
		\[
		\fpd \La < \infty   \,\,   \iff  \,\,   \fpd \La_{ \textsc{arv} } < \infty
		\]
		and the equivalence remains valid if $\fpd $ is replaced by $\Fpd $ or $\gd $.
	\end{thm}

	\begin{proof}
		Let $A^{\textsc{er}}_\La = A_1 \dot \cup A_2 \dot \cup \ldots \dot \cup A_m $ be an admissible ordered partition of $A^{\textsc{er}}_\La $, that is $A_j $ is removable in $ \La_j : = \La / { \la ( A_1  \dot  \cup \ldots \dot \cup A_{j-1} ) + I \ra } $ for every $j =1, 2, \ldots , m$. \cref{thm:GAR}{} implies that
		\[
		\fpd \La_{j-1} < \infty   \,\,   \iff  \,\,   \fpd \La_j < \infty
		\]
		for every $j = 1 , 2 , \ldots , m$, as well as the respective equivalences for $\Fpd $ and $\gd $. If the algebra $ \La_{ \textsc{arv} } $ was not arrow reduced, then there would be a non-empty removable set of arrows $A_{m+1}$ in $ \La_{ \textsc{arv} } $. Consequently, the set $ A^{\textsc{er}}_\La \dot \cup A_{m+1}$ would be eventually removable in $\La$, a contradiction to the maximality of $A^{\textsc{er}}_\La $.
	\end{proof}

	In the next corollary, we showcase the depth of our approach by linking arrow removability to the finiteness of the global dimension. We then provide combinatorial criteria (in \cref{subsec:4.2}{}) and illustrative examples (in \cref{subsec:4.3}{}) before introducing the arrow irredundant version in \cref{subsec:4.4}{}.
	
	\begin{cor}
			\label{cor:2}
		For a bound quiver algebra $ \La = k Q / I $, the following conditions are equivalent:
			\begin{enumerate}[\rm(i)]
				\item The global dimension of $ \La $ is finite.
				\item The arrow reduced version of $ \La $ is semisimple.
			\end{enumerate}
	\end{cor}

	\begin{proof}
		If $ \gd \La < \infty $, then the set of all arrows in $ \La $ is (two-sided) removable; see \cref{exam:ar.rem.0}{}. In particular, the arrow reduced version $ \La_{ \textsc{arv} } = \La / J( \La ) $ of $ \La $ is semisimple. The converse implication follows directly from \cref{thm:ar.rem.main}{}, since semisimple algebras have zero global dimension.
	\end{proof}

	\subsection{Two combinatorial criteria}		\label{subsec:4.2}
	
	This section establishes two criteria that will be employed in the next subsection; one for eliminating arrows from being eventually removable (\cref{cor:ar.rem.2}{}) and one for excluding algebras from being Iwanaga-Gorenstein (\hyperlink{lem:Gor}{\cref{lem:Gor}{}}{}). The relevance of the latter result rests on the well-known fact that Iwanaga-Gorenstein Artin algebras have finite little finitistic dimension, see \cite[Proposition~6.10]{applications.of.contrvar.finite}{}.
	Both criteria are based on the following lemma.

	\begin{lem}
		\label{lem:ar.rem.6}
		Let $\La = k Q / I$ be a bound quiver algebra with a squared-zero loop $\lal \colon i \to i$ and let $ {}_\La M$ be a module. If there is an element $  m_0 \in M  $ not contained in $\rad_\La M $ and such that $ e_i  m_0 = m_0 $ and $( \lal + I )  m_0 = 0 $, then $ \pd_\La M = \infty $.
	\end{lem}

	\begin{proof}
		It suffices to prove that $M$ is not projective and that $\Omega^1_\La (M)$ satisfies the same conditions as $M$. Observe that there is a projective cover of the form $ f \colon \La e_i \oplus P  \epic M$ such that $f( e_i , 0  ) = m_0$, where $ {}_\La P $ is some projective module. The element $z = ( \lal + I , 0 )$ is in $ \Omega^1_\La (M) = \Ker f$ implying that $M$ is not projective. However, it holds that $z $ is not in $ \rad  \Omega^1_\La (M) $ since $I$ is an admissible ideal of $k Q$. It remains to observe that $e_i  z = z$ and $( \lal + I )  z = 0$.
	\end{proof}
	
	\begin{rem}
		One can recover the Strong No Loop Theorem \cite{ILP}{} via \cref{lem:ar.rem.6}{}, but only for squared-zero loops. Indeed, if $\lal \colon i \to i $ is such a loop in $\La = k Q / I$, then the simple $\La$-module corresponding to vertex $i$ clearly satisfies the assumptions of \cref{lem:ar.rem.6}{} for any non-zero element.
	\end{rem}

	\begin{cor}
		\label{cor:ar.rem.2}
		Let $\La = k Q / I$ be a bound quiver algebra with a squared-zero loop $\lal \colon i \to i$, and assume that $\a  $ is an arrow with target $ t ( \a ) = i $. Then $ \a $ is not eventually removable in $\La $ if $\a \lal \in I$. In particular, a squared-zero loop is never eventually removable.
	\end{cor}
	
	\begin{proof}
		Let $A $ be a set of arrows in $\La $. If $ \a \in A$, then the ideal $\la A + I \ra $ satisfies the conditions of \cref{lem:ar.rem.6}{} as a right $ \La $-module for $m_0 = \a + I $. Therefore, it holds that $\pd \la A + I \ra_\La = \infty $, and a removable set of arrows cannot contain $\a$. Furthermore, if $A$ is a removable set of arrows in $\La$, then $\lal$ is again a squared-zero loop in the canonical representation of $\La / { \la A + I \ra } $, and $\a $ is an arrow satisfying the same properties as in $\La $. We conclude that $\a $ is not eventually removable by induction over some admissible partition of $ A^{ \textsc{er} }_\La $.
		
		It remains to note that $ \a = \lal $ clearly satisfies the assumptions of the corollary if $ \lal $ is a squared-zero loop. 
	\end{proof}

	\begin{cor}		\label{lem:Gor}		\hypertarget{lem:Gor}
		Let $\La = k Q / I$ be a bound quiver algebra with a squared-zero loop $\lal \colon i \to i$. If there is an arrow $\a \neq \lal $ with target $ t ( \a ) = i $ such that $\a \b \in I $ for every arrow $\b $, then $\La $ is not Iwanaga-Gorenstein. Furthermore, the same holds if there is an arrow $ \g \neq \lal $ with source $ s ( \g ) = i $ and $\b \g \in I$ for every arrow $ \b $.
	\end{cor}
	
	\begin{proof}
		We assume that $\a \neq \lal $ is an arrow with $t ( \a ) = i$  such that $\a \b \in I $ for every arrow $\b $. It suffices to show that $\pd _\La D (\La _\La ) = \infty $ where $D = \Hom_k ( - , k)$ is the standard duality between left and right finitely generated $\La $-modules. Let $\CC$ denote any set of paths in $Q$ such that $\CC + I $ is a $k$-basis of $\La $, and observe that all trivial paths and all arrows of $Q$ are in $\CC$ since $I$ is admissible. Let $\a^*$ be the element of $D ( \La _\La ) $ defined by $\a + I \mapsto 1_k $ and $ p + I \mapsto 0 $ for every other path $ p $ in $ \CC $.
		Then it holds that $\a^* \notin \rad_\La D(\La_\La )$.  Indeed, if $\phi \in D( \La _\La )$ and $ \b $ is any arrow of $Q$, then $( ( \b + I )  \phi)( \a + I ) = \phi ( \a \b + I ) = 0 $ since $\a \b \in I$. Furthermore, it holds that $e_i  \a^* = \a^* $ and $( \lal + I )  \a^* = 0$ since $\a \neq \lal $ and $I$ is an admissible ideal of $k Q$. Our claim follows now from \cref{lem:ar.rem.6}{} applied for $M = D( \La _\La )$ and $ m_0 = \a^*$.
		
		The proof of the second part is dual and therefore omitted.
	\end{proof}

	\subsection{Examples}		\label{subsec:4.3}
	In this subsection, we calculate the arrow reduced version of concrete bound quiver algebras. As a consequence, we show that the (big) finitistic dimension of the algebras in \cref{subsec:3.3}{} is finite, see \cref{exam:ar.rem.4}{}. Moreover, \cref{exam:ar.rem.5}{} illustrates that the arrow reduced version of an algebra depends on the chosen representation of the algebra as a bound quiver algebra and not just its isomorphism class.

	\begin{exam}
		\label{exam:ar.rem.4}
		Let $\La_3 = k Q / I_3 $ be the algebra of \cref{exam:ar.rem.3}{}. We have already shown that the arrow $\b $ is (only left) removable in $\La_3$. It can be similarly verified that the non-repetitive arrow $\a $ is only left removable in $\La_3 $. Specifically, it holds that $\la \a + I_3 \ra_{\La_3} \simeq ( e_2 \La_3 )^{ 3} $ is projective. However, we have $ \pd _{ \La _3 } { \la \a + I_3 \ra } = \infty $ since the $\Omega$-$1$-periodic module $X$ occurs as a direct summand of its first syzygy.

		On the other hand, the arrow $\g $ is not removable in $\La_3 $. Indeed, it is repetitive (for instance $\g \a \b \z \g$ is non-zero in $\La_3$), and $\pd _{\La _3} \la \g + I_3 \ra = \infty $ since $\Omega^1_{\La _3} ( \la \g + I_3 \ra ) \simeq X^{ 10 } $. Nonetheless, it holds that $\g $ is $0$-biremovable (and thus two-sided removable) in $\La_3 / { \la \b + I_3 \ra }$. Indeed, the ideal $I'$ in the canonical representation of $\La_3 / { \la \b + I_3 \ra }$ is generated by the relations $R_0 = \{  \a \e - \d \a   , \,  \d ^2   , \,  \e ^2   , \,  \z ^2  \}$.
		We conclude that $\g $ is eventually removable in $\La_3 $.
		
		It follows from \cref{lem:ar.rem.7}{} that arrows $ \a $, $ \b $ and $ \g $ are eventually removable. On the other hand, no loop of $ \La_3 $ is eventually removable as they are all squared-zero, see \cref{cor:ar.rem.2}{}. All in all, we have that $A^{ \textsc{er} }_{\La_3} = \{\a , \b , \g \} $. Furthermore, it holds that $\Fpd { ( \La_3 )_{ \textsc{arv} } } = 0$ as $ ( \La_3 )_{ \textsc{arv} } $ is isomorphic to the direct product three times with itself of the local algebra $\La_0$ with one vertex and one squared-zero loop. We conclude that $\Fpd { \La_3 } < \infty$ according to \cref{thm:ar.rem.main}{}, thus demonstrating the utility of the arrow reduced version for studying the finitistic dimensions.
		
		Similarly, one may calculate that $A_{ \La_1 }^{ \textsc{er} } = \{ \a , \b , \g \}$ for the algebra $\La_1 $ of \cref{exam:ar.rem.1}{}. In particular, it holds that $( \La_1 )_{\textsc{arv}} \simeq ( \La_3 )_{\textsc{arv}}$, and thus $\Fpd { \La_1 } < \infty $ 
		
		As for algebra $\La_2 $ of \cref{exam:ar.rem.2}{}, one can show that $A_{ \La _2 }^{ \textsc{er} } = \{ \a , \g \} $. Note that the relation $\b \z $ causes the ideal generated by $\b $ in $\La_2 / { \la \{ \a , \g \} + I_2 \ra } $ to have infinite projective dimension as a right $\La_2$-module according to \cref{lem:ar.rem.6}{}. It is clear though that $( \La_2 )_{ \textsc{arv} } $ is isomorphic to the direct product of $\La_0$ and a triangular matrix algebra defined by two copies of $\La_0$. In particular, we have that $\Fpd { ( \La_2 )_{ \textsc{arv} } } \leq 1$ (see \cite{FGR}{}), implying that $ \Fpd { \La_2 } < \infty $.
	\end{exam}

	\begin{exam}
		\label{exam:ar.rem.5}
		
		Let $Q$ be the quiver of \hyperlink{fig2'}{Figure \ref{fig2}{}}, and let $k $ be any field. Consider the bound quiver algebra $\La = k Q / I $ where $I$ is the ideal of $k Q$ generated by the relations $R$ in the table of the same figure.
		
		It is quite straightforward to verify that $\La $ is irreducible for any $m \geq 3 $ in the sense of \cref{defn:irredu}{}; see also \cref{exam:ar.rem.1}{}. For instance, the finitistic dimensions of $\La $ are non-zero as $ \La $ does not possess a right ideal isomorphic to $S_{\La^{\! \mathrm{op}}} ( { 1 } ) $, the Loewy length of $ \La $ is $ m+1 $ with longest non-zero paths the paths $\a _j \b_1 \ldots \b_{m-1} $ for $j =1 , 2$, and $\La $ is not monomial as the paths $\l_0 \a_1 $ and $\a_2 \l _1 $ are equal and non-zero in $\La $. Furthermore, no arrow of $Q$ is $0$-biremovable. Indeed, every arrow different from $\b_1 $ occurs in some generating relation of $I$ in a path of length 2. To see that the same holds for $ \b_1 $, note that one of the paths $\l_1 \b_1 \l_2 $, $\l_1 \b_1 $ and $\b_1 \l_2 $ has to occur in every generating set for $I$ as $\l_1 \b_1 \l_2 \in I$.

		Moreover, the algebra $\La$ is not Iwanaga-Gorenstein, as \cref{lem:Gor}{} applies for the loop $\l_0$ and the arrow $\e$.

		\begin{figure}[h!]		\phantomsection		\hypertarget{fig2'}{}
			\vspace*{0.6cm}
			\centering
			\begin{minipage}{0.69\textwidth}
				\centering
				\begin{tikzpicture}[scale=0.82]


					\draw  ($(0,0)$) circle (.08);
					\node at ($(0,0)+(270:0.36)$) {$\mathbf{3}$};
					
					\draw  ($(0,0)+(180:1.4)$) circle (.08);
					\node at ($(0,0)+(180:1.4)+(270:0.36)$) {$\mathbf{2}$};
					
					\draw  ($(0,0)+(180:2.8)$) circle (.08);
					\node at ($(0,0)+(180:2.7)+(270:0.36)$) {$\mathbf{1}$};
					
					\draw  ($(0,0)+(180:4.2)$) circle (.08);
					\node at ($(0,0)+(180:4.23)+(270:0.36)$) {$\mathbf{0}$};

					\draw  ($(0,0)+(0:2.6)$) circle (.08);
					\node at ($(0,0)+(0:2.6)+(270:0.36)$) {$\mathbf{m-1}$};
					
					\draw  ($(0,0)+(0:4)$) circle (.08);
					\node at ($((0,0)+(0:3.90)+(270:0.36)$) {$\mathbf{m}$};
					
					\draw  ($(0,0)+(180:3.5)+(270:1)$) circle (.08);
					\node at ($((0,0)+(180:3.5)+(270:1)+(359:0.80)$) {$\mathbf{m+1}$};

					\draw[fill=black]  ($(0,0)+(0:1.4)$) circle (.02);
					\draw[fill=black]  ($(0,0)+(0:1.25)$) circle (.02);
					\draw[fill=black]  ($(0,0)+(0:1.55)$) circle (.02);


					\draw[->,shorten <=6pt, shorten >=6pt] ($(0,0)+(180:4.2)$) .. controls +($(0:0.4)+(90:0.67)$) and +($(180:0.4)+(90:0.67)$) .. ($(0,0)+(180:2.8)$);
					\node at ($((0,0)+(180:3.4)+(90:0.77)$) {$\alpha_1$};
					\draw[->,shorten <=6pt, shorten >=6pt] ($(0,0)+(180:4.2)$) -- ($(0,0)+(180:2.8)$);
					\node at ($((0,0)+(180:3.5)+(90:0.2)$) {$\alpha_2$};

					\draw[->,shorten <=6pt, shorten >=6pt] ($(0,0)+(180:2.8)$) -- ($(0,0)+(180:1.4)$) ;
					\node at ($(0,0)+(180:2.1)+(90:0.3)$) {$\beta_1$};
					
					\draw[->,shorten <=6pt, shorten >=6pt] ($(0,0)+(180:1.4)$)  -- ($(0,0)$);
					\node at ($(0,0)+(180:0.7)+(90:0.3)$) {$\beta_2$};
					
					\draw[->,shorten <=6pt, shorten >=10pt] ($(0,0)$)  -- ($(0,0)+(0:1.4)$);
					\node at ($(0,0)+(0:0.6)+(90:0.3)$) {$\beta_3$};
					
					\draw[->,shorten <=10pt, shorten >=6pt] ($(0,0)+(0:1.4)$)  -- ($(0,0)+(0:2.6)$);
					
					\draw[->,shorten <=6pt, shorten >=6pt] ($(0,0)+(0:2.6)$)  -- ($(0,0)+(0:4)$);
					\node at ($(0,0)+(0:3.3)+(90:0.3)$) {$\, \, \beta_{m-1}$};

					\path 
					($(0,0)+(180:4.2)$)				coordinate (1)      
					($(0,0)+(180:3.5)+(270:1)$)		coordinate (2)
					($(0,0)+(180:2.8)$)			 	coordinate (3)
					;
					
					\coordinate (Mid 1 2) at ($(1)!0.6!(2)$);
					\coordinate (Mid 2 3) at ($(2)!0.45!(3)$);

					\draw[->,shorten <=6pt, shorten >=6pt] ($(0,0)+(180:4.2)$) -- ($(0,0)+(180:3.5)+(270:1)$) ;
					\node at ($(Mid 1 2)+(270-55.08:0.25)$) {$\gamma$};
					
					\draw[->,shorten <=6pt, shorten >=6pt] ($(0,0)+(180:3.5)+(270:1)$) -- ($(0,0)+(180:2.8)$) ;
					\node at ($(Mid 2 3)+(270+55.08:0.2)$) {$\delta$};


					\draw[->,shorten <=6pt, shorten >=6pt] ($(0,0)+(0:4)$) .. controls +($(180:0.5)+(90:2.7)$) and +($(0:0.5)+(90:2.7)$) .. ($(0,0)+(180:4.2)$);
					\node at ($(0,0)+(180:0.1)+(90:2.28)$) {$\varepsilon$};


					\draw[->,shorten <=6pt, shorten >=6pt] ($(0,0)+(180:4.2)$).. controls +(180+40:1) and +(180-40:1) .. ($(0,0)+(180:4.2)$);
					\node at ($(0,0)+(180:5.1)$) {$\l_0$};
					
					\draw[->,shorten <=6pt, shorten >=6pt] ($(0,0)+(0:4)$).. controls +(0+40:1) and +(0-40:1) .. ($(0,0)+(0:4)$);
					\node at ($(0,0)+(0:4.96)$) {$\l_m$};
					
					\draw[->,shorten <=6pt, shorten >=6pt] ($(0,0)+(180:2.8)$).. controls +(85+15:0.92) and +(85-30:1) .. ($(0,0)+(180:2.8)$);
					\node at ($(0,0)+(180:2.7)+(81:0.98)$) {$\l_1$};
					
					\draw[->,shorten <=6pt, shorten >=6pt] ($(0,0)+(180:1.4)$).. controls +(90+30:1) and +(90-30:1) .. ($(0,0)+(180:1.4)$);
					\node at ($(0,0)+(180:1.4)+(88:0.98)$) {$\l_2$};
					
					\draw[->,shorten <=6pt, shorten >=6pt] ($(0,0)$).. controls +(90+30:1) and +(90-30:1) .. ($(0,0)$);
					\node at ($(0,0)+(88:0.98)$) {$\l_3$};
					
					\draw[->,shorten <=6pt, shorten >=6pt] ($(0,0)+(0:2.6)$).. controls +(90+30:1) and +(90-30:1) .. ($(0,0)+(0:2.6)$);
					\node at ($(0,0)+(0:2.6)+(88:0.98)$) {$\l_{m-1}$};

					\draw[->,shorten <=6pt, shorten >=6pt] ($(0,0)+(180:3.5)+(270:1)$).. controls +(180+90+40:1) and +(180+90-40:1) .. ($(0,0)+(180:3.5)+(270:1)$);
					\node at ($(0,0)+(180:3.5)+(270:1.85)$) {$\l_{m+1}$};

				\end{tikzpicture}
			\end{minipage}\hfill
			\begin{minipage}{0.31\textwidth}
				\centering
				\resizebox{4cm}{!}{
					\begin{tabular}{| C{1.8em} l|}
						\hline
						\multicolumn{2}{|c|}{\TTT Relations $R$ \BBB}	 \\
						\hhline{|==|}
						
						\multicolumn{2}{|c|}{\TTT $ \l_0  \a_1  -  \a_2  \l_1  \, , \, \l_0  \a_2  -  \a_1  \l_1  $ \BBB}    \\
						\hline
						\multicolumn{1}{|c}{\TTT $\l_i^2  $}  & \multicolumn{1}{l|}{ $ \colon \,\,\,  0 \leq i \leq  m+1 $   \BBB}     \\
						\hline
						\TTT	$	\l_i \b_i  $ 		& 	\multirow{2}{*}{$  \colon \,\,\, 2 \leq  i \leq m-1 $} \\
						$	\b_i \l_{i+1} $ 	&	 \BBB \\ 
						\hline
						\multicolumn{2}{|c|}{\TTT $ \g \l_{m+1}  \, , \,  \l_{m+1} \d  \, , \,  \l_m \e  \, , \,  \e \l_0  $}  \\
						\multicolumn{2}{|c|}{$ \l_1 \b_1 \l_2  \, , \,  \d \b_1 \l_2  $ \BBB} 	 \\
						\hline
						\multicolumn{2}{|c|}{\TTT $ \b_{m-1} \e  \, , \,  \e \a_1  \, , \,  \e \a_2  \, , \,  \e \g $ \BBB}   \\ 
						\hline	
						\multicolumn{2}{|c|}{\TTT $ \l_1 \b_1 \b_2  \, , \,  \d \b_1 \b_2  $ \BBB}   \\ 
						\hline
					\end{tabular}
				}
			\end{minipage}
			\caption{Quiver $Q$ and relations defining $I$}
			\label{fig2}
		\end{figure}
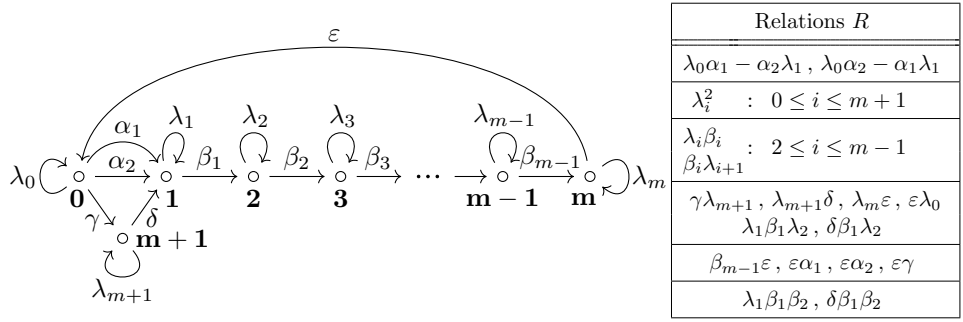

		Now let $A $ be a removable set of arrows in $\La $. Applying \cref{cor:ar.rem.2}{} multiple times for all loops and most non-loop arrows, we deduce that $ A_{ \La }^{ \textsc{er} }  \subseteq \{ \a_1   ,   \a_2   ,   \b_1   ,   \d  \} $; in particular, the same holds for $ A $. If $\b_1 \in A $, then the ideal $ \la A + I \ra $ satisfies the conditions of \cref{lem:ar.rem.6}{} as a right $\La$-module for the loop $\l_2$ and $m_0 = \l_1 \b_1 + I $, leading to the contradiction that $\pd \la A + I \ra_\La = \infty $. If $\d \in A$, then the simple module $S_{\La^{\! \mathrm{op}}} ( { 2 } ) $, which has infinite projective dimension according to \cite{ILP}{}, is a direct summand of the second syzygy of $ \la A + I \ra _\La $, leading to the same contradiction $\pd \la A + I \ra_\La = \infty $. Indeed, it can be verified that $S_{\La^{\! \mathrm{op}}} ( { 2 } ) $ is a direct summand of the first syzygy of $ Z =( \d + I ) \La $, which is a direct summand of $ \la A + I \ra_\La $.

		On the other hand, the set $A = \{ \a_1 , \a_2 \}$ is (only left) removable in $\La $. Indeed, it is pre-removable and non-repetitive in $\La$, the right $\La$-module $\la A + I \ra_\La \simeq e_1 \La ^{ 2} $ is projective and $\pd _\La \la A + I \ra = \infty $. To verify the last claim observe that the ideal $ \la A + I \ra $ contains as direct summands $2 (m -1 )$ copies of $S_{\La} ( 0 ) $ when viewed as a left $ \La $-module.
		
		Using similar arguments yields that $\La / \la A + I \ra $ is arrow reduced.
		
		All in all, we have that $ A_{ \La }^{ \textsc{er} } = \{ \a _1 , \a _2 \} $ and the arrow reduced version of $\La $ is the algebra $\La_{ \textsc{arv} } = k Q' / I '$ where $Q'$ is the quiver that results from $Q$ if we remove the arrows $\a_1$ and $\a_2 $, and $I'$ is the ideal of $k Q' $ generated by all the relations in the table of \hyperlink{fig2'}{Figure \ref{fig2}{}} except for the ones involving $\a_1 $ or $\a_2 $. In particular, the algebra $\La_{ \textsc{arv} } $ is monomial, implying that $\Fpd \La < \infty $ due to \cref{thm:ar.rem.main}{} and \cite{GKK}{}.
		
		Now let $ \phi \colon k Q \to k Q $ be the unique $k$-algebra homomorphism defined via the correspondence sending every trivial path of $ Q $ and every arrow different from $\a_1$ and $\a_2 $ to itself, and $\a_j $ to $\a_j + \g \d $ for $j = 1, 2 $. It holds that $\phi $ is surjective as $\phi ( \a_j - \g \d  ) = \a_j $ for $j =1, 2$. To see that $\phi$ is an isomorphism, write an arbitrary element $z \in k Q $ in the form
		\[
		z = \sum_{i=1}^\zeta \mu_i \cdot p_i + \sum_{j=1}^\xi \nu_j \cdot q_j
		\]
		for non-zero coefficients $\mu_i, \nu_j \in k $, where the paths $p_i$ avoid $\a_1 $ and $\a_2$, and each path $q_j$ passes through at least one of $\a_1 $ and $\a_2$; in other words, we have $ z_{ \nott \! A } = \sum \mu_i \cdot p_i $ and $ z_A = \sum \nu_j \cdot q_j $ for $A = \{ \a_1 , \a_2 \}$. Then $\phi (z) = z + z'$, where $z'$ is the sum of the elements of the form $\nu_j \cdot w_j^\e$ for $w_j^\e $ ranging over the paths that result from $q_j$ after replacing any non-empty combination of the arrows $\a_1$ and $\a_2 $ that occur in $q_j$ with $\g \d $. If $\xi > 0 $, then we may assume that $ q_1 $ has minimal length among the paths $q_j$. Then every path occurring in $z'$ has length greater than the length of $q_1 $, implying that $q_1$ occurs only once in $\phi (z )$ with coefficient $\nu_1 $. Therefore, if $\phi (z) = 0$ then $\nu_1= 0$, implying that $\xi = 0$ and $z = \phi (z ) = 0 $.

		Let $\tilde{ \La } = k Q / \tilde{ I } $ be the isomorphic algebra of $\La $ induced by $\phi$, that is $ \tilde{ I } = \phi ( I)$. It holds that
		the ideal $ \tilde{ I } $ is generated by $ \phi (R)$ or, equivalently, by all relations in $R$ except for the relations in the first row of the table of \hyperlink{fig2'}{Figure \ref{fig2}{}}, which have to be replaced by $\l_0 \a_1 - \a_2 \l_1 + \l_0 \g \d - \g \d \l_1 $ and $ \l_0 \a_2 - \a_1 \l_1 + \l_0 \g \d - \g \d \l_1 $. It follows that $\tilde{ I } $ is an admissible ideal of $k Q $ as the square of all loops is in $\tilde{ I }$ and $\b_{m-1} \e , \l_m \e \in \tilde{ I } $.
		
		If $\tilde{ A }  $ is a removable set of arrows in $\tilde{ \La }$, then one may argue in the same way as we did for $\La$ that $ \tilde{ A } \subseteq \{ \a_1 , \a_2 , \d \}$. On the other hand, the paths $\a_2 \l_1 $, $\l_0 \a_1 $ and $\l_0 \g \d - \g \d \l_1$ are not in $\tilde{ I } $, implying that the only subset of $ \{ \a_1 , \a_2 , \d \} $ that is pre-removable in $\tilde{ \La }$ is $ \tilde{ A } = \{ \a_1 , \a_2 , \d \}$. However, the set $\tilde{ A } $ is not removable. Indeed, the ideal $ \la \tilde{ A } + \tilde{ I } \ra $ decomposes as a direct sum of right $\tilde{ \La } $-modules as follows
		\[
		\la \tilde{ A } + \tilde{ I } \ra  = (\a_1 + \tilde{ I } ) \tilde{ \La } \oplus (\a_2 + \tilde{ I } ) \tilde{ \La } \oplus (\d + \tilde{ I } ) \tilde{ \La }
		\oplus (\l_0 \a_1 + \tilde{ I } ) \tilde{ \La } \oplus (\g \d + \tilde{ I } ) \tilde{ \La }
		\]
		where the first two summands are isomorphic to $e_1 \tilde{ \La }$ and the rest are isomorphic to each other. Furthermore, the simple module $S_{\tilde{ \La }^{\mathrm{op}}} ( { 2 } ) $ is a direct summand of the first syzygy of the right ideal $ (\d + \tilde{ I } ) \tilde{ \La }$, implying that $\pd { \la \tilde{ A } + \tilde{ I } \ra }_{\tilde{ \La }} = \infty $.
		
		We conclude that $\tilde{ \La }$ is arrow reduced even though its isomorphic algebra $\La $ has two eventually removable arrows.
	\end{exam}

	These examples illustrate the practical computation and properties of the arrow reduced version. Next, we explore a related but distinct concept: the arrow irredundant version, which addresses arrows that are `redundant' in a different sense. Since the assumptions on redundant arrows are stronger, an algebra and its arrow irredundant version have even more homological properties in common.

	\subsection{Arrow irredundant version}		\label{subsec:4.4}
	Let $ \La = k Q / I $ be a bound quiver algebra. This subsection aims to show that arrows not occurring in a generating set for $ I $ do not impact key homological properties of $\La$. Specifically, we show that condition (ii) of the arrow removal operation \cite{arrowrem1}{} (as recalled in the \hyperlink{appendix}{Appendix}) is redundant regarding the invariance of (finiteness of) the global, little and big finitistic dimensions. Moreover, it is redundant for the preservation of the three homological invariants studied in \cite{arrowrem2}{}.

		\begin{defn}
			Let $\La = k Q / I $ be a bound quiver algebra. An arrow $\a $ is \emph{redundant} if there is a generating set for $ I $ avoiding $ \a $. Furthermore, we denote the set of redundant arrows in $\La $ by $ A^{\textsc{red}}_\La $.
		\end{defn}

		Note that an arrow is redundant if and only if it is $0$-biremovable or, equivalently, if the quotient algebra induced by the ideal generated by the arrow is an arrow removal of the initial algebra in the sense of \cite{arrowrem1}{}. This is essentially a restatement of \cite[Proposition~4.5]{arrowrem1}{}; see also \cref{cor:ar.rem.1}{} and \cref{exam:1}{}.

		\begin{rem}
			Every $0$-biremovable set of arrows in $\La $ is contained in $ A^{\textsc{red}}_\La $, whereas the converse is not true in general.
			However, every subset of $ A^{\textsc{red}}_\La $ is at least pre-removable in $\La $, see \hyperlink{lem:ar.rem.1}{\cref{lem:ar.rem.1}{}}.
			In particular, for any subset $ A \subseteq A^{\textsc{red}}_\La $, the quotient algebra $\La / \la A + I \ra $ may, and will, be identified with its canonical representation introduced in \cref{defn:2}{}.
		\end{rem}

	\begin{defn}
		The \emph{arrow irredundant version} of $\La = k Q / I $, denoted by $\La_{ \textsc{aiv} }$, is the bound quiver algebra occurring as the canonical representation of the quotient algebra $\La / \la A^{\textsc{red}}_\La + I \ra  $.
	\end{defn}
	
	In the next theorem, we establish the key properties of the arrow irredundant version of an algebra. In particular, we show that this construction is indeed irredundant. 
	
	\begin{thm}
		\label{cor:ar.rem.3}
		For a bound quiver algebra $\La = k Q / I $ and a set of redundant arrows $ A $, let $ \Ga = \La / \la A + I \ra $. Then:
		\begin{enumerate}[\rm(i)]
			\item It holds that $\fpd \La < \infty $ if and only if $\fpd \Ga < \infty $, and the equivalence remains valid if $\fpd$ is replaced by $\Fpd$ or $\gd$.
			\item The algebra $\La $ is Iwanaga-Gorenstein if and only if the same holds for $  \Ga $.
			\item The algebra $\La $ satisfies the finite generation condition for the Hochschild cohomology if and only if the same holds for $  \Ga $.
			\item The singularity categories of $\La $ and $  \Ga $ are triangle equivalent.
		\end{enumerate}
		In particular, the above statements hold for the arrow irredundant version $ \Ga = \La_{ \textsc{aiv} } $ of $ \La $, which additionally possesses no more redundant arrows.
	\end{thm}
	
	\begin{proof}
		Let $ A $ be any set of arrows contained in $ A^{\textsc{red}}_\La  $. We show first that the set of redundant arrows in $\La / \la A + I \ra $ is $ A^{\textsc{red}}_\La \setminus A $. In particular, this will imply that the quotient algebra ${ \La / \la A^{\textsc{red}}_\La + I \ra   }$ does not possess any redundant arrows. It suffices to show that, given a redundant arrow $ \a $ in $ \La $, another arrow $ \b $ is redundant in $\La $ if and only if it is redundant in $\La / \la \a + I \ra $. In fact, it is enough to assume that $\a$ is pre-removable for the ``only if'' part.
		
		We begin by assuming that $\b $ is redundant in $\La $ and, thus, there is a finite generating set $ T $ of relations for $ I $ avoiding $ \b $. The set $ T_{\nott \! \a} $, which obviously avoids $ \b $, generates the admissible ideal $ I' = I \cap {}_k \la \BQnota \ra $ appearing in the canonical representation of $ \La / \la \a + I \ra $. In particular, the arrow $ \b $ is redundant in $ \La / \la \a + I \ra $.
		Conversely, assume that $ \b $ is redundant in $ \La / \la \a + I \ra $, that is the ideal $ I' $ is generated by a finite set $ T ' $ of relations avoiding $ \b $. It follows that $ I ' = I \cap {}_k \la \BQnota \ra $, viewed as a subspace of the path algebra $ k Q $, is contained in the ideal of $ k Q $ generated by $ T ' $. As $I$ is generated by a set avoiding $ \a $ by assumption, it is generated by $ I ' $ and thus by $ T ' $.
		
		We observe now that an arrow $ \a $ is redundant in $\La $ if and only if $\La / { \la \a + I   \ra } $ is an arrow removal algebra of $\La $ in the sense of \cite{arrowrem1}{}, see \cref{cor:ar.rem.1}{}. In particular, the arrow irredundant version of $\La $ can be obtained from $\La $ by successively taking arrow removal quotients by one arrow at a time. Therefore, the first part follows from \cref{thm:GAR}{}, and the remaining parts follow from the relevant results of \cite{arrowrem2}{}.
	\end{proof}

	Using techniques developed in \cite{arrowrem1}{}, we show that the finitistic (and global) dimensions are actually invariant (and not just their finiteness) when removing redundant arrows, that is when the quotient algebra is not semisimple.
	
	\begin{prop}	
			\label{thm:1}
		For a bound quiver algebra $\La = k Q / I $ and a set of redundant arrows $ A $, let $ \Ga = \La / \la A + I \ra $. Then
			\[
				\fpd \La = \left\{
				\begin{array}{ll}
					\fpd \Ga ,  & \mbox{if } \fpd \Ga > 0 \\
					\fpd \Ga + 1 , & \mbox{if } \fpd \Ga = 0 .
				\end{array}
				\right.
			\]
		Moreover, the statement remains valid if $ \fpd $ is replaced by $ \Fpd $ or $ \gd $.
	\end{prop}

	\begin{proof}
		It suffices to prove the equalities for $\La $ and the quotient $ \Ga = \La / \la \a + I \ra $, where $ \a $ is a redundant arrow.
		Note that any of the three homological dimensions of $\La $ is infinite if and only if the respective homological dimension of $\Ga $ is infinite according to \cref{thm:GAR}{}. Therefore, we may restrict our attention to the finite case, whence the desired equality for $\gd $ follows from the respective equality for $\fpd $ or $\Fpd $.

		For any module ${}_\La Y $ we have a short exact sequence of the form
			\[
				0 \to G(Y) \to \La \otimes_\Ga Y \to Y \to 0
			\]
		where $ G $ is by definition the kernel functor induced by the counit of the adjunction $( \La \otimes_\Ga - , {}_\Ga ( - ) )$; see also discussion before \cref{lem:ar.rem.3}{}, where we denote also by $ G $ the respective functor for right modules.
		
		In our case, the $ \La $-module $ G( Y )$ is always projective according to \cite[Lemma~2.7, Proposition~4.6]{arrowrem1}{}, since $ \Ga $ is an arrow removal algebra of $ \La $. 
		Therefore, we have
			\[
				\pd _\La Y = \left\{
				\begin{array}{ll}
					\pd _\La \La \otimes_\Ga Y,  & \mbox{if } \pd _\La \La \otimes_\Ga Y \geq 1 \\
					0 \mbox{ or } 1, & \mbox{else}
				\end{array}
				\right.
			\]
		from a well-known homological argumenent; see for instance \cite[Lemma~1.1]{FS}{}. To be more precise, the first case is determined by the condition $ \pd _\La \La \otimes_\Ga Y > \pd _\La G( Y ) = 0 $, and the second one by its complement and the inequality $ \pd _\La Y \leq \pd _\La G ( Y ) + 1 $. 
			
		On the other hand, let $ \i \colon \Ga \monicc \La $ denote the section algebra monomorphism of the natural epimorphism $ \pi \colon \La \epic \Ga $ established in \cref{rem:ar.rem.2}{}. Then $ \i $ is radical-preserving (with superfluous kernel) according to \cref{lem:cleft.2}{}; see also \cite{Giata1}{}. Furthermore, the ideal $ \la \a + I \ra $ is projective as a $\Ga$-bimodule according to \cite[Proposition~4.4.(ii)]{arrowrem1}{}; in particular, it is projective as a right $\Ga$-module. Therefore, the induction functor $\La \otimes_\Ga - $ preserves projective dimension, since it preserves projective covers non-trivially according to \cite[Corollary~3.9]{Giata1}{}, i.e.\ we have $\pd _\La \La \otimes_\Ga X = \pd _\Ga X $ for any module ${}_\Ga X $. Consequently, it holds that
			\[
				\pd _\La Y = \left\{
				\begin{array}{ll}
					\pd _\Ga Y ,  & \mbox{if } \pd _\Ga Y \geq 1 \\
					0 \mbox{ or } 1, & \mbox{else}
				\end{array}
				\right.
			\]
		for any module ${}_\La Y $.

		Assume that $\Fpd \Ga = d \geq 1 $ is finite. Let $ Y $ be a $ \La $-module of finite non-zero projective dimension. If ${}_\Ga Y $ is projective, then $\pd _\La Y \leq 1 \leq d $. Otherwise, we have that $ \pd _\La Y = \pd _\Ga Y $ is finite, implying that $ \pd_\La Y \leq d $. We deduce that $\Fpd \La \leq d $. Let now $ X $ be a $ \Ga $-module such that $\pd _\Ga X = d $, and set $ { Y ' } = {}_\La X $. Then ${}_\Ga { Y ' } $ has the same $ \Ga $-structure as the initial $ \Ga $-structure of $ X $ and thus $\pd_\La { Y ' } = \pd _\Ga { Y ' } = d $ since $ {}_\Ga { Y ' } = X $ is not projective. We conclude that $\Fpd \La = d = \Fpd \Ga $.
		
		Now let $ \Fpd \Ga = 0 $, and we aim to show that $ \Fpd \La = 1 $. We need the following characterization of zero finitistic dimension for an Artin algebra $ R $. According to \cite[Lemma~6.2]{Bass}{} and the standard duality between left and right finitely generated $ R $-modules, it holds that $ \Fpd R = 0 $ if and only if every simple right $ R $-module is isomorphic to a right ideal of $ R $.
		
		Since $ \a \colon i \to j $ is a redundant arrow, it holds that the quotient algebra $ \La / \la \a + I \ra $ is an arrow removal of $ \La $ in the sense of \cite{arrowrem1}{}. In particular, the ideal $ \la \a + I \ra $ is projective as a $ \La $-bimodule according to \cref{prop:ar.rem.1}{}, and thus also as a right $ \La $-module. Therefore, \cref{cor:cleft.3}{} implies that $ \Fpd \La \leq \Fpd \Ga + 1 = 1$, and it suffices to show that $ \Fpd \La $ is non-zero. We do this by showing that there is no right ideal of $ \La $ isomorphic to the simple right $ \La $-module corresponding to vertex $ i $ denoted by $ S_{ \La^{\textrm{op}}} (i) $.
		
		The existence of such an ideal that is isomorphic to $ S_{ \La^{\textrm{op}}} ( { i } ) $ is equivalent to the existence of an element $ z \in \La e_i $ such that $ z \ol{ \b } = 0 $ for every arrow $ \b $ of $ \La $. In the proof of \cref{lem:appendixA.1}{}, we show that the map $ f \colon \La e_i \otimes_k e_j \La \epic \la \a + I \ra $ defined by $ z_1 \otimes z_2 \mapsto z_1 ( \a + I ) z_2 $ for every $ z_1 \in \La e_i $ and $ z_2 \in e_j \La $ is a projective cover of $ \La $-bimodules. In this case, it is an isomorphism as the ideal $\la \a + I \ra $ is a projective $ \La $-bimodule. Therefore, if $ z \ol{ \a } = 0 $, then $ z = 0 $ as $ f( z \otimes e_j ) = z \ol{ \a } $. We conclude that there is no right ideal of $ \La $ isomorphic to $ S_{\La^{\! \textrm{op}}} ( { i } ) $, and thus $ \Fpd \La $ is non-zero.

		The proof for $\fpd $ is omitted as it is identical to the proof for $ \Fpd $, after observing that the restriction functors $ {}_\Ga ( - ) $ and $ {}_\La ( - ) $ preserve finite generation. For the case $ \fpd \Ga = 0 $, we also use the fact that the two finitistic dimensions of a left perfect ring are either both zero or both non-zero according to \cite[Theorem~6.3]{Bass}{}.  
	\end{proof}

	Using the well-known characterization of hereditary bound quiver algebras in terms of their bound quiver representation, we derive the following corollary.
	
	\begin{cor}
			\label{cor:3}
		For a bound quiver algebra $ \La $, the following are equivalent:
			\begin{enumerate}[\rm(i)]
				\item Algebra $ \La $ is hereditary (i.e.\ its global dimension is at most one).
				\item The arrow irredundant version of $ \La $ is semisimple.
			\end{enumerate} 
	\end{cor}

	\begin{proof}
		It is well-known that for $ \La = k Q / I $, it holds that $ \gd \La \leq 1 $ if and only if $ Q $ is acyclic and $ I = 0 $; see for instance \cite[Theorem~VII.1.7]{ASS}{}. Furthermore, it is clear that $ I = 0 $ if and only if $ A^{\textsc{red}}_\La = Q_1 $ or, equivalently, if the arrow irredundant version $\La / \la A^{\textsc{red}}_\La + I \ra  $ of $ \La $ is semisimple.
	\end{proof}

	Comparing \cref{cor:2,cor:3} quantifies in a way the difference between the generalized (\cref{subsec:3.1}{}) and the classical (\!\!\cite{arrowrem1}{}) arrow removal operation.

	We close this section with a criterion implying finite global dimension for cyclic Na\-ka\-ya\-ma algebras. Recall that an algebra is called Nakayama if all left and right indecomposable projective modules are uniserial. If the algebra is connected and can be represented by a bound quiver, then it is Nakayama exactly when its ordinary quiver is either the linearly oriented quiver of type $\mathbb{A}_n $ or the quiver $ C_n $ for some positive integer $ n $, where $C_n$ denotes the linearly oriented cyclic quiver with $ n $ vertices; see for instance \cite[Chapter~V]{ASS}{}. Furthermore, the algebra is called \emph{linear} or \emph{cyclic} Nakayama depending on the shape of its ordinary quiver.

	\begin{cor}
		Let $\La = k C_n / I $ be a cyclic Nakayama algebra with at least one redundant arrow. Then
			\[
				\gd \La \leq \lal_{\textsf{max}} \leq n - 1
			\]
		where $ \lal_{\textsf{max}} $ denotes the maximal length of a path avoiding redundant arrows. The provided bounds are both optimal.
	\end{cor}

		\begin{proof}
			By assumption, the set of redundant arrows $ A = A^{\textsc{red}}_\La $ is non-empty. It is not difficult to see that the arrow irredundant version $ \La_{ \textsc{aiv} } $ of $\La $ is isomorphic to a direct product of linear Nakayama algebras $\La_q $, one for each maximal path $q $ avoiding redundant arrows. Furthermore, the ordinary quiver of each $\La_q $ is the linearly oriented quiver of type $\mathbb{A}_{l(q) + 1} $, implying that $ \gd \La_q \leq l (q)$. Since the quiver of $ \La $ is not acyclic we have $\gd \La \geq 2 $ (see for instance \cite[Theorem~VII.1.7]{ASS}{}). Therefore, \cref{thm:1}{} implies that $\gd \La = \gd \La_{ \textsc{aiv} } $, and the desired inequalities follow from the above description of $ \La_{ \textsc{aiv} } $ as a direct product of linear Nakayama algebras.
			
			To see that the bound is optimal, consider the cyclic Nakayama algebra $\La = k C_n / I $, where $ n \geq 3 $ and $ I $ is the ideal generated by all relations of length $ 2 $ except for the two such relations passing through a distinguished arrow $ \a $. Then the algebra $ \La_{ \textsc{aiv} } = \La / { \la \a + I \ra } $ is the bound quiver algebra over the linearly oriented quiver of type $\mathbb{A}_{n}$ where all paths of length $ 2 $ are zero. Indeed, this is the canonical representation of the quotient algebra $ \La / { \la \a + I \ra } $ as a bound quiver algebra since $ \a $ is pre-removable (\cref{defn:2}{}). We conclude that
				\[
					\gd \La = \gd \La_{ \textsc{aiv} } = n - 1
				\]
			where the first equality follows from \cref{thm:1}{}, and the second is an easy computation. This completes the proof.
		\end{proof}

		We point out that the upper bounds provided in the above corollary are smaller than the one provided in \cite[Theorem~2.2]{MadMar}{} for arbitrary cyclic Nakayama algebras.

	All in all, in this section, we have introduced two distinct methods for simplifying bound quiver algebras -- the arrow reduced version and the arrow irredundant version. We have demonstrated that both constructions preserve the finiteness of key homological dimensions, and in the case of redundant arrows, precisely quantified their relative values. Our methods yield novel means to determine the homological properties of complex algebras by reducing them to the same properties over simpler but homologically equivalent algebras, as showcased by the provided examples.

	\section{Inverse operations to arrow removal}		\label{sec:adding.ar}
	
	\smallskip
	
	Understanding inverse operations is vital in algebra, as it allows for the controlled construction of complex algebraic structures from simpler ones, while preserving key properties. This section develops such an inverse for our generalized arrow removal (\cref{sec:gen.ar.rem}{}), enabling the creation of new algebras with finite finitistic dimensions.
	
	Toward this end, we make use of the concept of \emph{split extensions of algebras} as presented in \cite{Pierce}{}; see also \cite{thetaexts}{}, where the same concept is introduced in the broader context of arbitrary rings under the term \emph{$\theta$-extensions}. Given a ring $\Ga$ and a multiplicative $\Ga$-bimodule $(M, \theta)$, i.e.\ a $\Ga$-bimodule $M$ equipped with an associative homomorphism $\theta \colon M \otimes_\Ga M \to M$, we can construct a new ring $E = \Ga \ltimes_\theta M$ which is the direct sum $\Ga \oplus M$ as an abelian group. Furthermore, the subgroup $\wt{M} = 0 \oplus M$ is an ideal of $ E $ which plays a key role in our context.

	In the general case, preserving the finiteness of the finitistic dimension requires imposing homological conditions on $ M $ as a right (and sometimes also as a left) module over $ E $, where the $ E $-structures are the ones inherited from the ideal $\wt{M}$ under the obvious identification. Therefore, in \cref{subsec:adding.ar.2,subsec:adding.ar.3}{}, we establish two special setups where suitable conditions on $M$ as a $\Ga$-bimodule guarantee these homological requirements are met. Notably, the setup involving bimodules with strongly-finite projective dimension in \cref{subsec:adding.ar.3}{} leads to an explicit method for constructing bound quiver algebras with finite finitistic dimensions from smaller algebras with the same property; see \cref{exam:ar.rem.11}{} and \cref{exam:3}{}.

	\subsection{Nilpotent multiplicative bimodules}		\label{subsec:adding.ar.1}

	This subsection begins by recalling notions related to multiplicative bimodules and establishing the notation used throughout the section.
	
	Let $ \Ga $ be a ring. A \emph{multiplicative bimodule for $\Ga $} or a \emph{multiplicative $ \Ga $-bimodule} is a pair $(M, \theta)$ where $M$ is a non-zero $\Ga$-bimodule and $\theta \colon M \otimes_\Ga M \to M$ is an \emph{associative} $\Ga$-bimodule homomorphism, that is $\theta \circ ( { \theta \otimes 1_M } ) = \theta \circ ( { 1_M \otimes \theta } ) $. The \emph{split extension of $\Ga $ by $( M, \theta ) $}, denoted by $E = \Ga \ltimes _\theta M $, is the ring that is equal to $\Ga \oplus M$ as an abelian group with multiplication given by
		\[
			( \g_1 , m_1 )  ( \g_2 , m_2 ) = ( \g_1 \g_2 , \g_1  m_2 + m_1  \g_2 + \theta ( m_1 \otimes m_2 ) ) .
		\]
	For the purposes of this section, a \emph{split extension of a ring} is defined as an extension with respect to some multiplicative bimodule.
	
	It should be noted that the case $\theta = 0$ recovers the notion of trivial ring extensions, introduced by Nagata in 1962 \cite{Nagata}{}. Furthermore, we call the pair $ ( M , 0 ) $ for any non-zero $ \Ga $-bimodule $ M $ the \emph{trivial multiplicative bimodule induced by $ M $}.
	
	Taking a split extension of a ring is the inverse operation to taking a split quotient, as one would expect. Indeed, for any ring cleft extension $(\La , \Ga , \pi , \i )$, ring $\La $ is isomorphic to the ring $E = \Ga \ltimes_\theta \Ker \pi $, where the kernel of $ \pi $ is viewed as a $\Ga $-bimodule through restriction of scalars along $ \i $, and the map $\theta \colon \Ker \pi \otimes_\Ga \Ker \pi \to \Ker \pi $ is the one induced by multiplication in $\La $. Specifically, an isomorphism $\La \isomto E $ is given by $a \mapsto ( \pi(a) , a - \i \pi(a) )$ with inverse $(\g , x ) \mapsto \i (\g ) + x $, for every $a \in \La $, every $\g \in \Ga $ and $x \in \Ker \pi $. We call this map the \emph{standard isomorphism} between $\La $ and $ E $. Furthermore, we call the multiplicative bimodule $( \Ker \pi , \theta )$ the \emph{induced multiplicative bimodule of the ring cleft extension $(\La , \Ga , \pi , \i )$}.

	Conversely, for an arbitrary multiplicative bimodule $ ( M , \theta ) $ over a ring $ \Ga $, the split extension ring $E = \Ga \ltimes_\theta M$ naturally gives rise to a ring cleft extension $( E , \Ga , \pi, \i )$. Specifically, the map $\i \colon \Ga \monicc E$ is the \emph{inclusion monomorphism} sending $\g \in \Ga$ to $ ( \g , 0 ) $, and $\pi \colon E \epic \Ga $ is the \emph{projection epimorphism} mapping $( \g , m ) \in E$ to $\g$. We denote the kernel of $ \pi $ by $ \wt{ M } = 0 \oplus M$ throughout the rest of this section.

	Now let us restrict our attention to Artin algebras over a commutative artinian ring $k$. Then the notion of an algebra cleft extension $(\La , \Ga , \pi , \i )$ (see \cref{sec:rad.pres.clefts}{}), or equivalently the notion of a split algebra quotient of $\La $, correspond to multiplicative $ \Ga $-bimodules $(M, \theta )$ such that $k$ acts centrally on $M$ making it a finitely generated module. In fact, this extra condition is equivalent to requiring that the split extension ring $E = \Ga \ltimes_\theta M$ is an Artin algebra with the natural $ k $-algebra structure induced by the ring homomorphism $ k \to E $ sending $ c $ to $ ( c \cdot 1_\Ga , 0) $.

	For the rest of this section, we use the convention that a bimodule $ M $ over an Artin $ k $-algebra $ \Ga $ is such that \emph{$ k $ acts centrally on $ M $ making it a finitely generated module}. Consequently, for a multiplicative $ \Ga $-bimodule $(M, \theta )$, we call the ring $ E = \Ga \ltimes_\theta M $ a split extension \emph{algebra} of $\Ga $ to indicate that $ E $ is equipped with the canonical $ k $-structure discussed above.

	We begin by characterizing when a multiplicative bimodule $(M , \theta )$ for a bound quiver algebra $\Ga $ makes $\Ga$ an arrow pre-removal of $E = \Ga \ltimes_\theta M$ (see \cref{defn:ar.rem.4}{}).
	To do this, we use the notion of nilpotency from \cite{Pierce}{}. For a multiplicative bimodule $(M , \theta ) $ over a ring $\Ga $, define recursively $\theta ^r \colon M^{\otimes_\Ga ( r + 1 )} \to M$ by $\theta^0 = \id_M$ and $ \theta^r = \theta \circ ( { \theta^{r-1} \otimes 1_M } ) $ for every $r \geq 1 $.
	The bimodule $( M , \theta )$ is then \emph{nilpotent} if $\theta^N = 0 $ for some positive integer $ N $. This is clearly equivalent to $ \wt{ M } $ being a nilpotent ideal of $ E $ (with $ \wt{ M }^{N+1} = 0 $).

	\begin{defn}
		\label{defn:ar.rem.4}
		Let $\Ga $ and $\La $ be two finite dimensional algebras over a field $k$. Then $\Ga $ is an \emph{arrow pre-removal algebra of $\La $} if there is a representation $k Q_\La / I_\La $ of $\La $ as a bound quiver algebra and a pre-removable set of arrows $A$ in $k Q_\La / I_\La $ such that $\Ga $ is isomorphic to the quotient of $k Q_\La / I_\La $ over the ideal generated by $A$.
	\end{defn}

	The next technical lemma forms the basis for what is to follow.
	
	\begin{lem}
		\label{lem:ar.rem.8}
		Let $\Ga = k Q_\Ga / I_\Ga  $ be a bound quiver algebra, and let $(M , \theta )$ be a multiplicative $ \Ga $-bimodule. Then the split extension algebra $E = \Ga \ltimes_\theta M $ is isomorphic to a bound quiver algebra whose underlying quiver has the same number of vertices as $Q _\Ga $ if and only if $( M , \theta )$ is nilpotent. Furthermore, algebra $\Ga $ is an arrow pre-removal of $ E $ in that case.
	\end{lem}

	\begin{proof}
		Assume that $ Q _\Ga $ has $ n $ vertices, labeled $ 1, 2, \ldots, n $. Assume $ E $ is isomorphic to a bound quiver algebra with $n$ vertices. Then the set $\{ ( e_i , 0 ) \}_{i \in (Q_\Ga) _0 }$ is a complete set of primitive orthogonal idempotents for $ E $. Indeed, if $( e_i , 0 ) $ was not primitive for some $i \in ( Q_\Ga )_0 $, it would follow that $ E $ has more vertices than $\Ga $ since primitive idempotents correspond to vertices. In particular, the decomposition $ E  = {}_k \la (e_i , 0 ) \colon i \in (Q_\Ga )_0 \ra \oplus J( E ) $ holds since $ E $ is isomorphic to a bound quiver algebra, implying that $\dim _k E = n + \dim _k J ( E ) $. Considering also the analogous decomposition for $ \Ga $ yields $\dim _k J( E ) = \dim _k J (\Ga ) + \dim_k M$.
		
		On the other hand, the quotient $E / ( J(\Ga) \oplus M )$, with $ \Ga $-structure induced by the fact that $\wt{M} (E / (J(\Ga) \oplus M)) = 0$, is isomorphic to ${}_\Ga { \Ga / J (\Ga ) } $. The fact that $E / ( J(\Ga) \oplus M )$ is semisimple over $ \Ga $ implies the same property over $ E $. In particular, the inclusion $J(E) \subseteq J(\Ga) \oplus M$ holds. We deduce that $J ( E ) = J (\Ga ) \oplus M$ due to the equality of $k$-dimensions for the two subspaces. Furthermore, the ideal $ \wt{M} $ of $ E $ is nilpotent as $J ( E )$ is nilpotent. Therefore, there exists a minimal integer $N > 0 $ such that $ \wt{ M }^{N+1} = 0$, which implies $\theta^N = 0 $.

		Conversely, assume $( M , \theta )$ is a nilpotent multiplicative $ \Ga $-bimodule, with $ N $ being the minimal positive integer such that $\theta ^N = 0$. Note that $\theta $ cannot be surjective, as nilpotency would then imply $M = 0 $. More generally, the equality $ \Image \,  \theta^{  i + 1 }  = \Image \,  \theta ^{ i }  $ holds if and only if $i \geq N $. It follows that the chain of $\Ga$-bimodules
			\[
				0 = \Image \,  \theta ^N  \leq \Image \,  \theta ^{  N-1  }  \leq \ldots \leq \Image \, \theta \leq M
			\]
		is strictly decreasing as the power of $ \theta $ increases.
		
		Now fix a projective cover
		\[
		f \colon \oplus_{\k = 1}^t ( \Ga e_{i_\k} \otimes_k e_{j_\k} \Ga ) \epic M
		\]
		of $M$ as a $\Ga $-bimodule. Since $M$ is generated by $m_\k = f( e_{i_\k} \otimes e_{j_\k} ) $ for $\k = 1, \ldots, t$ as a $\Ga$-bimodule and $ \theta $ is not surjective, we may assume that there is an integer $t_0$ such that $1 \leq t_0 \leq t$ and $m_\k  \in \Image \, \theta $ if and only if $\k > t_0 $.
		Let $Q_E = Q_\Ga \, \dot \cup A$ where $A =  \, \{ \a_\k \colon i_\k \to j_\k \, /  \, \k = 1, 2 , \ldots, t_0 \} $, i.e.\ the quiver $Q_E $ results from $Q_\Ga $ if we add the $t_0$ new arrows $\a_\k \colon i_{\k } \to j_\k $. Note that the arrows to be added are determined by the indecomposable direct summands of the projective cover of $M$ whose generator does not map into the image of $ \theta $ under $ f $. Furthermore, isomorphic but different such summands give rise to (distinct) parallel arrows. We show next that $ E $ is isomorphic to an admissible quotient of $k Q_E $.
		
		Let $\phi \colon k Q_E \to E $ be the unique $k$-algebra homomorphism defined by $e_i \mapsto (e_i , 0 )$ for every vertex $i $ of $ Q_\Ga $, and by $\b \mapsto (\b + I_\Ga, 0 )$ for every arrow $\b $ of $ Q_\Ga $ and $\a _\k \mapsto (0 , m_\k) $ for every new arrow $ \a_\k \in A $. To see that $ \phi $ is well-defined, note for instance that $(e_{i_\k} , 0) ( 0 , m_\k ) ( e_{j_\k} , 0 ) = ( 0 , m_\k )$; see also \cite[Theorem~II.1.8]{ASS}{}. It is straightforward to verify that $\Ga \oplus 0 \subseteq \Image \, \phi$, since $\phi (p) = (p + I_\Ga , 0 )$ for every path $p$ in $Q_\Ga $. Therefore, it remains to show that $( 0 , m_\k ) \in \Image \, \phi $ for all $\k$ in order for $\phi $ to be surjective, given that $ \{ m_\k \}_{\k =1 }^t$ generates $M$ as a $\Ga $-bimodule. Indeed, for any paths $p$ and $q $ in $Q_\Ga $ such that $t (p ) = i_\k $ and $s( q ) = j_\k $ for some index $ \k $, we have $( 0 , ( p + I_\Ga ) m_\k  ( q + I_\Ga ) ) = ( p + I_\Ga , 0 )  ( 0, m_\k )  ( q + I_\Ga , 0 )$. Note that the elements $(0 , m_\k ) $ are in $ \Image \, \phi $ for all $\k \leq t_0 $ as $\phi ( \a_\mu ) = ( 0 , m_\k ) $ by definition. Therefore, it suffices to show that $ 0 \oplus \Image \, \theta \subseteq \Image \, \phi $ since $ ( 0 , m_\k ) \in 0 \oplus \Image \, \theta $ for every $\k > t_0 $.

		In order to achieve the above goal, we establish a convenient description for the elements of $\Image \, \theta $. First, observe that every element in $\Image \,  \theta ^{ N - 1 }  $ is a sum of terms of the form $m = \theta^{ N-1 } ( x_0 \otimes x_1 \otimes \ldots \otimes x_{N-1} )$, where each $x_\nu \in M$. If $x_\nu \in \Image \, \theta $ for some $ \nu $ then $m = 0$, and we may therefore assume that all $ x_\nu $ belong to the $\Ga $-subbimodule of $M$ generated by $m_\k $ for $\k \leq t_0 $. It follows that $\Image \,  \theta ^{ N - 1 }  $ is generated as a $k$-vector space by the elements of the form
		\phantomsection		\hypertarget{eq:2'}{}	
		\begin{equation}	
			\label{eq:2}
			z = \theta^{r} (  \big[ ( p_0 + I_\Ga )  m_{\k_0}  ( p_1 + I_\Ga ) \big] \otimes \big[ m_{\k_1}  ( p_ 2 + I_\Ga ) \big] \otimes \ldots \otimes \big[ m_{\k_{r} }  ( p_{ r + 1 } + I_\Ga ) \big] )
		\end{equation}
		where $r = N-1$; the indices $\k_\xi $ are in $ \{ 1, 2, \ldots, t_0  \}$ for every $\xi \in \{ 0 , 1 , \ldots , r \} $, and $p_\xi $ are non-zero paths in $\Ga$ such that $t (p_\xi ) = i_{\k_\xi}$ for every $\xi \in \{ 0 , 1 , \ldots, r \} $ and $ s( p_\xi ) = j_{\k_{\xi-1}} $ for every $\xi \in  \{ 1, 2, \ldots, r+1  \}$. Similarly, every element in $\Image \,  \theta ^{ N-2 }  $ is a sum of elements of the form $m = \theta^{ N-2 } ( x_0 \otimes x_1 \otimes \ldots \otimes x_{N-2} )$ where $x_\nu \in M$. If $x_\nu \in \Image \, \theta$ for some $ \nu $ then $ m \in \Image \,  \theta ^{ N-1 }  $, and such elements are already accounted for. Therefore, the subspace $\Image \,  \theta ^{ N-2 }  $ is generated by all the elements of the form
		\hyperlink{eq:2'}{(\ref{eq:2}{})} where $r \in \{ N-2 , N-1 \}$. Inductively, every element in $ \Image \, \theta $ is a $k$-linear combination of elements of the form \hyperlink{eq:2'}{(\ref{eq:2}{})} where $r \in \{ 1, 2 , \ldots, N-1  \}$. This completes the proof of our claim that $0 \oplus \Image \, \theta\subseteq \Image \, \phi $, since
		\begin{multline*}
			\phi ( p_0 \a_{\k_0} p_1  \ldots p_r \a_{\k_r} p_{r+1}   ) =		\\
			( p_0 + I_\Ga , 0 )  ( 0 , m_{\k_0} )  ( p_1 + I_\Ga , 0 )   \ldots ( p_r + I_\Ga ) ( 0 , m_{\k_r } )  ( p_{r+1} + I_\Ga , 0 ) 	
			= ( 0 , z )
		\end{multline*}
		implying that all elements of the form \hyperlink{eq:2'}{(\ref{eq:2}{})} are contained in the image of $ \phi $ for any $ r $ between $ 1 $ and $ N - 1 $. We deduce that $ \phi $ is surjective.
		
		Note that we have also proven that $ \phi ( {}_k \la \mathcal{B}_{\! {Q_E}}^{ A} \ra ) = \wt{ M } $.

		We show next that the kernel of $\phi $, denoted by $I_E$, is an admissible ideal of $k Q_E $, and that the set $A$ is pre-removable in $k Q_E / I_E $. Recall that $A$ is pre-removable in $k Q_E / I_E $ if for any $z \in I_E$, the elements $z_A $ and $z_{\nott \! A } $ belong to $ I_E $, where $z = z_A + z_{\nott \! A}$ is the unique decomposition of $ z $ as a sum of elements from ${}_k \la \mathcal{B}_{\! {Q_E}}^{ A} \ra$ and ${}_k \la \mathcal{B}_{{Q_E}}^{\, \nott \!\! A} \ra$, respectively.
		But we have already shown that $\phi (z_A) \in \wt{ M } $ and $\phi (z_{\nott \! A } ) = ( z_{\nott \! A } + I_\Ga , 0 )$ for every $z \in k Q_E $.
		Therefore, it holds that $z \in I_E $ if and only if both $z_A $ and $z_{\nott \! A} $ are in $ I_E $. In particular, we have that $ I_E $ decomposes as a direct sum $ I_E = I_\Ga \oplus I_A $, where $ I_A $ consists of elements $ z = z_A $ that are in $ I_E $; see also \hyperlink{lem:ar.rem.2}{\cref{lem:ar.rem.2}{}}{}.
		
		As for the admissibility of $ I_E $, observe that any path occurring in $ I_E $ has length at least two. Furthermore, every path in $Q_E $ of length at least $ (N+1) \Ll (\Ga ) $ is in $I_E$, where $\Ll(\Ga )$ is the Loewy length of $\Ga $. Indeed, such a path either contains a subpath of length $ \Ll(\Ga ) $ that is in $ Q_\Ga $ or it is divided by some arrow in $ A $ at least $ N + 1 $ times. In both cases, it follows that the path is in $ I_E $, in the latter case due to the fact that $\theta^N = 0 $.
		
		We conclude that $ E $ is isomorphic to the bound quiver algebra $k Q_E / I_E $, where the set of arrows $A$ is pre-removable.
		Moreover, the induced isomorphism from $ k Q_E / I_E $ to $ E $ maps the ideal $\la A + I_E \ra $ onto $ \wt{ M } $. Consequently, algebra $\Ga $ is an arrow pre-removal of $ E $, since it is isomorphic to the quotient $ ( { k Q_E / I_E } ) / \la A + I_E \ra $. 
	\end{proof}
	
	The second part of the following corollary may be found in \cite[Subsection~1.2]{ACT}{}, in a different but equivalent form.
	
	\begin{cor}
			\label{rem:ar.rem.1}		\hypertarget{rem:ar.rem.1}
		Let $(M, \theta )$ be a nilpotent multiplicative bimodule for a bound quiver algebra $\Ga $, and let $ E = { \Ga \ltimes_\theta M } $. If $ N $ is the minimal positive integer such that $\theta^N = 0$, then
			\[
				\Ll(\Ga ) \leq \, \Ll( E ) \, \leq (N+1) \Ll(\Ga) .
			\]
		Furthermore, there is a bound quiver algebra $ k Q_E / I_E $ isomorphic to $E$ such that $ \wt{ M } = 0 \oplus M $ corresponds to an ideal generated by a pre-removable set of arrows.
	\end{cor}
	
	\begin{proof}
		The corollary follows directly from the proof of \cref{lem:ar.rem.8}{}.
	\end{proof}

	\begin{exam}
		\label{exam:ar.rem.6}
		Let $A = k Q_A / I_A $ and $B = k Q_B / I_B $ be two bound quiver algebras, and let $M$ be an $A$-$B$-bimodule where $k$ acts centrally making it a finite-dimensional vector space. We claim that:
			\begin{enumerate}[\rm(i)]
				\item the algebra $ C= A \times B$ is an arrow pre-removal of the triangular matrix algebra $\La = \big(\begin{smallmatrix}
					A & M\\
					0 & B
				\end{smallmatrix}\big)$;
			\item the quiver of $\La $ is not strongly connected, i.e.\ there are distinct vertices $ i , j $ in the ordinary quiver of $ \La $ such that there is no path in the quiver with source $ i $ and target $ j $.
			\end{enumerate}
		
		To see the first claim, observe that the algebra $ C $ is isomorphic to the bound quiver algebra $k Q_C / I_C $ where $Q_C $ is the disjoint union of $Q_A $ and $Q_B$, and $I_C$ is the ideal of the path algebra $k Q_C$ generated by $I_A $ and $I_B$. Furthermore, the algebra $\La $ is isomorphic to the trivial extension $ C \ltimes M $ where $M$ is considered as a $ C $-bimodule through the actions $(a, b) m = a m $ and $m (a , b ) = m b $ for every $a \in A$, every $b \in B$ and $m\in M$. Therefore, the algebra $\La $ is isomorphic to a bound quiver algebra with underlying quiver equal to the disjoint union of $Q_A$ and $Q_B$, with extra arrows that are determined by the indecomposable projective $ C $-bimodules occurring as direct summands of the projective cover of $M$; see the proof of \cref{lem:ar.rem.8}{} for more details.

		For the second claim, note that $M e_A = 0$ and $e_B M = 0$ for the idempotents $e_A = ( 1_A , 0 ) $ and $e_B = ( 0 , 1_B ) $.
		Consequently, every indecomposable projective $C$-bimodule that occurs as a direct summand of the projective cover of $M$
		is of the form $ C ( e_i^A , 0 ) \otimes_k ( 0 , e_j^B ) C $, where $e_i^A$ and $ e_j^B$ are trivial paths of $A $ and $B$ respectively. We conclude that all arrows of $\La $ that are not in $Q_A $ or $Q_B $ have source in $Q_A $ and target in $Q_B$. In particular, there is no path in the quiver of $\La$ with source in $ Q_B $ and target in $ Q_A $.
	\end{exam}

	It is well-known that an algebra over a field is isomorphic to a bound quiver algebra if and only if it is basic and all simple modules are one-dimensional, with the latter condition being redundant if the base field is algebraically closed; see \cite[Corollary~III.1.10]{repth}{}. If $\La $ is such an algebra over an arbitrary field $k$, then we say that \emph{$k Q_\La / I_\La $ is a representation of $\La $ as a bound quiver algebra} when $ \La \simeq k Q_\La / I_\La $.

	We characterize next split algebra quotients over superfluous ideals in terms of pre-removable sets of arrows for the above class of algebras.

	\begin{prop}[\mbox{cf.\ \cite[Subsection~1.4]{ACT}{}}]
		\label{prop:ar.rem.3}
		The following are equivalent for two basic algebras $\Ga $ and $\La $ with one-dimensional simple modules.
		\begin{enumerate}[\rm(i)]
			\item Algebra $\Ga $ is an arrow pre-removal of $\La $.
			\item Algebra $\Ga $ is isomorphic to a split quotient of $\La $ over a superfluous ideal.
			\item Algebra $\La $ is isomorphic to a split extension of $\Ga $ by a nilpotent multiplicative bimodule.
		\end{enumerate}
	\end{prop}

	\begin{proof}
		Condition (ii) follows directly from condition (i) by the definitions. Condition (iii) follows from condition (ii) by the discussion at the beginning of this subsection and the fact that every ideal contained in the Jacobson radical of a finite dimensional algebra is nilpotent. Finally, condition (i) follows from condition (iii) by \cref{lem:ar.rem.8}{}.
	\end{proof}
	
	\begin{cor}
		\label{cor:ar.rem.4}
		For a basic algebra $ \La $ with one-dimensional simple modules, the natural algebra epimorphism induced by an ideal contained in $J(\La)^2$ does not split.
	\end{cor}
	
	\begin{proof}
		Let $\La$ be a basic algebra over a field $k$ with one-dimensional simple modules. Let $K \subseteq J( \La ) $ be an ideal of $\La $ such that there is an algebra monomorphism $ \i \colon  \Ga \monicc \La $ with $\pi \i = \id_\Ga $, where $\Ga = \La / K $ and $\pi \colon \La \epic \Ga $ is the natural epimorphism.
		Furthermore, consider the split extension algebra $ E = \Ga \ltimes_\theta K $, where $ ( K , \theta )$ is the multiplicative bimodule induced by the algebra cleft extension $( \La , \Ga , \pi , \i )$.
		We claim that there is a representation of $\La $ as a bound quiver algebra such that $K$ corresponds to the ideal generated by a pre-removable set of arrows.
		
		Observe that $ \Ga $ is also a basic algebra with one-dimensional simple modules. Indeed, its basicness comes from \cref{lem:semiperfect.clefts}{} and the basicness of $ \La $, and the simple $ \Ga $-modules are one-dimensional as they can be obtained from the simple $ \La $-modules via restriction of scalars along $ \i $. In particular, algebra $ \Ga $ is isomorphic to a bound quiver algebra.
		Our claim follows now from the fact that the standard isomorphism between $ \La $ and $ E $ maps $ K $ onto $ \wt{ K } $, and the second part of \hyperlink{rem:ar.rem.1}{\cref{rem:ar.rem.1}{}}{}.
		Let $ k Q_\La / I_\La $ be a representation of $ \La $ as a bound quiver algebra with a pre-removable set of arrows $ A $ such that the ideal $ K $ corresponds to the ideal $ \la A + I \ra $.
		
		If we assume that $ K \subseteq J ( \La )^2 $, then the ideal $ \la A + I \ra $ is contained in the square of the Jacobson radical of $ k Q_\La / I_\La $, a contradiction since arrows are not contained in the square of the Jacobson radical of a bound quiver algebra.
	\end{proof}


	\subsection{Removable and perfect multiplicative bimodules}		\label{subsec:adding.ar.2}
	
	This subsection establishes conditions on a multiplicative bimodule $ ( M , \theta ) $ over a bound quiver algebra $ \Ga $ such that $ \Ga $ is a generalized arrow removal of the split extension algebra $ E = \Ga \ltimes_\theta M $. The interest in such conditions stems from the fact that the finiteness of any of the three homological dimensions of $ \Ga $ studied in this paper (little/big finitistic and global) implies the finiteness of the respective dimension of $ E $.
	Specifically, the first main result of the subsection (\hyperlink{thm:main.ar.rem.3}{\cref{thm:main.ar.rem.3}{}}{}) establishes necessary and sufficient conditions, and may also be viewed as a characterization of generalized arrow removal algebras in the spirit of \cite[Proposition~4.5]{arrowrem1}{}.
	Note that the term `generalized arrow removal' is used in a broader sense here (cf.\ \cref{defn:ar.rem.1}{}), since algebras are allowed to be represented by any isomorphic bound quiver algebra.
	
	In the second main result (\cref{cor:ar.rem.6}{}), we show that tensor nilpotent bimodules that are perfect (or only right perfect if the extension is trivial) are equipped with the above properties independently of the chosen multiplicative structure. A non-zero bimodule $ M $ over a ring $ \Ga $ is called \emph{tensor nilpotent} if $ M^{\otimes_\Ga i } = 0 $ for some integer $ i > 1 $; see also \cref{defn:ar.rem.5}{} for (right) perfectness.
	For recent results analogous to \cref{cor:ar.rem.6}{}, we refer the reader to \cite[Corollary 6.14]{PanosPsarou}{} in the context of injective generation for derived module categories, and to \cite[Proposition 5.8]{Panos}{} for Iwanaga-Gorensteiness of noetherian rings.



	Before presenting our results, we make some useful observations about the bimodule $ \wt{M} = { 0 \oplus M } $ of a split extension $E = \Ga \ltimes_\theta M$ over a ring $\Ga $. First, it is easy to verify that if we identify $ \wt{ M } $ with $ M $ as abelian groups, the resulting $ E $-bimodule structure on $  M  $ is given by $m \ast (\g_1 , m_1) = m  \g_1 + \theta ( m \otimes m_1)$ and $ (\g_1 , m_1) \ast m  = \g_1  m  + \theta ( m_1 \otimes m)$ for all $ m , m_1 \in M $ and $ \g_1 \in \Ga $. Second, the $E$-structures on $\wt{M}$ are generally different from the $E$-structures on $M$ via restriction of scalars along the projection epimorphism $ \pi \colon E \epic \Ga $. In fact, the two left (or right) $E$-structures coincide if and only if $\theta=0$, that is if $E$ is a trivial extension of $\Ga$. Finally, the $\Ga$-bimodule structure of $ \wt{ M } $ through restriction of scalars along the inclusion monomorphism $ \i \colon \Ga \monicc E $ coincides with the original $ \Ga $-bimodule structure of $M$.
	
	Retaining the above notation, we introduce the core concept of this subsection.
	
	
	\begin{defn}
			\label{defn:ar.rem.7}
		Let $ \Ga $ be an Artin algebra, and let $ E = \Ga \ltimes_\theta M $ for a nilpotent multiplicative bimodule $(M, \theta )$. Then $ (M, \theta ) $ is:
		\begin{enumerate}[\rm(i)]
			\item \emph{two-sided removable} if both $\pd \wt{M}_E $ and $\pd_E \wt{M}  $ are finite;
			\item \emph{only left removable} if $\pd \wt{M}_E $ is finite but $\pd_E \wt{M} $ is infinite, and $\theta = 0$;
			\item \emph{(left) removable} if it is two-sided or only left removable.
		\end{enumerate}
	\end{defn}

	As discussed at the beginning of \cref{subsec:adding.ar.1}{}, if $ \Ga $ is an Artin algebra and $(M, \theta )$ is a multiplicative $ \Ga $-bimodule (where by assumption the base commutative artinian ring acts centrally on $ M $ making it a finitely generated module), then the ring $ \Ga \ltimes_\theta M $ is equipped with a canonical Artin algebra structure. If we assume that $(M, \theta )$ is nilpotent, then the ideal $ \wt{ M } $ is nilpotent and therefore contained in $ J ( \Ga \ltimes_\theta M ) $. If we furthermore assume that $(M, \theta )$ is removable, then it holds that
		\[
			\fpd { \Ga \ltimes_\theta M } < \infty  \, \, \iff  \, \,  \fpd \Ga < \infty 
		\]
	by an application of \cref{thm:main.III}{} for $ \La = \Ga \ltimes_\theta M $ and $ K = \wt{ M } $, as $ \Ga \simeq ( \Ga \ltimes_\theta M ) / \wt{ M } $ via the projection epimorphism. Moreover, the analogous equivalences are valid for the big finitistic and global dimensions of the algebras.
	
	The next example illustrates \cref{defn:ar.rem.7}{} in the classic context of triangular matrix algebras.

	\begin{exam}[cf.\!\mbox{\cite[Corollary 4.21]{FGR}{}}]
			\label{exam:2}
		Let $ A $ and $ B $ be two Artin $ k $-algebras, and let $ M $ be an $ A $-$ B $-bimodule where $ k $ acts centrally making it a finitely generated module. Recall that the associated triangular matrix algebra
		$\La = \big(\begin{smallmatrix}
			A & M\\
			0 & B
		\end{smallmatrix}\big)$
		can be identified with the trivial extension ${ C } \ltimes M $ where $ C = A \times B $ and $ M $ is endowed with the $ C $-bimodule structure defined by $ ( a ,b )m = a m$ and $ m ( a , b ) = m b $ for all $ a \in A $, all $ b \in B $ and $ m \in  M $. It holds that $ C $ has a natural Artin algebra structure inherited from the $ k $-structures of $ A $ and $ B $. Furthermore, we have that $ \La $ is an Artin algebra with respect to the canonical $ k $-structure discussed at the beginning of \cref{subsec:adding.ar.1}{}, as $ k $ acts centrally on $ M $ when the latter is viewed as the above $ C $-bimodule, and the induced $ k $-structure is the same as the initial one.
		
		
		It is clear that $\pd \wt{ M }_\La = \pd M_B ( = \pd M_C ) $ since $ \wt{M} = \wt{M} e_B $ and $ e_B \La e_A = 0 $ for the idempotents
		$e_A = \big(\begin{smallmatrix}
			1 & 0 \\
			0 & 0
		\end{smallmatrix}\big)
		$
		and
		$e_B = \big(\begin{smallmatrix}
			0 & 0 \\
			0 & 1
		\end{smallmatrix}\big)
		$. Therefore, the trivial multiplicative $ C $-bimodule induced by $ M $ is removable if and only if $\pd M_B < \infty $. In particular, using \cref{thm:main.III}{}, we recover the well-known fact that $\fpd \La < \infty $ if and only if $\fpd A < \infty $ and $\fpd B < \infty $ under the assumption $ \pd M_B < \infty $, as well as the analogous equivalences for $\Fpd $ and $\gd $.
	\end{exam}

		We proceed with the first main result of the subsection, which is a characterization of generalized arrow removal algebras.
	
	\begin{thm}		\label{thm:main.ar.rem.3}		\hypertarget{thm:main.ar.rem.3}
		Let $\Ga $ and $\La $ be two basic finite dimensional algebras with one-dimensional modules over a field $ k $. Then $\Ga $ is a generalized arrow removal of $\La$ if and only if $\La $ occurs as a split extension of $\Ga $ by a removable multiplicative bimodule.
	\end{thm}
	
	\begin{proof}
		Let $\Ga $ be a generalized arrow removal of $\La$. Then $\La $ is isomorphic to a bound quiver algebra $k Q_\La / I_\La $ such that $\Ga $ is isomorphic to the quotient of $ { k Q_\La / I_\La } $ over the ideal generated by a removable set of arrows $A$.
		If we fix such an isomorphism and a section algebra monomorphism for the natural epimorphism induced by the ideal $  \la A + I_\La \ra $, then we obtain an algebra cleft extension of the form $( { k Q_\La / I_\La } , \Ga , \i' ,\pi' )$.
		Let $ ( M , \theta  ) $ be the multiplicative $ \Ga $-bimodule induced by this cleft extension, that is $M = \la A + I_\La \ra $ and $ \theta $ is induced by multiplication in $ k Q_\La / I_\La $. Then the standard isomorphism between $k Q_\La / I_\La $ and $E = \Ga \ltimes _\theta M$ maps the ideal $\la A + I_\La \ra $ onto $ \wt{ M } $.
		In particular, it holds that $\pd \wt{M} _E = \pd { \la A + I_\La \ra }_{ k Q_\La / I_\La } $ is finite since $A$ is removable. Similarly, the dimension $\pd_E \wt{M} = \pd_{k Q_\La / I_\La }  { \la A + I_\La \ra } $ is finite if $A$ is two-sided removable. Otherwise, the set $ A $ is only left removable, in which case
			\[
				( 0 \oplus { \Image \, \theta } ) = {\wt{ M }}^2 \simeq \la A + I_\La \ra^2 = 0
			\]
		implying $ \theta = 0 $. Therefore, the multiplicative bimodule $(M, \theta )$ is removable in any case, and the result follows since $\La \simeq E$.

		Conversely, let $ ( M , \theta) $ be a removable multiplicative  $\Ga $-bimodule such that $\La $ is isomorphic to the split extension algebra $E = \Ga \ltimes_\theta M$. Then it follows from \cref{lem:ar.rem.8}{} and \hyperlink{rem:ar.rem.1}{\cref{rem:ar.rem.1}{}} that there is a representation $k Q_E / I_E $ of $ E $ as a bound quiver algebra such that the ideal $ \wt{ M } $ corresponds to the ideal $ \la A + I_E \ra $ for a pre-removable set of arrows $ A $. In particular, the algebra $\Ga $ is isomorphic to the quotient $ ( k Q_E / I_E ) / \la A + I_E \ra $. Moreover, the set $A$ is removable in $k Q_E / I_E $  because $ ( M , \theta ) $ is a removable multiplicative bimodule. Specifically, the dimension $\pd \la A + I_E \ra _{ k Q_E / I_E } $ is finite as it is equal to $ \pd \wt{M} _E $. Furthermore, if $ ( M , \theta ) $ is two-sided removable then the dimension $\pd _{ k Q_E / I_E } \la A + I_E \ra  = \pd _E \wt{M} $ is also finite and $ A $ is two-sided removable. Otherwise, the multiplicative bimodule $ ( M , \theta ) $ is only left removable implying that the ideal $ \la A + I_E \ra^2 $ is trivial as $ \theta $ is the zero map in this case. The proof is complete as $ \La \simeq k Q_E / I_E $.
	\end{proof}

		In \cref{exam:2}{}, we saw an instance of a trivial extension such that the finiteness of the right projective dimension of the bimodule over the initial algebra ensures removability. However, a nilpotent multiplicative bimodule for an algebra $\Ga $ fails in general to be removable if we only impose conditions on the underlying bimodule over $\Ga$. We illustrate this fact with an example.

	\begin{exam}
		\label{exam:ar.rem.7}
		A nilpotent multiplicative bimodule $ ( M , \theta ) $ over a semisimple Artin algebra $ \Ga $ is removable if and only if its associated split extension algebra $ E = \Ga \ltimes_\theta M $ has finite global dimension. Indeed, if $ ( M , \theta ) $ is removable, then $ \gd E < \infty $ since $ \gd \Ga = 0 < \infty $; see also the discussion after \cref{defn:ar.rem.7}{}. Conversely, if $ \gd E < \infty $, then $ ( M , \theta ) $ is two-sided removable as the ideal $ \wt{ M } $ has finite projective dimension on both sides.
	\end{exam}

	We continue this subsection by establishing sufficient conditions for a bimodule $ M $ over an Artin algebra $ \Ga $ ensuring that $ M $ equipped with any multiplicative structure is removable. The following notion is central to this end.

	\begin{defn}[\mbox{\!\!\cite[Definition 4.4]{perfectbimods}{}, \cite[Definition 5.1]{PanosPsarou}{}}]
		\label{defn:ar.rem.5}
		A bimodule $M$ over a ring $\Ga$ is \emph{right perfect} if it satisfies the following conditions:
		\begin{enumerate}[\rm(i)]
			\item	$\Tor_i^\Ga ( M^{ \otimes _\Ga j } , M ) = 0 $ for every $i , j \geq 1 $;
			\item	$\pd M_\Ga < \infty$.
		\end{enumerate}
		Furthermore, left perfectness is defined analogously, and a bimodule is \emph{perfect} if it is left and right perfect.
	\end{defn}

		\begin{rem}
			Condition (i) in \cref{defn:ar.rem.5}{} already appears in \cite{FGR}{} in the context of trivial extensions of abelian categories; see also \cite{Beli2}{}, where it is considered in the more general context of abelian category cleft extensions. Furthermore, condition (i) is equivalent to its dual; see \cite[Corollary 4.3]{perfectbimods}{}. 
		\end{rem}

	The following lemma, which is essentially \cite[Proposition~6.6]{PanosPsarou}{}, showcases the good homological behavior of nilpotent right perfect bimodules with respect to split extensions. We include a simpler proof of the lemma for the sake of completeness, after introducing the necessary notation.
	
	Let $E = \Ga \ltimes_\theta M$ be a split extension of the ring $\Ga$. We write $\er = ( - )_\Ga $ and $\lr = - \otimes_\Ga E $ for the restriction and induction functors, respectively, along the inclusion monomorphism $ \i \colon \Ga \monicc E $. Furthermore, we use $G$ to denote the endofunctor of the category of right $ E $-modules ($ E^{\textrm{op}} \rMod  $) defined as the kernel of the canonical counit $\e \colon \lr \er \to \id_{ E^{\textrm{op} } \rMod }$ of the adjoint pair $(\lr , \er )$. Recall that $\e_Y \colon Y \otimes_\Ga E \to Y$ is induced by the right $E$-action on $Y$ for every $Y \in E^{\textrm{op}} \rMod  $, that is $\e_Y(y \otimes a)= y a $ for every $y \in Y $ and $a \in E $.
	

	\begin{lem}
		\label{lem:ar.rem.3}
		Let $(M, \theta)$ be a multiplicative bimodule for a ring $\Ga$ and $ E = \Ga \ltimes_\theta M $. If $M$ is tensor nilpotent right perfect over $ \Ga $, then $ \pd \widetilde{M}_E < \infty $. 
	\end{lem}
	
	\begin{proof}
		As $ M $ is tensor nilpotent, there is a minimal positive integer $ N' $ such that $M^{\otimes _\Ga ( N' +1 ) } = 0  $. We begin by showing that $\pd ( M^{ \otimes _\Ga j } \otimes_\Ga E )_E \leq  j \pd M_\Ga $ for every $ j = 1 , 2 , \ldots , N' $.
		First, we have that $\pd M^{ \otimes _\Ga j }_\Ga \leq j  \pd M_\Ga $ due to \cite[Lemma 4.5]{perfectbimods}{}. Now let $\mathbb{P} $ denote a projective resolution of minimal length for $M^{ \otimes _\Ga j} $, viewed as a right $\Ga$-module. Then the complex $\mathbb{P} \otimes_\Ga E $ is a projective resolution of $ M^{ \otimes _\Ga j} \otimes_\Ga E $. Indeed, we have
			\[
				\Tor_i^\Ga ( M^{ \otimes _\Ga j} , E ) \simeq \Tor_i^\Ga ( M^{ \otimes _\Ga j} , \Ga ) \oplus \Tor_i^\Ga ( M^{ \otimes _\Ga j} , M )
			\]
		as ${}_\Ga E \simeq {}_\Ga \Ga \oplus {}_\Ga M $, and all the above Tor-groups are zero for $ i \geq 1 $ since the regular module $ {}_\Ga \Ga $ is flat and $ M $ is right perfect; hence, the complex $\mathbb{P} \otimes_\Ga E $ is exact. Our claim follows now from the fact that tensor functors preserve projectivity, implying that the complex $\mathbb{P} \otimes_\Ga E $ is a projective resolution of the right $ E $-module $ M^{ \otimes _\Ga j } \otimes_\Ga E $ of length $\pd M^{ \otimes _\Ga j }_\Ga \leq j  \pd M_\Ga $.
		
		
		For the final claim, we begin by noting that there exists a short exact sequence 
		\[
		0 \to G^{j+1}( \widetilde{M} ) \to \lr \er  G^j ( \widetilde{ M } )  \to G^j ( \widetilde{ M } ) \to 0
		\]	
		in $ E^{ \mathrm{op} } \rMod $ for every integer $j \geq 0$ by the definition of the endofunctor $G$, where $G^0 = \id_{ E^{\mathrm{op } }\rMod }$. Splicing the above short exact sequences at their common ends we get a long exact sequence
		\[
		\ldots \to  \lr \er G^{ j + 1 }  ( \widetilde{ M } )  \to  \lr \er  G^j  ( \widetilde{ M } )  \to \ldots \to  \lr \er  G  ( \widetilde{ M } )  \to  \lr \er  ( \widetilde{ M } ) \to \widetilde{ M } \to 0.
		\]
		Let $ F $ denote the tensor endofunctor of $ \Ga^{\textrm{op}} \rMod $ induced by $M$. According to \cite[Lemma 2.4]{arrowrem1}{}, it holds that $ e G^j \simeq F^j e $ for every $ j \geq 0 $. Therefore, it holds that $ \lr \er G^{j} ( \widetilde{ M } ) \simeq \lr F^j \er ( \widetilde{ M } ) \simeq M^{ \otimes_\Ga (j +1 )} \otimes_\Ga E $ for every integer $ j \geq 0 $, where the second isomorphism follows from the fact that $ \er ( { \widetilde{ M } } ) \simeq M_\Ga $. In particular, it holds that $ \lr \er G^{j} ( \widetilde{ M } ) = 0 $ for every $j \geq N' $, and $ \pd \lr \er G^{j} ( \widetilde{ M } ) \leq  N' \pd M_\Ga $ for every other $ j \geq 0 $. All in all, we have
			\[
				\pd \widetilde{ M }_E \leq N' \pd M_\Ga + ( N' - 1 ) =  N' ( 1 + \pd M_\Ga ) - 1  < \infty 
			\]
		by a standard iterative argument on long exact sequences of Ext-functors, and our claim follows.
	\end{proof}

	\begin{cor}
		\label{cor:ar.rem.6}
		Let $(M, \theta)$ be a multiplicative bimodule for an Artin algebra $\Ga$, where $M$ is tensor nilpotent as a $ \Ga $-bimodule. If
		\begin{enumerate}[\rm(i)]
			\item the bimodule $M$ is perfect over $\Ga$, or
			\item the bimodule $M$ is right perfect over $\Ga$ and $\theta = 0$,
		\end{enumerate}
		then $ ( M , \theta ) $ is removable.
		In particular, it holds that
		\[
		\fpd { \Ga \ltimes_\theta M } < \infty  \, \, \iff  \, \,  \fpd \Ga < \infty
		\]
		and the equivalence remains valid if $\fpd $ is replaced by $\Fpd $ or $\gd $.
	\end{cor}

	\begin{proof}
		The theorem is a direct application of \cref{lem:ar.rem.3}{} and its dual, and \cref{thm:main.III}{} for the ideal $ K = \wt{ M } $. See also the discussion after \cref{defn:ar.rem.7}{}.
	\end{proof}


	The next example shows that the condition of tensor nilpotency is irredundant in \cref{cor:ar.rem.6}{} and \cref{lem:ar.rem.3}{}.
	
	\begin{exam}[\mbox{cf.\ \cite[Corollary~7.12]{Beligiannis}{}}]
		\label{exam:ar.rem.8}	
		Let $ \La = k Q / I $ be a bound quiver algebra, and let $ \Ga = \La / J( \La ) $ be its semisimple part.
		Then the algebra monomorphism $ \i \colon \Ga \monicc \La $ defined by $\ol{e_i} \mapsto e_i $ for every vertex $i \in Q_0 $ is a section of the natural epimorphism $\pi \colon \La \epic \Ga $ (\cref{rem:ar.rem.2}{}). Therefore, algebra $ \La $ may be realized as the split extension $ E = \Ga \ltimes_\theta J ( \La ) $ where $ ( J( \La ) , \theta ) $ is the multiplicative bimodule induced by the cleft extension $ ( \La , \Ga , \pi , \i ) $, see the discussion at the beginning of \cref{subsec:adding.ar.1}{}. In particular, the ideal $ J( \La ) $ is viewed as a $ \Ga $-bimodule via restriction of scalars along $ \i $, and it is always perfect since $\Ga $ is semisimple.
		
		However, the induced bimodule $ ( J( \La ) , \theta ) $ is removable if and only if the global dimension of $ \La $ is finite, since the modules $ J ( \La )_\La $ and $ \wt{ J ( \La ) } _E $ are the same if we identify $ \La  $ and $ E $ through the standard isomorphism (see also \cref{exam:ar.rem.0}{}). Furthermore, the analogous assertion holds for the left side.
		We claim that $J(\La )$ is tensor nilpotent over $\Ga$ in an even more restrictive case, that is if and only if the underlying quiver of $ \La $ is acyclic.


		Let $\CC^{ \geq 1 }$ denote a set of non-trivial paths in $Q$ such that the set $\CC^{ \geq 1 } \dot \cup \{ e_i \}_{i \in Q_0} $ forms a $k$-basis of $\La $ modulo $I$. This is always possible since the set of all paths generates $ \La $ as a $ k $-vector space. Then
		\[
		J( \La  ) \simeq \bigoplus_{p \in \CC^{ \geq 1 }} \Ga ( s(p) , t(p) )
		\]
		where $ J (\La )$ is viewed as a $\Ga$-bimodule through restriction of scalars along $ \i $, and $ \Ga ( i , j ) $ denotes the projective indecomposable $ \Ga $-bimodule $ \Ga e_i \otimes_k e_j \Ga $ for all vertices $i , j \in Q_0 $.
		
		Restricting our attention to indecomposable projective $ \Ga $-bimodules, we have
		\[
		\Ga ( i_1 , j_1  ) \otimes_\Ga \Ga ( i_2 , j_2  ) \simeq \Ga e_{i_1}  \otimes_k ( e_{j_1} \Ga \otimes_\Ga  \Ga e_{i_2} ) \otimes_k e_{j_2} \Ga \simeq \Ga ( i_1 , j_2 )^{ \dim_k e_{ j_1 } \Ga e_{ i_2 } }
		\]
		for any vertices $i_\k , j_\k \in Q_0 $ and $\k =1, 2$, as the $ k $-vector space isomorphism $ e_{j_1} \Ga \otimes_\Ga  \Ga e_{i_2} \simeq e_{ j_1 } \Ga e_{ i_2 }$ holds. Here, we write $ \Ga ( i_1 , j_2 )^{ \dim_k e_{ j_1 } \Ga e_{ i_2 } } $ to denote the direct sum of $ { \dim_k e_{ j_1 } \Ga e_{ i_2 } } $ copies of $ \Ga ( i_1 , j_2 ) $.
		Furthermore, it holds that $ \dim_k e_{ j_1 } \Ga e_{ i_2 } = \d_{j_1 , i_2} $ as $\Ga$ is basic semisimple, where $ \d_{j_1 , i_2} $ is the Kronecker delta.
		In particular, we have
		\[
		J( \La  )^{ \otimes_\Ga l } \simeq \bigoplus_{q} \Ga ( s(q) , t(q) )
		\]
		for any positive integer $l$, where $q $ ranges over paths of the form $p_1 p_2 \ldots p_l $ for $ p_i $ in $ { \CC^{ \geq 1 } }$ implying that $ l ( q ) \geq l $.
		
		For the proof of our initial claim, let $\a_1 \a_2 \ldots \a_m $ be an oriented cycle in $Q$; in other words, let $\a_1 , \a_2 , \ldots, \a_m$ be arrows such that $t(\a_\m ) = s( \a_{\m +1 })$ for every $\m =1 , 2, \ldots, m-1 $ and $t( \a_m ) = s( \a_1 )$.
		Then, letting $\a_0 = \a_m $, the $ \Ga $-bimodule $J( \La  )^{\otimes_\Ga l} $ possesses an indecomposable direct summand isomorphic to $\Ga ( s(\a_1 ) , t(\a_r) ) $ for every $l > 0 $, where $r$ is the remainder of the division of $l$ by $m$. On the other hand, if $Q$ does not possess any oriented cycle, then $J(\La )^{\otimes_\Ga | Q_0 | } = 0 $ as every path in $Q$ has length at most $ | Q_0 | - 1 $.
	\end{exam}
	
	In general, it is still not an easy task to compute whether a specific bimodule over a bound quiver algebra is tensor nilpotent and/or (right) perfect. For this reason, we introduce a more specialized class of tensor nilpotent right perfect bimodules through the novel notion of strongly-finite projective dimension.

	\subsection{Strongly finite projective dimension}		\label{subsec:adding.ar.3}

	Let $\Ga $ be a left perfect ring, and fix a basic set of primitive orthogonal idempotents $\{ e_i \}_i $. 
	To any module $ {}_\Ga M $ with minimal projective resolution
	\[
	\ldots \to P_j \to \ldots \to P_1 \to P_0 \to M \to 0
	\]
	we associate 
	the following sets of primitive idempotents.
	\begin{enumerate}[\rm(i)]
		\item The \emph{support} of $M $: the set $\supp ( {}_\Ga M ) = \{ e_i \, | \, e_i M \neq 0 \}  $.
		\item The \emph{$j$-th homological support} of $ M$, denoted $\supp_j ( {}_\Ga M )  $: the set containing all idempotents $e_i$ such that $ \Ga e_i $ is isomorphic to a direct summand of $P_j $.
		\item The \emph{$\infty$-homological support} of $M$, denoted $\supp_\infty ( {}_\Ga M )  $: the union of the $j$-th homological supports of $M$ for all integers $j \geq 0 $.
	\end{enumerate}
	Evidently, one may define analogous sets for right modules over a right perfect ring.
	
	\begin{rem}
		The definition of homological supports over a left perfect ring $ \Ga $ is motivated by the fact that left perfectness ensures all projective modules are isomorphic to a unique direct sum of the indecomposable modules $ \Ga e_i $ (see \cite[Proposition~28.13]{AndersonFuller}{}). Consequently, each $ P_j $ in the minimal projective resolution of a $ \Ga $-module $ M $ is isomorphic to a direct sum of the form $ \oplus_i \Ga e_i ^{ ( T_i^j ) } $ for sets $ T_i^j $ of uniquely determined cardinalities, and $ e_i \in \supp_j ( {}_\Ga M ) $ precisely when $ T_i^j $ is non-empty.
		
		Similarly, these homological supports can also be defined for finitely generated modules over a semiperfect ring, as analogous decompositions are available for finitely generated projective modules in this case (see \cite[Characterization~27.13]{AndersonFuller}{}).
	\end{rem}
	
	We illustrate the above notions with an example.
	
	\begin{exam}
		\label{exam:ar.rem.12}
		Let $ {}_\Ga M$ be a module over a left perfect ring $\Ga $. Then it is clear that $\supp_0 ({}_\Ga M ) \subseteq \supp ({}_\Ga M) $, and the two sets are actually equal if $M$ is semisimple. In that case, we have an isomorphism
		\[
		M \simeq \oplus_i ( \topp_\Ga \Ga e_i )^{(T_i)}
		\]
		for possibly infinite sets $T_i$, where $ T_i $ is non-empty if and only if $ e_i \in \supp ({}_\Ga M) $. Furthermore, we have $e_j { \topp_\Ga \Ga e_i } \neq 0 $ if and only if $j = i $, due to the isomorphisms
		\[
		e_j ( \topp_\Ga \Ga e_i ) \simeq \Hom _\Ga ( \Ga e_j , { \topp_\Ga \Ga e_i } ) \simeq \Hom _\Ga ( { \topp_\Ga \Ga e_j } , { \topp_\Ga \Ga e_i } )
		\]
		that follow from Schur's Lemma and the fact that the set $\{ e_i \}_i$ is basic.
		It remains to note that the projective cover of $M$ is $ \oplus_i { \Ga e_i } ^{(T_i)} $.
	\end{exam}


	\begin{defn}
			\label{defn:strongly-finite}
		Let $M$ be a bimodule over a perfect ring $\Ga$.
		We call the right projective dimension of $M$ \emph{strongly-finite} if
			\begin{enumerate}[\rm(i)]
				\item it is finite (i.e.\ $ \pd M_\Ga < \infty $), and
				\item the set $ \supp ( {}_\Ga M ) \cap \supp_\infty ( M_\Ga ) $ is empty.
			\end{enumerate} 	
	\end{defn}
	
	\begin{rem}
		The requirement that the intersection $ \supp ( {}_\Ga M )  \cap \supp_\infty ( M_\Ga ) $ is empty can also be expressed through the notion of codominant dimension, which is dual to dominant dimension introduced in \cite{domdim}{}. Let $e$ denote the sum of all idempotents in $\{ e_i \}_i$ and let $ f $ denote the sum of all idempotents in $ \{ e_i \}_i \setminus \supp ({}_\Ga M)$. Then the intersection $ \supp ( {}_\Ga M ) \cap \supp_\infty ( M_\Ga ) $ is empty if and only if the relative codominant dimension of $ M_\Ga $ with respect to the right $ \Ga $-module $ f \Ga $ is infinite, denoted $ ( f \Ga ) \codomdim { M_\Ga } = \infty $. See also \cite{Tiago}{}.
	\end{rem}
	
	\begin{exam}
		\label{exam:ar.rem.10}
		Let $\La$ be a basic algebra over a field $k$ with one-dimensional simple modules, and let $\Ga = \La / J ( \La )$ be its semisimple part. We claim that the right projective dimension of the ideal $J(\La )$ viewed as a $\Ga$-bimodule is strongly-finite if and only if every vertex of the ordinary quiver of $\La$ is a source or a sink.
		Without loss of generality, we identify $\La $ with one of its representations as a bound quiver algebra, say $k Q / I $, and retain the notation of \cref{exam:ar.rem.8}{}.
		It holds that
			\[
				\supp_\infty ( J(\La )_\Ga ) = \supp_0 ( J(\La )_\Ga ) = \supp ( J(\La )_\Ga )
			\]
		on the grounds that $J( \La )$ is projective and semisimple as a right $\Ga $-module, see \cref{exam:ar.rem.12}{}.
		It follows that a vertex $ i $ is in $ \supp_\infty ( J(\La )_\Ga ) $ if and only if there is a non-zero path of positive length with target $ i $, that is if $ i $ is not a source. 
		Similarly, it holds that $ i \in \supp ( {}_\Ga J(\La ) ) $ if and only if $ i $ is not a sink. Consequently, the intersection $ \supp_\infty ( J(\La )_\Ga ) \cap \supp ( {}_\Ga J( \La ) )  $ is empty if and only if every vertex of $ Q $ is a source or a sink.
	\end{exam}

	We show next that bimodules of strongly-finite right projective dimension are tensor nilpotent and right perfect.

	\begin{prop}
		\label{lem:ar.rem.5}
		Let $ \Ga $ be a perfect ring, and let $M$ be a $ \Ga $-bimodule of strongly-finite right projective dimension. Then $M$ is right perfect and $M \otimes_\Ga M = 0$.
	\end{prop}
	
	\begin{proof}
		We have to show that $ M \otimes_\Ga M = 0 $ and $\Tor_i^\Ga (M, M) = 0 $ for every $i \geq 1$. Let $\{ e_i \}_{i}$ be a basic set of primitive orthogonal idempotents for $\Ga $. If
		\[
		\mathbb{P} \colon \, \, \, 0 \to P_m \to \ldots \to P_1 \to P_0 \to M \to 0
		\]
		is a minimal projective resolution of $ M_\Ga $, then we may assume that $P_j = \bigoplus_i { e_i \Ga }^{ ( T_i^j ) } $ for possibly infinite sets $ T_i^j $ and $j = 0 , 1 , \ldots , m $, where the set $ T_i^j $ is non-empty if and only if $ e_i \in \supp_j ( M_\Ga ) $.
		
		We claim that $ P_j \otimes_\Ga M = 0 $ for every $j = 0 , 1 , \ldots , m $. Indeed, the intersection $\supp({}_\Gamma M) \cap \supp_\infty(M_\Gamma)$ is empty as the right projective dimension of $ M $ is strongly-finite by assumption, implying that for every $e_i \in \supp_j(M_\Gamma)$ we have $ e_i M = 0 $. Therefore, we have $ e_i \Ga \otimes_\Ga M \simeq e_i M = 0 $ for every $e_i \in \supp_j(M_\Gamma)$, and our claim follows from the fact that $P_j $ is a direct sum of such modules.
		
		Consequently, tensoring the right part of $ \mathbb{P} $ with $ M $ from the right yields the sequence
			\[
				0 \to 0 \to M \otimes_\Ga M \to 0
			\]
		which is exact as tensor functors are right exact; in particular, we have $M \otimes_\Ga M = 0 $.
		We deduce that $\Tor_i^\Ga ( M , M ) = 0 $ for all $i \geq 0 $ as $\mathbb{P } \otimes_\Ga M $ is the zero complex.
	\end{proof}

	\begin{cor}
		Let $ \Ga $ be an Artin algebra, and let $M$ be a $ \Ga $-bimodule of strongly-finite right projective dimension. Then the trivial multiplicative bimodule over $M$ is removable.
	\end{cor}

	\begin{proof}
		The corollary follows directly from \cref{lem:ar.rem.5}{} and \cref{cor:ar.rem.6}{}. Note that the only possible multiplicative structure on $ M $ is the trivial one.
	\end{proof}

	The next result shows that the notion of strongly-finite right projective dimension can be equivalently described over the induced trivial extension algebra.
	
	\begin{prop}
		\label{prop:ar.rem.4}
		Let $\La $ be a perfect ring. Let $K \subseteq J(\La )$ be an ideal such that $K^2 = 0 $ and the natural ring epimorphism $ \pi \colon \La \epic \La / K = \Ga $ splits. Then the right projective dimension of $K$ is strongly-finite over $\La$ if and only if it is over $\Ga$.
	\end{prop}

	\begin{proof}
		Let $ \i \colon \Ga \monicc \La $ be any section ring monomorphism of $ \pi $, and consider $ K $ as a $\Ga$-bimodule through restriction of scalars along $ \i $. Let $\{e_i \}_i $ be a basic set of primitive orthogonal idempotents for $\Ga$. Then it follows from \cref{lem:semiperfect.clefts}{} that $\{ \i (e_i ) \}_i $ is a basic set of primitive orthogonal idempotents for $\La$. Note that the sets $\supp ( {}_\Ga  K )  $ and $ \supp ( {}_\La  K ) $ are equal if we identify the idempotent $e_i$ of $\Ga $ with its image $ \i( e_i ) $ in $\La $ for every $ i $. In the rest of the proof we identify $ \Ga $ with a subring of $ \La $ so that $ \i $ is just the inclusion monomorphism.

		For one direction, assume that the right projective dimension of $K$ is strongly-finite over $\Ga$. Then the groups $\Tor_j^\Ga ( K , K ) $ are trivial for all $j \geq 0$ according to \cref{lem:ar.rem.5}{}. Let now $ \mathbb{P} $ be a minimal projective resolution of $K_\Ga $. Then the complex $ \mathbb{ P } \otimes_\Ga \La $ is a minimal projective resolution of $( K _\Ga ) _\La $ according to \cref{lem:cleft.3}{}. Furthermore, observe that the right $ \La $-module structure of $( K _\Ga ) _\La $ is the same as the structure that $ K $ inherits as a right ideal of $ \La $ since $K^2 = 0 $. Hence, it holds that $\supp_j ( K_\Ga ) = \supp_j ( K_\La ) $ for every non-negative integer $j$. In particular, we have that $\supp_\infty ( K_\La ) = \supp_\infty ( K_\Ga ) $ and $ \pd K_\La = \pd K_\Ga $ is finite. It remains to notice that $ \supp_\infty ( K_\La ) $ and $ \supp ( {}_\La  K ) $ are disjoint as $ \supp ( {}_\La  K )  = \supp ( {}_\Ga  K ) $ and the sets $ \supp_\infty ( K_\Ga ) $ and $ \supp ( {}_\Ga  K ) $ are disjoint by assumption.
		
		Conversely, let the right projective dimension of $K$ be strongly-finite over $ \La $ and
		\[
		\mathbb{P} \colon \, \, \, 0 \to P_m \to \ldots \to P_1 \to P_0 \to K \to 0
		\]
		be a minimal projective resolution of $K_\La$. We may assume that $P_j = \bigoplus_i  e_i \La ^{ ( T_i^j ) } $ for possibly infinite sets $ T_i^j $, where the set $ T_i^j $ is non-empty if and only if $e_i \in \supp_j ( K_\La ) $. Our target is to show that the complex $ \mathbb{ P }_\Ga $ is a minimal projective resolution of $K _\Ga $.

		Observe that if an idempotent $ e_i $ is in $ \supp_j ( K_\La ) $ for some $j \geq 0$, then $e_i K = 0$ as the right projective dimension of $K$ is strongly-finite over $ \La $ by assumption; in particular, we have $e_i \La = e_i \Ga $ for every $ e_i \in \supp_\infty ( K_\La ) $ as $\La =  \Ga \oplus K$. Therefore, the inclusion $ \i $ (which is also a monomorphism of right $\Ga$-modules) yields the identification $e_i \Ga = e_i \La_\Ga $. In particular, the complex $ \mathbb{ P }_\Ga $ is a projective resolution of $K _\Ga $. For the minimality, note that
		\[
			\rad  e_i \La _\La = e_i J( \La ) = e_i J( \Ga ) = \rad e_i \Ga_\Ga = \rad e_i \La  _\Ga     
		\]
		due to the equality $J ( \La ) = J( \Ga ) \oplus K$ established in \cref{lem:cleft.2}{}.
		It remains to note that $ ( P_j )_\Ga \simeq \bigoplus_i { e_i \Ga } ^{ ( T_i^j ) } $, and thus the equality $\supp_j ( K_\Ga )  = \supp_j ( K_\La )$ holds for every $ j \geq 0 $. In particular, we have $\supp_\infty ( K_\Ga ) = \supp_\infty ( K_\La ) $ which implies the disjointness of $ \supp_\infty ( K_\Ga ) $ and $ \supp ( {}_\Ga K ) $ as in the proof of the first direction, and $ \pd K_\Ga = \pd K_\La $ is finite. 
	\end{proof}

	\begin{exam}
		Let $\La_1 = k Q / I_1 $ be the algebra of \cref{exam:ar.rem.1}{}. It holds that $\g $ is a non-repetitive arrow in $\La_1 $, i.e.\ the square of the ideal $\la \g + I_1 \ra $ is zero. Furthermore, we have $\pd { \la \g + I_1 \ra }_{\La_1} = 1 $ as there is a minimal projective resolution of the form
		\[
		0 \to e_2 \La_1 \to e_1 \La_1 \to \la \g + I_1 \ra \to 0.
		\]
		We conclude that the right projective dimension of $\la \g + I_1 \ra $ viewed as a $\La_1 $-bimodule is strongly-finite since $\supp_\infty ( \la \g + I_1 \ra _{ \La _1 } ) = \{ e_1, e_2 \} $ and $ \supp ( {}_{ \La _1 } \la \g + I_1 \ra ) = \{ e_3 \} $, where the latter fact is due to the zero relations $ \b \g $ and $ \z \g $.
		
		As $\g$ is pre-removable in $\La_1$, \cref{prop:ar.rem.4}{} implies that the ideal $\la \g + I_1 \ra $ is also of strongly-finite right projective dimension as a $\Ga_1 $-bimodule, where $\Ga_1 = \La_1 / \la \g + I_1 \ra $. The $\Ga_1$-bimodule structure of $\la \g + I_1 \ra $ is given by restriction of scalars along any section of the natural algebra epimorphism $ \pi \colon \La_1 \epic \Ga_1 $.
	\end{exam}

	Before proceeding with the last part of this subsection, we introduce a lemma that measures how different are the various types of multiplicative bimodules studied so far. In particular, the following lemma shows that \hyperlink{thm:main.ar.rem.3}{\cref{thm:main.ar.rem.3}{}} is strictly more general than \cref{cor:ar.rem.6}{}.

	\begin{lem}
		The following hold for a bound quiver algebra $\La $ up to isomorphism:
		\begin{enumerate}[\rm(i)]
			\item The algebra $\La$ occurs as the split extension of a semisimple algebra by a removable multiplicative bimodule if and only if $ \gd \La < \infty $.
			\item The algebra $\La $ occurs as the split extension of a semisimple algebra induced by a tensor nilpotent (perfect) bimodule if and only if the quiver of $\La $ is acyclic.
			\item The algebra $\La $ occurs as the trivial extension of a semisimple algebra by a bimodule of strongly-finite right projective dimension if and only if every vertex in the quiver of $\La $ is a source or a sink.
		\end{enumerate}
		In all cases, the base semisimple algebra is $ \La / J ( \La ) $.
	\end{lem}
	
	\begin{proof}
		We begin by establishing some useful technical facts.
		Let $ \Ga $ be a semisimple algebra  isomorphic to a finite number of copies of the base field $ k $.
		Let $( M , \theta )$ be a nilpotent multiplicative bimodule for $ \Ga $ such that there is an isomorphism $\phi \colon \La \isomto { \Ga \ltimes_\theta M } = E $. Let $ i \colon \Ga \monicc E $ and $\pi \colon E \epic \Ga $ denote the inlcusion monomorphism and the projection epimorphism, respectively. Nilpotency of the multiplicative bimodule $( M , \theta )$ implies that the ideal $ \wt{ M } $ is nilpotent. As $ E $ is isomorphic to a bound quiver algebra according to \cref{lem:ar.rem.8}{}, it holds that $J(E)$ is the largest nilpotent ideal of $ E $ and, thus, it contains $ \wt{ M } $. Moreover, the equality $J(E) = \wt{ M } $ holds according to \cref{lem:cleft.2}{}, as $ \Ga $ is semisimple; in particular, the isomorphism $\phi$ maps the ideal $J( \La )$ onto $ \wt{ M } $ and $ \Ga $ is isomorphic to $ \La / J ( \La ) $. Furthermore, the projective dimension of $J(\La)$ over $\La$ from either side is equal to the respective projective dimension of $\wt{ M } $ over $E$.
		
		In particular, if the global dimension of $ \La $ is finite, then the projective dimension of $ \wt{ M } $ over $ E $ is finite from both sides and $ ( M , \theta ) $ is a removable multiplicative bimodule. Conversely, if $ ( M , \theta ) $ is removable, then the right projective dimension of $ J ( \La ) $ over $ \La $ is finite implying that $ \gd \La < \infty $.
		
		We identify now $\La $ with one of its representations as a bound quiver algebra, say $k Q / I$, so that $ \phi $ maps $ e_j $ into $ \Ga \oplus 0 $ for every vertex $ j \in Q_0 $. Furthermore, we identify $ \Ga $ with $ \La / J ( \La ) $ through the correspondence $\pi ( \phi ( e_j ) ) \mapsto \ol{ e_j } $ for every $ j \in Q_0 $.
		Then the two squares of the following diagram commute
		\[\begin{tikzcd}
			{\Lambda =k Q / I} & {\Gamma \ltimes_\theta M = E} \\
			\Gamma & \Gamma
			\arrow["\sim", "\phi"', from=1-1, to=1-2]
			\arrow["{\pi'}", curve={height=-6pt}, two heads, from=1-1, to=2-1]
			\arrow["\pi", curve={height=-6pt}, two heads, from=1-2, to=2-2]
			\arrow["{i'}", curve={height=-6pt}, hook', from=2-1, to=1-1]
			\arrow[equals, from=2-1, to=2-2]
			\arrow["i", curve={height=-6pt}, hook', from=2-2, to=1-2]
		\end{tikzcd}\]
		where $\pi ' $ is the natural algebra epimorphism, and $i'$ is the section of $\pi '$ such that $ i'( \ol{ e_j } ) = e_j $ for every vertex $ j \in Q_0 $. It follows that $\phi $ induces an isomorphism of $ \Ga $-bimodules between $J( \La )$ and $ \wt{ M } $, where the $ \Ga $-bimodule structures are due to restrictions of scalars along $i'$ and $i$ respectively.

		For the second equivalence, note that $ Q $ is acyclic if and only if $J ( \La )$ is tensor nilpotent according to \cref{exam:ar.rem.8}{}. Claim (ii) follows now from the isomorphism of $ M $ and $J ( \La )$ as $ \Ga $-bimodules, taking into account the fact that all bimodules are perfect over a semisimple ring.
		Similarly, the last claim holds as every vertex of $ Q $ is a source or a sink if and only if the right projective dimension of $ J ( \La ) $ is strongly-finite over $ \Ga $ according to \cref{exam:ar.rem.10}{}.
	\end{proof}

	The final part of this section establishes a novel method for constructing bound quiver algebras of finite finitistic (or global) dimensions from smaller algebras with the same property, leveraging \cref{cor:ar.rem.6}{} and \cref{lem:ar.rem.5}{}. Moreover, the method is illustrated with the construction of an explicit algebra with finite finitistic dimensions, a highly non-trivial result in the sense of \cref{defn:irredu}{}; see \cref{exam:3}{}.
	
	We begin with a combinatorial construction.
	
	\begin{con}
		\label{exam:ar.rem.11}
		Let $\Ga = k Q_\Ga / I_\Ga $ be a bound quiver algebra, and let $i , j $ be a pair of distinct vertices. Let $ V $ be a $ \Ga $-submodule of $\Ga e_i$ such that $ e_j \Ga e_i \subseteq V \subseteq \rad _\Ga \Ga e_i $, and choose a a generating set $\{ u_\nu + I_\Ga \}_\nu $ for $ V $ as a left $\Ga$-module. We denote by $ \La = \Ga_{ i \to j }^V $ the algebra $ k Q_\La / I_\La $, where:
			\begin{enumerate}[\rm(i)]
				\item $Q_\La = Q_\Ga \dot \cup \{ \a \colon i \to j \}$ is the quiver that results from $Q_\Ga $ if we add a new arrow $\a \colon i \to j$;
				\item $I_\La $ is the ideal of $k Q_\La $ generated by $ I_\Ga \cup T_\a $, where $ T_\a = \{ u_\nu \a \}_\nu $.
			\end{enumerate}
		We call the algebra $ \La $ a \emph{trivial one-arrow extension} of $ \Ga $. 
	\end{con}

	The notation introduced in \cref{exam:ar.rem.11}{} will be retained throughout the rest of the subsection. We recall that $ \Ll( - ) $ denotes the Loewy length. 
	
	\begin{lem}	\phantomsection		\hypertarget{lem:last}{}
			\label{lem:last}
		The following hold for the one-arrow extension algebra $ \Ga_{ i \to j }^V $.
			\begin{enumerate}[\rm(i)]
				\item It does not depend on the choice of a generating set for $ V $.
				\item It is a bound quiver algebra, where every path passing through $ \a $ at least twice is zero.
				\item It holds that $ \Ll( \Ga ) \leq \Ll( \Ga_{ i \to j }^V ) \leq 2 \Ll ( \Ga ) $.
			\end{enumerate}
	\end{lem}
	
	\begin{proof}
		(i) It holds that every $ z \in V $ is a $ k $-linear combination of elements of the form $ p u_\nu + I_\Ga $, where $ p $ stands for a path in $ Q_\Ga $. In particular, it holds that $ z \a \in I_\La $ for every $ z \in k Q_\Ga e_i $ such that $ z + I_\Ga \in V $. This implies directly that if $ \{ u'_\nu + I_\Ga \}_\nu $ is another generating set for $ V $, then $ u'_\nu \a \in I_\La $ for every $ \nu $. A symmetric argument shows that $ I_\La $ does not depend on the chosen generating set for $ V $.
		
		(ii) Every path occurring in $ I \cup T_\a $ is clearly of length at least two as every path occurring in $ V \subseteq \rad _\Ga \Ga e_i $ has length at least one. Furthermore, every path $ r $ of $ Q_\La $ of length at least $ 2 \Ll( \Ga ) $ is in $I_\La $. Indeed, path $ r $ either possesses a subpath in $Q_\Ga $ of length $\Ll( \Ga )$, or it possesses a subpath of the form $ \a q \a $ for some path $ q $ in $Q_\Ga $. In the first case, we have $ r \in I_\Ga \subseteq I_\La $. In the second case, the desired property follows from the first part of the proof, since $ q + I_\Ga \in e_j \Ga e_i \subseteq V $ implies that $ q \a \in I_\La $.
		
		(iii) We have already proven the first inequality. For the second, it remains to observe that a non-zero path in $ \Ga $ is also non-zero in $\La $ since all paths occurring in $ T_\a $ are divided by $ \a $.
	\end{proof}
	
	\begin{prop}
			\label{thm:last}
		Every trivial one-arrow extension $ \Ga_{ i \to j }^V $ of a bound quiver algebra $ \Ga $ is isomorphic to a trivial extension of $ \Ga $ by a $ \Ga $-bimodule of strongly-finite right projective dimension. In particular, it holds that
			\[
				\fpd { \Ga_{ i \to j }^V } < \infty  \,\, \iff \,\, 	\fpd \Ga < \infty 
			\]
		and the equivalence remains valid if $ \fpd $ is replaced by $ \Fpd $ or $ \gd $.
	\end{prop}
	
	\begin{proof}
		To see that $ \La = \Ga_{ i \to j }^V $ may be realized as a trivial extension of $ \Ga $, it suffices due to \cref{prop:ar.rem.3}{} to show that there is a split algebra epimorphism $ \La \epic \Ga $ such that the square of its kernel is zero. Observe that the added arrow $ \a \colon i \to j $ is pre-removable in $ \La $ as $ I _\La $ is generated by the set $ I_\Ga \cup T_\a $, and $ \Ga $ is the canonical representation of the quotient algebra $ \La / \la \a + I \ra $. Therefore, one part of our first claim follows from \hypertarget{lem:ar.rem.2}{\cref{lem:ar.rem.2}{}}{}. Furthermore, the square of $ \la \a + I \ra $ is zero as for every path $ q $ in $ Q_\Ga $ either $ \a q \a = 0 $, or $ q $ is a path from to $ j $ to $ i $ whence $ q + I_\Ga \in V $ and $ q \a \in I_\La $. We conclude that $ \La \simeq \Ga \ltimes \la \a + I_\La \ra $ where the $ \Ga $-structure on both sides of $ \la \a + I_\La \ra $ is given by restriction of scalars along the algebra monomorphism $ \i \colon \Ga \monicc \La $ such that $ p + I_\Ga \mapsto p + I_\La $ for every path $ p $ of $ Q_\Ga $; see \cref{rem:ar.rem.2,lem:ar.red.1}{}.
		
		To complete the proof of the first claim, we have to show that the right projective dimension of the ideal $ \la \a + I \ra $ is strongly-finite over $ \Ga $. We do that by showing that $ \la \a + I \ra $ is isomorphic to $ M =  W \otimes_k e_j \Ga $ as a $ \Ga $-bimodule, where $ W = \Ga e_i / V $. The latter $ \Ga $-bimodule can also be identified with the quotient of the inlcusion of $ \Ga $-bimodules $ V \otimes_k e_j \Ga \monicc \Ga e_i \otimes_k e_j \Ga $.
		
		First, we establish that the $ \Ga $-bimodule homomorphism $ f \colon { \Ga e_i \otimes_k e_j \Ga } \to \la \a + I_\La \ra $ defined by $ e_i \otimes e_j \mapsto \a + I_\La $ is surjective. This follows immediately from the fact that the square of $ \la \a + I_\La \ra $ is zero, as it means equivalently that every path in $ Q_\La $ divided by $ \a $ at least twice is in $ I_\La $. In particular, it implies that every element of $ \la \a + I_\La \ra $ is a $ k $-linear combination of paths of the form $ p \a q $ modulo $ I_\La $, where $ p $ and $ q $ are in $ Q_\Ga $. Furthermore, for every $ z \in  k Q_\Ga e_i $ such that $ z + I_\Ga \in V$ and every $ w \in e_j k Q_\Ga $, we have $ ( z + I_\Ga ) ( \a + I_\La ) ( w + I_\Ga ) = z \a w + I_\La = 0 $ as $ z \a \in I_\La $. In other words, the kernel of $ f $ contains $ V \otimes_k e_j \Ga $, inducing thus a $ \Ga $-bimodule epimorphism $ \bar{ f } \colon M \epic \la \a + I_\La \ra $. We show next that $ \bar{ f } $ is actually an isomorphism via a dimension argument that leverages the thechnical work done in the proof of \cref{lem:ar.rem.8}{}.
		
		Let $ E = \Ga \ltimes M $ be the trivial extension algebra induced by $ M $. It holds that the quiver of $E$ is equal to $Q_\La$ according to the proof of \cref{lem:ar.rem.8}{}. Moreover, it follows from the same proof that the unique $k$-algebra homomorphism $\phi \colon k Q_\La \to E $ sending $ e_i $ to $ ( e_i , 0 ) $ for every vertex $i $, every arrow $ \b  $ in $Q_\Ga$ to $ ( \b + I_\Ga , 0  )$ and the arrow $\a $ to $ ( 0 , \ol{ e_i } \otimes e_j ) $, is surjective. Our target is to show that the kernel of $ \phi $ contains $I_\La $.
		
		Recall that for every $ z \in k Q_\La $ we have that $ \phi ( z_{ \nott \! \a } ) \in \Ga \oplus 0 $, while $ \phi ( z_\a ) \in 0 \oplus M $. Specifically, if $ p $ is a path in $ Q_\Ga $ we have $ \phi ( p ) = ( p + I_\Ga , 0 ) $, and if $ p $, $ q $ are paths with target $ i $ and source $ j $, respectively, then $ \phi ( p \a q ) = ( 0 , \ol{ p + I_\Ga } \otimes ( q + I_\Ga ) ) $. 
		Moreover, since $ \a $ is pre-removable, it holds that $ z_{ \nott \! \a } , z_\a \in I_\La $ if $ z \in I_\La $. But $ I_\La \cap \la \BQnota \ra = I_\Ga $, implying that $ \phi ( z_{ \nott \! \a } ) = 0 $ for every $ z \in I_\La $. For $ z_\a $, observe that it is a $ k $-linear combination of elements of the form $ p w q $, where $ w \in I_\Ga $ or $ w = w' \a $ for $ w' \in V $, and $ p $, $ q $ are paths. Let $ w \in I_\Ga $. Then $ \a $ divides $ p $ or $ q $, and we may assume that it divides exactly one of them because, otherwise, it is immediate that $ p w q \in \Ker \phi $ as $ \a $ maps into $ 0 \oplus M $ and $ E $ is a trivial extension. Let $ p = p_1 \a p_2 $ for paths $ p_1, p_2 $ and assume that $ q $ avoids $ \a $. Then $ \phi ( p w q ) = ( 0 , \ol{ p_1 + I_\Ga } \otimes ( p_2 w q + I_\Ga ) ) = 0 $ as $ p_2 w q \in I_\Ga $. Now let $ w = w' \a $ for some $ w' \in V $, and observe that it suffices to deal with the case where both $ p $ and $ q $ avoid $ \a $. Then $ \phi ( p w q ) = ( 0 , \ol{ p w' + I_\Ga } \otimes ( q + I_\Ga ) ) = 0 $ as $ p w' + I_\Ga \in V $. All in all, we have shown that $ I_\La \subseteq \Ker \phi $.
		Therefore, the map $ \phi $ induces a $ k $-algebra epimorphism $ \bar{ \phi } \colon \La \epic E $. This implies in particular that $ \dim_k \la \a + I_\La \ra \geq \dim_k M $, which shows that the two $ \Ga $-bimodules are in fact isomorphic due to the previouslh established $ \Ga $-bimodule epimorphism $ \bar{ f } \colon M \epic \la \a + I_\La \ra $.
		
		Now it is immediate that $ \la \a + I_\La \ra $ has strongly-finite right projective dimension due to its identification with $ M $. Indeed, it is evident that $ M $ is isomorphic to $ \dim_k W $ copies of $ e_j \Ga $ as a right $ \Ga $-module and thus projective. Furthermore, it holds that $ \supp_\infty ( M_\Ga ) = \supp_0 ( M_\Ga ) = \{ e_j \} $ and $ e_j M = 0 $ as $ e_j ( \Ga e_i / V ) = 0 $ due to the assumption $ e_j \Ga e_i \subseteq V $.
		
		The last claim follows now from \cref{cor:ar.rem.6}{} and \cref{lem:ar.rem.5}{}, completing the proof.
	\end{proof}

	\begin{exam}
		If $e_j \Ga e_i = 0 $ for vertices $i, j $ of a bound quiver algebra $ \Ga = k Q_\Ga / I_\Ga $, then we recover the main result of \cite[Section~4]{arrowrem1}{} through \cref{thm:last}{} by choosing $ V $ to be the trivial submodule of $\Ga e_i $ in the context of \cref{exam:ar.rem.11}{}.
	\end{exam}

	\begin{exam}
			\label{exam:3}
		Let $ Q ' $ be the quiver of the following figure, and let $ k $ be any field. We consider the bound quiver algebra $ \La_3' = k  Q' / I_3' $ where $ I_3' $ is the ideal of $ k Q' $ generated by the relations in
			\[
				R_3' = \{ \a \e - \d \a   , \,  \d ^2   , \,  \e ^2   , \,  \z ^2   , \,   \e \b  , \,   \b \g   , \,  \b \z \g \a   , \,  \b \eta   , \,  \b \z \eta  \} .
			\]
		Then it is not difficult to verify that $ \La_3' $ is the trivial one-arrow extension of the algebra $\La_3 $ of \cref{exam:ar.rem.3}{}, determined by the new arrow $\eta \colon 3 \to 2 $ and the new relations $ \b \eta $ and $ \b \z \eta $ corresponding to a basis of the $\La_3$-submodule $ V $ of $ \La_3 e_3$ generated by the subspace $ e_2 \La_3 e_3 = {}_k \la \b + I_3 , \b \z + I_3 \ra $. Consequently, it follows from \cref{thm:last}{} that $\Fpd \La_3' < \infty $, as we have already established that $\Fpd \La_3 < \infty $ in \cref{exam:ar.rem.3}{}.
		
		
		\begin{figure}[hbt!]
			\vspace*{-0.15cm}
			\centering
			\resizebox{11em}{!}{
				\begin{tikzpicture}
					
					
					\draw  ($(0,0)+(90:1.2)$) circle (.08);
					\node at ($(0,0)+(90:0.6)$) {$\mathbf{3}$};
					
					\draw  ($(0,0)+(210:1.2)$) circle (.08);
					\node at ($(0,0)+(210:0.6)$) {$\mathbf{1}$};
					
					\draw  ($(0,0)+(330:1.2)$) circle (.08);
					\node at ($(0,0)+(330:0.6)$) {$\mathbf{2}$};

					\draw[->,shorten <=7pt, shorten >=7pt] ($(0,0)+(210:1.2)$) -- ($(0,0)+(330:1.2)$);
					
					\draw[->,shorten <=7pt, shorten >=7pt] ($(0,0)+(330:1.2)$) -- ($(0,0)+(90:1.2)$) ;
					
					\draw[->,shorten <=7pt, shorten >=7pt] ($(0,0)+(90:1.2)$)  -- ($(0,0)+(210:1.2)$);

					\node at ($(0,0)+(150:0.85)$) {$\gamma$};
					
					\node at ($(0,0)+(270:0.85)$) {$\alpha$};
					
					\node at ($(0,0)+(30:0.85)$) {$\beta$};

					\draw[->,shorten <=4pt, shorten >=4pt] ($(0,0)+(210:1.3)$).. controls +(210+45:1) and +(210-45:1) .. ($(0,0)+(210:1.3)$);
					\node at ($(0,0)+(210:2.07)$) {$\delta$};
					
					\draw[->,shorten <=4pt, shorten >=4pt] ($(0,0)+(330:1.3)$).. controls +(330+45:1) and +(330-45:1) .. ($(0,0)+(330:1.3)$);
					\node at ($(0,0)+(330:2.05)$) {$\varepsilon$};
					
					\draw[->,shorten <=4pt, shorten >=4pt] ($(0,0)+(90:1.3)$).. controls +(90+45:1) and +(90-45:1) .. ($(0,0)+(90:1.3)$);
					\node at ($(0,0)+(90:2.13)$) {$\zeta$};
					
					\draw[->,shorten <=7pt, shorten >=7pt] ($(0,0)+(90:1.2)$) arc (90:-30:1.2);
					\node at ($(0,0)+(30:1.4)$) {$\eta$};
					
				\end{tikzpicture}
			}
			\label{fig3}
			\vspace*{-0.35cm}
		\end{figure}		
		
		It is not difficult to check that $ \La_3' $ is an irreducible bound quiver algebra in the sense of \cref{defn:irredu}{}. For instance, note that $\Ll( \La_3' ) = \Ll ( \La_3 ) = 8 $ as the longest non-zero path of $ Q' $ passing through $\eta $ is $\z \eta \b \z \g \d$ (see also \cref{lem:last}{}), and the regular right $ \La_3' $-module does not possess a submodule isomorphic to the simple right module corresponding to vertex $ 3 $. In particular, it holds that $ \Fpd \La_3' $ is positive finite.
		
		For a final example, observe that one may also choose $ V' = \rad_{\La_3} \La_3 e_3 $ instead of $ V $. Then the resulting trivial one-arrow extension algebra is $ \La_3'' = k Q' / I'' $, where $ I'' $ is the ideal of $ k Q' $ generated by $ R_3'' = R_3' \cup \{ \z \eta \}$. One can verify that $ \La_3'' $ is an irreducible bound quiver algebra of positive finite big finitistic dimension, in the same way as we did for $ \La_3 ' $.
	\end{exam}

	\medskip

	\appendix

	\setcounter{thm}{0}
	\renewcommand{\thethm}{A.\arabic{thm}}
	
		\phantomsection		\hypertarget{appendix}{}
	
		\section*{Appendix. Characterization of classical arrow removal}

	\smallskip
	
	Let $\La = k Q / I$ be a bound quiver algebra with a set of arrows $ A $ such that $\La / \la A + I \ra $ is an arrow removal algebra of $\La $ in the sense of \cite{arrowrem1}{}. The aim of this appendix is to show that this is equivalent to $ A $ being pre-removable in $\La $ (\cref{defn:ar.rem.6}{}) and $\la A + I \ra $ being projective as a $\La $-bimodule.
	It will follow in particular that $ A $ is two-sided removable in $\La $ in this case (\cref{defn:ar.rem.1}{}) and, therefore, our reduction technique generalizes the classical arrow removal operation.
	
	We begin by recalling some basic notions and facts about Gr{\"o}bner bases for admissible ideals of path algebras due to Green \cite{Green}{}, which play an important role in the proof of \cref{prop:ar.rem.1}{}.
	We refer the reader to the beginning of \cref{sec:gen.ar.rem}{} for notation not explained here.

	Let ``$\preceq$'' denote a fixed \emph{admissible} order on $\BQ$ (see \cite[Subsection~2.2.2]{Green}{} for more details).
	%
	Then for a non-zero element $w = \sum \lambda_p \cdot p \in k Q$,
	its \emph{tip} is
	\[
	\tip(w) = \max\nolimits_\preceq \{p \in \BQ \, | \, \lambda_p \neq 0 \}.
	\]
	Similarly, we let $\tip( T ) = \{ \tip( t ) \, | \, t \in T \setminus \{ 0 \} \}$ and $\ntip( T ) = \BQ \setminus \tip( T )$ for any subset $T \subseteq  k Q$.

	Now let $\La = k Q / I$ be an admissible quotient of the path algebra $k Q$, that is all elements in the ideal $ I $ comprise paths of length at least two and $ I $ contains all paths of sufficiently large length. The set of paths $ \NN = \ntip (I)$ is closed under subpaths, and is a $k$-basis of $\La  = k Q / I$ modulo $ I $. Furthermore, there is always a finite set of relations $ \GG $ for $ I $ such that $ \la \tip (\GG ) \ra = \la \tip ( I  ) \ra $, i.e.\ there is always a \emph{finite Gr{\"o}bner basis for $I$} (see \cite[Subsection~2.2.3]{Green}{}). A key property of $ \GG $ is that, for any path $ p \in \BQ $, it holds that $ p \in \tip (I) $ if and only if $ p $ is divided by the tip $ \tip ( g ) $ of some element $ g \in \GG $.
	

	Let us recall now the definition of an arrow removal algebra in the sense of \cite{arrowrem1}{}. For a set of arrows $A \subseteq Q_1$, the quotient algera $ \La / \la A + I \ra $ is called an \emph{arrow removal} of $\La$ if the following conditions hold.
	\begin{enumerate}[\rm(i)]
		\item There is a generating set $ T $ for $I$ that avoids all arrows in $A$.
		\item It holds that $e_{t(\a_1)} \La e_{s(\a_2)} = 0 $ for all arrows $\a_1 , \a_2 \in A $.
	\end{enumerate}
	
	The following preliminary lemma is employed in the proof of \cref{prop:ar.rem.1}{}.
	
	\begin{lem}
		\label{lem:appendixA.1}
		Let $\a \colon i \to j $ be an arrow of the bound quiver algebra $\La = k Q / I $ and let $\CC_{ \rightsquigarrow i} $, $\CC_{ j \rightsquigarrow }$ be two sets of paths with target $i$ and source $j$, respectively, such that the sets $\CC_{ \rightsquigarrow i} + I$ and $ \CC_{j \rightsquigarrow} + I $ are $k$-bases of $ \La e_i $ and $e_j \La $, respectively. Then the ideal $\la \a + I \ra $ is a projective $\La $-bimodule if and only if the set $\CC_{ \rightsquigarrow i} \a \CC_{ j \rightsquigarrow }$ is linearly independent modulo $I$. In particular, it holds that $e_j \La e_i = 0  $ in this case. 
	\end{lem}
	
	\begin{proof}
		The map $f \colon \La e_i \otimes_k e_j \La \epic \la \a + I \ra $ defined by $z_1 \otimes z_2 \mapsto z_1 (  \alpha  + I ) z_2$ for every $z_1 \in \La e_i$ and $z_2 \in e_j \La $ is a projective cover of $ \la \a + I \ra $ as a $\La$-bimodule, since $I$ is admissible. Therefore, the ideal $ \la \a + I \ra $ is a projective $\La$-bimodule if and only if $f$ is bijective or, equivalently, if $\CC_{\rightsquigarrow i} \alpha \CC_{j \rightsquigarrow}$ is linearly independent modulo $I$. Now assume that the ideal $\la \a + I \ra $ is projective as a $\La$-bimodule. If $ p $ is a non-zero path in $ \La $ with source $ j $ and target $ i $, then $p \alpha p $ is non-zero as $(p + I ) \otimes (p + I) $ is non-zero in $ \La e_i \otimes_k e_j \La $ and $f$ is bijective. Inductively, we get that $p (\alpha p)^i $ is non-zero in $ \La $ for every positive integer $ i $, a contradiction given that $I$ is admissible.
	\end{proof}
	
	Recall that, for a set of arrows $A  \subseteq Q_1$, we use $\BQA$ to denote the set of paths in $Q$ that pass through at least one of the arrows in $A$, and $\BQnotA $ denotes the complement of $ \BQA $ in $ \BQ$. Furthermore, for every $z \in k Q$, we write $ z_A$ and $z_{\nott \! A}$ to denote the unique elements in $ {}_k \la \BQA \ra$ and ${}_k \la \BQnotA \ra$, respectively, such that $z = z_A + z_{\nott \! A}$. We also use the same notation for subsets of $k Q$.
	
	
	\begin{prop}
		\label{prop:ar.rem.1}
		Let $\La = k Q  / I$ be a bound quiver algebra and $ A $ a set of arrows. It holds that $\La / \la A + I \ra$ is an arrow removal algebra of $\La $ in the sense of \textup{\cite{arrowrem1}{}} if and only if $A $ is pre-removable in $\La $ and $ \langle A + I \rangle $ is projective as a $\La$-bimodule.
	\end{prop}

	\begin{proof}
		We divide the proof into two steps, where we deal first with single arrow removals and then build on this step to prove the general case.
		
		\textbf{Step I:} The case $A = \{ \a \} $ for a single arrow $\a \colon i \to j$.	\\[0.05cm]
		Let $ T $ be a generating set for $I$ avoiding $ \a $. Then for any admissible order on $\BQ$ there is a Gr{\"o}bner basis $\GG$ for $I$ avoiding $\a$; see \cite[Lemma A.4]{arrowrem1}{}. We assume without loss of generality that every element of $ T $ is a relation.
		It is clear that $\a $ is pre-removable as condition (i) of arrow removal implies that $ I $ has a generating set avoiding $ \a $, which by definition implies pre-removability.
		
		Let now $ \NN = \ntip(I) $ and recall that the paths of $ \NN e_i $ and $ e_j \NN $ form $k$-bases of $\La e_i$ and $e_j \La$ modulo $I$, respectively.
		Let
			\[
				w = \sum_{p, q} \lambda_{p, q} \cdot p \alpha q \in I
			\]
		for coefficients $\lambda_{p , q} \in k$, where $p$ and $q$ range over $ \NN e_i $ and $ e_j \NN $ respectively. If $w $ is non-zero, then $p_0 \alpha q_0 \in \tip(I)$ for some paths $p_0 \in \NN e_i $ and $q_0 \in e_j \NN $. In particular, there is an element $ g \in \GG $ such that $ \tip ( g ) $ divides $ p_0 \alpha q_0 $. It follows that $\tip ( g )$ divides either $ p_0 $ or $ q_0 $ since $\tip ( g )$ avoids $\a $, a contradiction to the fact that $p_0 , q_0 \in \NN $. Therefore, we have $\lambda_{p, q} = 0$ for all $p$ and $q$, and the set $ ( \NN e_i ) \a ( e_j \NN ) $ is linearly independent modulo $ I $. We deduce that the ideal $ \la \a + I \ra $ is a projective $\La $-bimodule according to \cref{lem:appendixA.1}{}.

		Conversely, assume that the arrow $ \a \colon i \to j $ is pre-removable in $\La $ and the ideal $\la \a + I \ra $ is projective as a $\La$-bimodule. It suffices to prove that the subspace $ I \cap {}_k \la \BQnota \ra $ generates $I$, since $ e_j \La e_i = 0  $ according to \cref{lem:appendixA.1}{}.
		Specifically, we have to prove that every element $z \in I \cap {}_k \la \BQa \ra$ can be generated by elements of $ I \cap {}_k \la \BQnota \ra $ as $ I = ({ I \cap {}_k \la \BQa \ra }) \oplus ({ I \cap {}_k \la \BQnota \ra }) $, see \hyperlink{lem:ar.rem.2}{\cref{lem:ar.rem.2}{}}.
		
		The first step is to show that if we take two elements $x \in k Q  e_i$ and $y \in e_j k Q$ such that $x \in I$ or $y \in I$, then $ x \alpha y $ can be generated by elements in $ I \cap {}_k \la \BQnota \ra $. Assume for instance that $x \in I$. Then both $x_{\nott \! \a} \a y$ and $x_{\a} \a y  $ are generated by $ I \cap {}_k \la \BQnota \ra $. Indeed, we have that $x_{\nott \! \a} \in I \cap {}_k \la \BQnota \ra $ according to the above decomposition of $ I $ (\hyperlink{lem:ar.rem.2}{\cref{lem:ar.rem.2}{}.(iii)}), and each path occurring in $x_{\a} \a y$ contains a subpath avoiding $ \a $ with source $ j $ and target $ i $ which is in $I$ as $e_j \La e_i = 0 $.
		
		Now let $z = \sum_{m} x_m \alpha y_m $ be an arbitrary element of $ I  \cap {}_k \la \BQa \ra $,
		where $ x_m \in kQ e_i$ and $ y_m \in e_j kQ$. As the set $ \NN e_i + I$ is a $k$-basis of $\La e_i$, for every $ m $ we can write
		\[
			x_m =  \sum_p \lambda_m^p \cdot p  + x'_m
		\]
		for coefficients $\lambda_m^p \in k$, where $p$ ranges over $ \NN e_i $ and $ x'_m \in I e_i $. In particular, 
		\begin{equation*}
			\begin{split}
				z &
				= \sum_p p \alpha  (\sum_m \lambda_m^p \cdot y_m  ) + \sum_m x'_m \alpha y_m 
			\end{split}
		\end{equation*}
		where the first summand, denoted by $\tilde{ z }$, is in $I$, as the same holds for $z $ and the second summand. Similarly, let
		\[
		\sum_m \lambda_m^p \cdot y_m  = \sum_q \mu^q_p \cdot q + y'_p
		\]
		for every $ p \in \NN e_i $, for coefficients $\mu^q_p \in k$, where $q$ ranges over $ e_j \NN $ and $y'_p \in e_j I$. Then 
		we have that
		\[
		\tilde{ z } = \sum_{p, q} \mu^q_p \cdot p \alpha q + \sum_p p \alpha y'_p
		\]
		implying that $ \sum_{p, q} \mu^q_p \cdot p \alpha q \in I$. But the set $ ( \NN e_i ) \a ( e_j \NN ) + I $ is linearly independent according to \cref{lem:appendixA.1}{}, implying that $ \mu^q_p = 0 $ for all $p$ and $q$. We conclude that
		\[
			z = \sum_p p \alpha y'_p + \sum_m x'_m \alpha y_m
		\]
		is generated by $ I \cap {}_k \la \BQnota \ra $ as the same holds for every summand due to the fact that $x'_m, y'_p \in I$. See also the previous paragraph.
		
		\smallskip
		
		Before proceeding with the general case, note that we have proved that an arrow $ \a $  induces an arrow removal algebra of $ \La $ in the sense of \cite{arrowrem1}{} if and only if $ I $ possesses a generating subset avoiding $ \a $. In other words, the second condition of the arrow removal operation is redundant for single arrows.
		
		\smallskip
		
		\textbf{Step II:} The general case: $A = \{ \a_\k \colon i_\k \to j_\k  \, | \,  \k = 1 , 2 , \ldots , m \} $ for $ m \geq 1 $.	\\[0.05cm]
		We begin by assuming that $\La / \la A + I \ra $ is an arrow removal algebra of $\La $ in the sense of \cite{arrowrem1}{}.
		It follows from the previous paragraph that each arrow $\a_\k $ induces itself an arrow removal algebra of $ \La $. Therefore, the first part of the proof implies that $\a_\k $ is pre-removable in $\La $ and the ideal $\la \a_\k + I \ra $ is a projective $\La $-bimodule for all $ \k $. In order for $ \la A + I \ra = \sum_{ \k =1 }^m { \la \a_\k + I \ra }$ to be a projective $\La $-bimodule, it remains to prove that the sum of the ideals $\la \a_\k + I \ra $ is direct.
		
		Since $e_{j_{\k_1}} \La e_{i_{\k_2}} = 0 $ for all indices $\k_1, \k_2 \in \{ 1, 2, \ldots, m \}$, every element of $\la \a_\k + I \ra $ is equal modulo $I$ to a $k$-linear combination of paths passing trough $\a_\k $ exactly once and avoiding all arrows in $A \setminus \{ \a_\k \}$. Let $z_{ \a_\k } $ be such an element for every $ \k $ and assume that $ w = \sum_{ \k } z_{ \a _\k } \in I$. Since $ \a_\k $ is pre-removable in $\La $ for every $ \k $, it follows from \hyperlink{lem:ar.rem.2}{\cref{lem:ar.rem.2}{}.(iii)} that $z_{\a _\k} = w_{ \a_\k } \in I$ for every $ \k $. We deduce that the ideal $ \la A + I \ra = \oplus_{ \k =1 }^m { \la \a_\k + I \ra }$ is projective as a $ \La $-bimodule.
		
		In order to see that $A $ is pre-removable in $\La $, take $z \in I$ and define inductively the elements $z^\k = (z^{\k-1})_{\nott \! \a_\k} $ for $ \k = 1 , 2 , \ldots , m $, where $z^0 := z$. Then $z_{\nott \! A} = z^m $ since $z^m$ contains exactly the paths occurring in $ z $ that avoid all arrows in $ A $ with the same coefficients. Furthermore, applying the first part of the proof successively for the arrows $\a_1, \a_2 , \ldots, \a_m $ yields that $z_{ \nott \! A } \in I$.

		For the converse implication, assume that the set of arrows $A $ as above is pre-removable in $\La $ and the ideal $ \la A + I \ra $ is a projective $\La $-bimodule. It is straightforward to verify that the unique $\La$-bimodule homomorphism
		\[
		f \colon \bigoplus_{\k = 1}^m ( \La e_{i_\k} \otimes_k e_{j_\k} \La ) \epic \la A + I \ra
		\]
		such that $f( e_{i_\k} \otimes e_{j_\k} ) = \a_\k + I $ for every $ \k $ is a projective cover of $ \la A + I \ra $ as a $\La $-bimodule. It follows that $f$ is an isomorphism and $ \la A + I \ra  = \bigoplus_\k \la \a_\k + I \ra $. In particular, the ideal $\la \a_\k + I \ra $ is a projective $\La $-bimodule implying that $e_{j_{\k}} \La e_{i_{\k}} = 0$ for every $ \k $ according to \cref{lem:appendixA.1}{}. Now let $ \k_1 , \k_2 \in \{ 1 , 2, \ldots , m \} $ be two distinct indices, and assume that there is a non-zero path $ p $ in $\La $ with source $ j_{\k_1 } $ and target $ i_{\k_2} $. Then the projectivity of the ideal $\la \a_{\k_\nu} + I \ra $ as a $\La $-bimodule for $ \nu = 1 , 2 $ implies that $\a_{\k_1} p \a_{\k_2} $ is non-zero in $ \La $ as in the proof of \cref{lem:appendixA.1}{}, a contradiction since the intersection of ${ \la \a_{\k_1} + I \ra }$ and $ {  \la \a_{\k_2} + I \ra } $ is trivial.
		
		Our next target is to show that for every arrow $\a \in A$ there is a generating set for $I$ avoiding $\a $. We show first that every $\a  $ is pre-removable in $\La $ or, equivalently according to \hyperlink{lem:ar.rem.2}{\cref{lem:ar.rem.2}{}.(iii)}, that $z_\a \in I$ for every $z \in I$. Let $z_{\leq 1}$ and $z_{\geq 2} $ be the unique elements of $k Q$ such that each path occurring in $z_{\leq 1}$ passes through at most one arrow of $A$ and at most once, each path occurring in $z_{\geq 2}$ passes through some arrow in $A$ at least twice (for two possibly distinct arrows in $ A $), and $z = z_{\leq 1} + z_{\geq 2}$. By the previous paragraph, each path of $z_{\geq 2}$ is in $I$ implying that $(z_{\geq 2})_\a  $ and $ z_{\leq 1} = z - z_{ \geq 2}  $ are in $ I$. On the other hand, it holds that $(z_{\leq 1})_A = \sum_\k (z_{\leq 1})_{\a_\k} $, and it follows from the pre-removability of $ A $ that $(z_{\leq 1})_A  \in I$. It remains to notice that $ (z_{\leq 1})_{\a_\k}  \in I $ for every $\k $ due to the decomposition of $ \la A + I \ra $ as the direct sum of the ideals $\la \a_\k + I \ra $; in particular, we have $ (z_{\leq 1})_{\a}  \in I $ implying that $z_\a = (z_{\leq 1})_\a + (z_{\geq 2})_\a \in I$.
		Our claim follows now from the first part of the proof as we have shown that $\a$ is pre-removable in $\La $ and the ideal $\la \a + I \ra $ is a projective $\La$-bimodule.
		
		
		Now let $T$ be a generating set for $I$ and define inductively the sets $T^\k = ( T^{\k-1} )_{\nott \! \a_\k }$ for $ \k = 1 , 2 , \ldots , m $, where $T^0 := T$. A successive application of \cref{cor:ar.rem.5}{} for the arrows $\a_1, \a_2, \ldots , \a_m $ yields that $ T^m $ generates $I$. The set $ T^m $ clearly avoids all arrows in $A$ by construction.
	\end{proof}
	
	\begin{cor}
		\label{cor:appendixA.1}
		Let $\La = k Q  / I$ be a bound quiver algebra and $ A $ a set of arrows such that $\La / \la A + I \ra$ is an arrow removal algebra of $\La $ in the sense of \cite{arrowrem1}{}. Then $A$ is two-sided removable in $\La $.
	\end{cor}
	
	\begin{proof}
		Direct consequence of \cref{prop:ar.rem.1}{} and \cref{defn:ar.rem.1}{}.
	\end{proof}

	We end this appendix by showing that condition (ii) of the original arrow removal operation is redundant when $ A $ comprises a single arrow.
	
	\begin{cor}[\mbox{cf.\ \cite[Proposition~4.5]{arrowrem1}{}}]
		\label{cor:ar.rem.1}
		Let $\La  = k Q / I$ be a bound quiver algebra and $\a \in Q_1 $ an arrow. Then $\La / \la \a + I \ra $ is an arrow removal algebra of $\La $ in the original sense if and only if there is a generating set for $I$ avoiding $\a $.
	\end{cor}
	
	\begin{proof}
		If there is a generating set for $I$ avoiding $\a $, then it follows from the proof of \cref{prop:ar.rem.1}{} that the ideal $\la \a + I \ra $ is projective as a $\La $-bimodule. \cref{lem:appendixA.1}{} implies now that $e_{t( \a )} \La e_{ s( \a )} = 0 $. 
	\end{proof}

\end{document}